\documentclass[10pt]{article}

\usepackage{ifpdf}
\usepackage{graphicx}
\newcommand{\Title}{Mixing times for exclusion
  processes on hypergraphs}
\usepackage[british]{babel}
\newcommand{\Date}{\today}

\usepackage{calc, xspace}
\usepackage{amsmath, amsthm, amsfonts, amssymb, euscript}
\usepackage{enumerate}

\usepackage[longnamesfirst]{natbib}
\setcitestyle{round}
\usepackage[textsize=tiny]{todonotes}

\ifpdf

\usepackage{a4wide}
\parskip \medskipamount

\theoremstyle{plain}
\newtheorem{thm}{Theorem}[section] \newtheorem{prop}[thm]{Proposition}
\newtheorem{lemma}[thm]{Lemma}
\newtheorem{cor}[thm]{Corollary}

\theoremstyle{definition}
 \newtheorem{defn}[thm]{Definition}
\newtheorem{assumption}[thm]{Assumption}
\newtheorem{inbox}[thm]{Box}
\theoremstyle{remark}
\newtheorem{rmk}[thm]{Remark}

\newcommand{\ink}{\mbox{ink}}

\def\cS{\mathcal{S}}

\def\cP{\mathcal{P}}

\def\cL{\mathcal{L}}

\def\cC{\mathcal{C}}

\def\tv{_{\mathrm{TV}}}

\def\P{\mathbb{P}}
\def\E{\mathbb{E}}
\def\eps{\varepsilon}

\newcommand{\indic}[1]{\mathbf{1}_{\left\{#1\right\}}}
\newcommand{\bra}[1]{\left[#1\right]}

\newcommand{\rhofn}{\beta}
\usepackage[arxivVersion]{optional}

\numberwithin{equation}{section}

\usepackage{framed}			
\usepackage[small]{caption}
\usepackage[font=small]{subcaption}

 \usepackage[
 \ifpdf
   pdftex,
   pdfstartview=FitH,
 \fi
 ]{hyperref}

\usepackage{geometry}
\usepackage{fancyhdr}
\begin{document}

 \title{\Title} \author{
   Stephen B. Connor\footnote{Work supported in part by EPSRC Research
     Grant EP/J009180.}\,\,\footnote{Both authors supported in part by LMS Research in Pairs (Scheme 4) Grant 41215.}  \quad and\quad Richard J. Pymar\footnote{Work supported in part by Leverhulme Research Grant RPG-2012-608 held by
Nadia Sidorova and EPSRC Grant EP/M027694/1 held by Codina Cotar.}\,\,\footnotemark[2]}

 \date{\Date}
 \maketitle

\begin{abstract}
We introduce a natural extension of the exclusion process to
   hypergraphs and prove an upper bound for its mixing time. In
   particular we show the existence of a constant $C$ such that for
   any connected, regular hypergraph $G$ within some natural class, the $\varepsilon$-mixing time of the
   exclusion process on $G$ with any feasible number of particles can
   be upper-bounded by $CT_{\text{EX}(2,G)}\log(|V|/\varepsilon)$,
   where $|V|$ is the number of vertices in $G$ and
   $T_{\text{EX}(2,G)}$ is the 1/4-mixing time of the corresponding
   exclusion process with just two particles. Moreover we show this is optimal in the sense that there exist hypergraphs in the same class for which $T_{\mathrm{EX}(2,G)}$ and the mixing time of just one particle are not comparable. The proofs involve an
   adaptation of the \emph{chameleon process}, a technical tool
   invented by \cite{Morris2006a} and developed by
   \cite{Oliveira2013a} for studying the exclusion process on a graph.
\end{abstract}

 \begin{quotation}
 \end{quotation}

 \begin{quotation}\noindent
 \end{quotation}


\section{Introduction}

Let $G=(V,E)$ be a finite connected graph with vertex set $V$ and edge
set $E$. Fix $k\in\{1,\ldots,|V|\}$ and consider $k$ indistinguishable
particles moving on $V$ using the following rules:
\begin{enumerate}
\item each vertex is occupied by at most one particle,
\item each edge $e\in E$ rings at the times of a Poisson process of
  rate 1, independently of all other edges,
\item when an edge $e=\{u,v\}$ rings, the occupancy states of vertices
  $u$ and $v$ are switched.
\end{enumerate}

For each $v\in V$ and $t\ge0$, let $\eta_t(v)=1$ if $v$ is occupied at
time $t$, and $\eta_t(v)=0$ if $v$ is vacant at time $t$. The process
$(\eta_t)_{t\ge0}$ is called the \emph{$k$-particle exclusion process
  on $G$}: see Figure \ref{F:exclusion}. In this paper we are
interested in a natural extension of the exclusion process to
\emph{hypergraphs}.

\begin{figure}[!ht]
  \centering
  \includegraphics[scale=1.1]{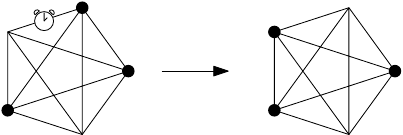}
  \caption{Example transition of 3-particle exclusion process on
    $K_5$. When the edge indicated rings, the single particle
    currently on that edge moves to the vertex at the other end of the
    edge.}
  \label{F:exclusion}
\end{figure}

Let $G=(V,E)$ be a finite connected hypergraph, where $E\subseteq
\mathcal{P}(V)$, the power set of $V$.  For each $e\in E$, denote by
$\cS_e$ the symmetric group on the elements in $e$, and let $f_e:\mathcal{S}_e \to
[0,1]$ be a probability measure on $\cS_e$. We write $f$ to denote
$\{f_e:\,e\in E\}$, the set of these measures. Consider $k$
indistinguishable particles moving on $V$ using rules 1. and 2. above
and in addition:
\begin{enumerate}
\item [3$'$.] when an edge $e$ rings, a permutation $\sigma\in
  \mathcal{S}_e$ is chosen according to $f_e$ and every particle on a
  vertex in $e$ moves simultaneously according to $\sigma$, i.e. a
  particle at vertex $v$ moves to vertex $\sigma(v)$. (Note that as
  $\sigma$ is a permutation, rule 1.\! is preserved.)
\end{enumerate}
Setting $\eta_t(v)=1$ if $v$ is occupied at time $t$ and 0 otherwise,
we obtain a process $(\eta_t)_{t\ge0}$ referred to as the
\emph{$k$-particle exclusion process on $(f,G)$}, or simply EX($k,
f,G$): see Figure \ref{F:exclusion2}. (Note that if each edge $e\in E$
contains exactly two vertices, and $f_e$ puts all of its mass on the
transposition belonging to $\mathcal S_e$, then EX($k, f,G$) is just
the $k$-particle exclusion process on the graph $G$, as above.)

\begin{figure}[!ht]
  \centering
  \includegraphics[scale=1.1]{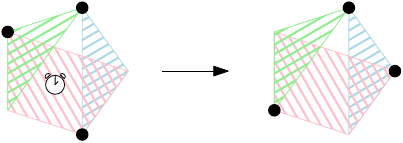}
  \caption{Example transition of 3-particle exclusion process on a
    hypergraph with 5 vertices and 3 edges (indicated by the different
    shaded regions, i.e. here there are two edges of size 3 and one of size 4). When the edge containing four vertices rings, the
    two particles currently belonging to that edge are permuted.}
  \label{F:exclusion2}
\end{figure}

Our main aim in this paper is to study the total-variation mixing time
of EX($k, f,G$), and to establish an upper bound in terms of the
mixing time of EX($2,f,G$). Recall that for a continuous-time Markov process $X$ on a finite set
$\Omega$ with transition probabilities $\{q_t(x,y)\}$ and equilibrium
distribution $\pi$, the total variation $\varepsilon$-mixing time is
defined as
\begin{align}\label{eq:tmixdef}
  T_X(\varepsilon):=\inf\left\{t\ge0:\,\max_{x\in\Omega}\|q_t(x,\cdot)-\pi\|_{\mathrm{TV}}\le
    \varepsilon\right\},
\end{align}
where $\|\cdot\|_{\mathrm{TV}}$ is the total-variation norm. 

In several parts of the proof it will be useful to consider the associated process where the $k$
particles are distinguishable.  Suppose the particles are labelled
$1,\ldots,k$ and set $\hat\eta_t(v)$ to be the label of the particle
at vertex $v$ at time $t$. If there is no particle at $v$ at time $t$,
set $\hat\eta_t(v)=0$. The process $(\hat\eta_t)_{t\ge0}$ is the
\emph{$k$-particle interchange process on $(f,G)$}, or simply
IP$(k,f,G)$. Note that the exclusion process may be recovered from the interchange process simply by `forgetting' the labels of the particles.

Throughout we will make the following assumptions about the hypergraph $G$ and the set of measures $f$ (with notation appearing below being formally defined in Section \ref{S:processconstr}).

\begin{assumption}\label{assumpf}\
  \begin{enumerate}
  \item\label{assumpf-regular} The hypergraph $G$ is regular (every vertex has the same degree).
  \item\label{assumpf-conjugacy} For every $e$, $f_e$ is constant on the conjugacy classes of
    $\cS_e$ (i.e. in group-theoretic terms, $f_e$ is a \emph{class
      function}). That is, if $\sigma_1$ and $\sigma_2$ are elements
    from $\mathcal{S}_e$ with the same cycle structure, then
    $f_e(\sigma_1)=f_e(\sigma_2)$.
  \item\label{assumpf-fixed} For every $e$ and each $v\in e$, $\sum_{\sigma\in
      \cS_e}f_e(\sigma)\indic{\sigma(v)= v}\le 1/5$. In other words, the probability (under $f_e$) of a
    vertex $v\in e$ being a fixed point of $\sigma$ is at most $1/5$.
   \item\label{assumpf-irreducible} The interchange process IP$(k,f,G)$ is \emph{irreducible} for any number of particles $k\in\{1,\dots,|V|-1\}$.
  \end{enumerate}
\end{assumption}
These assumptions are more than enough to imply that the exclusion process is reversible and ergodic, with uniform stationary distribution. Although we state it as an assumption on $f$, the fourth assumption also implies that the underlying hypergraph $G$ is connected. Our main theorem is the following:

\begin{thm}\label{T:mainexcl}
  There exists a universal constant $C>0$ such that for every
  $(f,G)$ satisfying Assumption \ref{assumpf} and every
  $k\in\{1,\ldots,|V|-1\}$ and $\eps>0$,
  \[
  T_{\mathrm{EX}(k,f,G)}(\eps)\le C\log
  (|V|/\varepsilon)T_{\mathrm{EX}(2,f,G)}(1/4).
  \]
\end{thm}

\begin{rmk}
In all further statements we implicitly assume that Assumption~\ref{assumpf} holds.
\end{rmk}
\begin{rmk}\label{trans3cycles}
The exclusion process on a hypergraph $G$ with the edge set $E$ consisting only of edges of size 2 or 3 exhibits the \emph{negative correlation} property (which we shall discuss further in the sequel). As a result, for this subset of hypergraphs we can actually extend the main theorem, replacing the $T_{\mathrm{EX}(2,f,G)}(1/4)$ appearing in the right-hand side by $T_{\mathrm{EX}(1,f,G)}(1/4)$ (note that we later refer to $\mathrm{EX}(1,f,G)$ as $\mathrm{RW}(1,f,G)$, in recognition that the exclusion process with just one particle is equivalent to a single particle performing a random walk).
\end{rmk}

\begin{rmk}
A simple example suffices to show that Theorem~\ref{T:mainexcl} is optimal in the sense that we cannot replace $T_{\mathrm{EX}(2,f,G)}(1/4)$ on the right-hand side with $T_{\mathrm{EX}(1,f,G)}(1/4)$ (even under our standing assumptions). Let $G=(V,E)$ with $V=\{1,1',2,2',\ldots,m,m'\}$ and \mbox{$E=\{\{i,i',j,j'\}:\,i\neq j\}$}. Suppose that $f_{\{1,1',2,2'\}}(\sigma)=1/6$ if $\sigma$ is  a cycle of size 4 (and otherwise $f_{\{1,1',2,2'\}}(\sigma)=0$). For $\{i,i',j,j'\}\neq\{1,1',2,2'\}$, we set $f_{\{i,i',j,j'\}}(\sigma)=1/3$ if $\sigma$ is a composition of two disjoint transpositions (and otherwise $f_{\{i,i',j,j'\}}(\sigma)=0$). It can be readily checked that this hypergraph satisfies Assumption~\ref{assumpf}. Note that there are $\binom{m}{2}$ edges, each ringing at rate 1. It is easy to see that a random walker mixes in time of order $1/m$ since each vertex is in order $m$ edges. Now consider a 2-particle exclusion process started from $\{3,3'\}$. Notice that up until the first time that both particles occupy vertices belonging to the edge $\{1,1',2,2'\}$, if vertex $i$ is occupied then vertex $i'$ is also occupied. So the process cannot mix until the edge $\{1,1',2,2'\}$ is visited by the particles. Regardless of where the particles are before this time, there are always two edges that can ring which would bring the particles to the set $\{1,1',2,2'\}$, and so this happens at rate $2$; we conclude that it takes a time of order 1 for the 2-particle exclusion process to mix.
\end{rmk}

\subsection{Motivation and connections with the
  literature}\label{S:motiv}
Our results contribute to the general question of when properties of a
multi-particle system can be deduced from properties of a system with
only a few particles.  Arguably the most significant recent result in
this area has come from \cite{Caputo2010} who showed that the spectral
gap of the interchange process on a graph is equal to the spectral gap
of a random walker on the same graph, proving a conjecture of Aldous
that had been open for 20 years. Proving results in this area is
particularly important in applications since the large reduction in
the size of the state space often makes it much easier to
compute or estimate statistics.

While interacting particle system models (e.g. exclusion process,
interchange process, voter model, contact process, zero range process)
on graphs have received considerable attention, there has so far been
little study of such processes on hypergraphs. Studying these
processes on hypergraphs is very natural though, as hypergraphs allow
simultaneous interactions of multiple particles, rather than only
pair-wise interactions. One model for which its analogue on
hypergraphs has been recently studied is the voter model
(\cite{Chung2014, Istrate2014}), for which various properties are
considered, including the mixing time.

Any interchange process (with $k=|V|$) on a graph can be viewed as a
card shuffle by transpositions, and there is now an extensive literature concerning mixing times of such shuffles. Notable examples include the
top-to-random transposition shuffle (star graph;  \cite{Flatto1985}), random-to-random
transposition shuffle (complete graph; \cite{Diaconis1981}) and nearest-neighbour transposition shuffle
(the cycle; \cite{Lacoin2016}). Of course, transposition shuffles are just one class of shuffle, and
there is significant interest in mixing times of more general shuffles
in which multiple cards are moved simultaneously. A large class of
time-homogeneous shuffles can be represented as interchange processes
on hypergraphs; recent examples can be found in \citep{Sengul2014} and \citep{Hough2015}.

Achieving tight bounds on the mixing time of an interacting particle system typically involves
finding an argument tailored specifically to the model in question. If
we care less about the specific constant multiple (at which mixing
occurs) and instead focus on the order, a result of
\cite{Oliveira2013a} can prove particularly useful as a general way of
bounding mixing times of exclusion processes:

\begin{thm}[\cite{Oliveira2013a}]\label{t:Oliveira_IP}
  There exists a constant $C>0$ such that for every connected weighted
  graph $G$ and every $k\in\{1,\ldots,|V|-1\}$ and $\eps\in(0,1/2)$,
  \[
  T_{\mathrm{EX}(k,G)}(\eps)\le C\log
  (|V|/\varepsilon)T_{\mathrm{RW}(G)}(1/4)\,,
  \]
  where $T_{\mathrm{RW}(G)}(1/4)$ is the mixing time of the random
  walk on $G$. 
\end{thm}

Our main result
extends Theorem~\ref{t:Oliveira_IP} to a class of hypergraphs. Furthermore, our
results hold for a large class of measures acting on the symmetric
group $\mathcal{S}_{|V|}$ which goes beyond the standard framework
studied by previous authors, in which a conjugacy class is fixed and
then sampled from uniformly (e.g. \cite{Lulov2002, Sengul2014}). Indeed, our measures $f_e$ can vary
dramatically between edges $e\in E$, and furthermore we do not
require each $f_e$ to be supported on a fixed conjugacy class.

\subsection{Heuristics and structure of the proof}
   
The proof of Theorem \ref{T:mainexcl} depends on the size of the vertex set $V$. If $|V|$ is sufficiently small, the proof is fairly simple and we state the result as the following lemma:
\begin{lemma}\label{L:smallV}
  There exists a constant $C>0$ such that for every hypergraph $G=(V,E)$ with $|V|<36$, every $f$ and every $k\in\{1,\ldots,|V|/2\}$ and $\eps>0$,
  \[
  T_{\mathrm{EX}(k,f,G)}(\eps)\le C\log
  (1/\varepsilon)T_{\mathrm{EX}(2,f,G)}(1/4).
  \]
\end{lemma}

On the other hand, the argument for $|V|\ge 36$ is much more intricate and is split into two parts, the first being the following lemma which is of independent interest (and is stronger than needed for our main theorem, as it relates to the interchange process):

\begin{lemma}\label{L:kto4}
  There exists a constant $C>0$ such that for every hypergraph $G=(V,E)$
  with $|V|\ge 36$, every $f$ and every $k\in\{1,\ldots,|V|/2\}$ and $\eps>0$,
  \[
  T_{\mathrm{IP}(k,f,G)}(\eps)\le C\log
  (|V|/\varepsilon)T_{\mathrm{EX}(4,f,G)}(1/4).
  \]
\end{lemma}

\cite{Oliveira2013a} proves his main result (bounding the mixing time of
the $k$-particle exclusion process by the mixing time of a random
walker) by first relating the mixing time of a $k$-particle
interchange process to that of a 2-particle interchange process.
Roughly speaking, this is possible due to the fact that any time an
edge of the graph under consideration rings, at most two particles
move under interchange, and so it is pairwise interactions that
determine the mixing rate. This contrasts with the exclusion process
on hypergraphs considered here, in which \emph{many} particles can
move at the same time. Nevertheless, a suitable adaptation of the
techniques appearing in \citep{Oliveira2013a} provides the proof of Lemma \ref{L:kto4}.

\begin{rmk}
	Lemma~\ref{L:kto4} only holds when $|V|$ is sufficiently large and $k\le |V|/2$. We cannot hope to remove these conditions and replace EX$(4,f,G)$ with EX$(2,f,G)$ in this statement, even for hypergraphs satisfying Assumption~\ref{assumpf}, as the following example illustrates. Let $G=(V,E)$ with $V=\{1,2,3\}$ and $E=\{V\}$, i.e. there is just a single edge which contains all three vertices in the hypergraph. Suppose that $f$ gives probability $1-\delta$ to the conjugacy class of 3-cycles, and probability $\delta$ to the class of transpositions. For $\delta$ sufficiently small this satisfies Assumption~\ref{assumpf}. The 2-particle interchange process cannot mix until a transposition is chosen (as half of the states cannot be reached before this time), whereas this event is not necessary for the 2-particle exclusion process to mix, and hence it is straightforward to see that as $\delta\to0$ we have $T_{\mathrm{IP}(2,f,G)}(1/4)/T_{\mathrm{EX}(2,f,G)}(1/4)\to\infty$.
\end{rmk}

The second part of the proof for $|V|\ge 36$ requires showing that
$T_{\mathrm{EX}(4,f,G)}$ and $T_{\mathrm{EX}(2,f,G)}$ are of the same
order: 

\begin{lemma}\label{L:4to2}
 There exists a constant $\lambda>0$ such that for any hypergraph $G$ with $|V|\ge 36$, and $f$,
  \[
  T_{\mathrm{EX}(4,f,G)}(1/4)\le \lambda T_{\mathrm{EX}(2,f,G)}(1/4).
  \]
\end{lemma}

We now demonstrate that Theorem \ref{T:mainexcl} follows simply from Lemmas \ref{L:smallV}, \ref{L:kto4} and \ref{L:4to2}.

\begin{proof}[Proof of Theorem \ref{T:mainexcl}]
 The contraction principle (see \cite{AldousFill2014}) gives
  \[
   T_{\mathrm{EX}(k,f,G)}(\eps)\le T_{\mathrm{IP}(k,f,G)}(\eps)\,,
  \]and so provided $k\le |V|/2$, we have the result for $|V|\ge 36$ by Lemmas \ref{L:kto4} and \ref{L:4to2} and for $|V|<36$ by Lemma \ref{L:smallV}. However, note
  that switching the roles of occupied and unoccupied vertices in EX$(k,f,G)$ yields the process EX\mbox{$(|V|-k,f,G)$}. It follows that
  $$T_{\mathrm{EX}(k,f,G)}(\eps)=T_{\mathrm{EX}(|V|-k,f,G)}(\eps)\,,$$
and so the proof of Theorem \ref{T:mainexcl} is complete.
\end{proof}

We finish this section with a brief overview of the rest of the
paper. In Section \ref{S:processes} we define formally the processes
considered in this paper and present some preliminary results. In addition, we
demonstrate that the negative correlation property, which is
fundamental to the result in \citep{Oliveira2013a}, fails to hold for
the hypergraph setting. In Section~\ref{S:existence} we prove Lemma~\ref{L:kto4} subject to the existence of a process with certain key properties that relate it to an interchange process: see Lemma~\ref{L:existence} for the precise statement. This process is constructed in Section~\ref{S:cham} and we show it has the desired properties in Section~\ref{S:chamhasprops}. Proving Lemma~\ref{L:existence} is the most challenging (and technical) part of this paper.

In Section~\ref{S:4to1} we prove Lemma~\ref{L:4to2} by first
characterizing every hypergraph as one of two types depending on how long
it takes any two of four independent particles to meet.

 We use some of the ideas developed in Section~\ref{S:4to1} to prove Lemma~\ref{L:smallV} in Section~\ref{S:smallV}.
\opt{noarxivVersion}
{
A few of the more technical proofs required are included in two
appendices which can be found in the extended arXiv-version of this
paper \citep{Connor-Pymar-2016}.
}
\opt{arxivVersion}
{
\hspace{-3.5mm}A few of the more technical proofs required are included in two
appendices.
}

\section{Preliminaries}\label{S:processes}

\subsection{Random walks, exclusion and
  interchange processes}\label{S:processconstr}
We formally define the main processes studied in this paper,
RW($f,G$), RW($k,f,G$), EX($k,f,G$) and IP($k,f,G$), by explicitly
stating their generators. In the next section we shall present a
\emph{graphical construction} of these processes, similar to that of
\cite{Liggett1999} for the standard interchange and exclusion
processes. This graphical construction will allow us to simultaneously
define the processes on the same probability space, and thus directly
compare them.

Recall $\cS_e$ as the group of permutations of elements in $e$. Our
processes of interest evolve by the action of permutations from these
groups. However, it will often be convenient to consider permutations as
acting on $V$ and we can easily do this by extending a permutation
$\sigma_e\in\cS_e$ to a permutation in $\cS_V$ by setting
$\sigma_e(v)=v$ for all $v\notin e$. We can also consider such
permutations as acting on a subset of $V$ or on vectors with elements
being distinct members of $V$. To do this we can define, for a set
$A\subseteq V$, $\sigma_e(A):=\{\sigma_e(a):\,a\in A\}$, and for a
vector ${\bf x}$ of $k$ distinct elements of $V$ we define
$\sigma_e({\bf x}):=(\sigma_e({\bf x}(i)))_{i=1}^k$.

{\bf Set notation:} For $k\in\mathbb{N}$ we
define
\[
\binom{V}{k}:=\{A\subseteq V:\,|A|=k\},
\]and for a set $A\subseteq V$ we write
\[
(A)_k:=\{{\bf a}=({\bf a}(1),\ldots,{\bf a}(k))\in A^k:\,{\bf
  a}(i)\neq{\bf a}(j)\,\forall i\neq j\}.
\]

{\bf Generators:} We now explicitly state the generators of the processes. For a
hypergraph $G$ and a suitable set of functions $f$, the simple random
walk on $G$, RW($f,G$), is the continuous-time Markov chain with state
space $V$ and generator
\[
U^{\text{RW}}h(u)=\sum_{e\in
  E}\sum_{\sigma\in\mathcal{S}_e}f_e(\sigma)(h(\sigma(u))-h(u))
\]
for all $u\in V$ and $h:V\to\mathbb R$.

We denote by RW$(k,f,G)$ the product of $k$ independent random walkers
on $G$. This process is the continuous-time Markov chain with state
space $V^k$ and generator
\[
U^{\text{RW}(k)}h({\bf u})=\sum_{e\in
  E}\sum_{i=1}^k\sum_{\sigma\in\mathcal{S}_e}f_e(\sigma)(h({\bf u}^i_{\sigma({\bf
      u}(i))})-h({\bf u})),
\]
for all ${\bf u}\in V^k$ and $h:V^k\to \mathbb R$, where \[ {\bf
  u}^i_v(j)=
\begin{cases}
  {\bf u}(j)&j\neq i,\\
  v&j=i.
\end{cases}
\]
The $k$-particle exclusion process EX($k,f,G$), is the continuous-time
Markov chain with state space $\binom{V}{k}$ and generator
\[
U^{\text{EX}}h(A)=\sum_{e\in
  E}\sum_{\sigma\in\mathcal{S}_e}f_e(\sigma)(h(\sigma(A))-h(A)),\]for
all $A\in\binom{V}{k}$ and $h:\binom{V}{k}\to\mathbb{R}$.

The $k$-particle interchange process IP($k,f,G$), is the
continuous-time Markov chain with state space $(V)_k$ and generator
\[
U^{\text{IP}}h({\bf x})=\sum_{e\in
  E}\sum_{\sigma\in\mathcal{S}_e}f_e(\sigma)(h(\sigma({\bf x}))-h({\bf x})) ,\] for all ${\bf x}\in
(V)_k$ and $h:(V)_k\to\mathbb{R}$.

\subsection{Graphical construction}\label{S:graphical}
We first construct an independent sequence of $E$-valued random
variables $\{e_n\}_{n\in\mathbb{N}}$ such that each $e_n$ is
identically distributed with $\P\bra{e_n=e}=1/|E|.$ Given the sequence
$\{e_n\}_{n\in\mathbb{N}}$, let $\{\sigma_n\}_{n\in\mathbb{N}}$ be a
sequence of permutations with $\sigma_n\in \cS_{e_n}$ independently
chosen and satisfying for each $n\in\mathbb{N}$, and $e\in E$,
$\P\bra{\sigma_n=\sigma|\,e_n=e}=f_{e}(\sigma)$. Now that we have the
sequence of edges that ring and the permutations to apply, it remains
to determine the update times of the processes.

Let $\Lambda$ be a Poisson process of rate $|E|$ and for $0<s<t$
denote by $\Lambda[s,t]$ the number of points of $\Lambda$ in
$[s,t]$. For every $0<s<t$, we define a random permutation
$I_{[s,t]}:V\to V$ associated with the time interval $[s,t]$ to be the
composition of the permutations performed during this time; that is,
\[
I_{[s,t]}=\sigma_{e_{\Lambda[0,t]}}\circ\sigma_{e_{\Lambda[0,t]-1}}\circ\cdots\circ\sigma_{e_{\Lambda[0,s)+1}}
.\] We set $I_t:=I_{[0,t]}$ for each $t>0$, and $I_{(t,t]}$ to be the
identity. Note (cf Proposition 3.2 of \cite{Oliveira2013a}) that 
\begin{equation}\label{eq:law_I_inverse}
\cL[I_{(s,t]}] = \cL[I^{-1}_{(s,t]}]\,,
\end{equation}
where we write $\cL$ for the law of a process.

We can lift the functions $I_{[s,t]}$ to functions on
$\binom{V}{k}$ and $(V)_k$ in the following way: for
$A\in\binom{V}{k}$,
\[
I_{[s,t]}(A)=\{I_{[s,t]}(a):\,a\in A\},
\]
and for ${\bf x}\in(V)_k$,
\[
I_{[s,t]}({\bf x})=(I_{[s,t]}({\bf x}(1)),\ldots,I_{[s,t]}({\bf
	x}(k))).
\]

The following proposition is fundamental: its proof follows by
inspection.
\begin{prop}\label{prop:graph}
  Fix $s>0$. Then:
  \begin{enumerate}
  \item For each $u\in V$, the process $\{I_{[s,s+t]}(u)\}_{t\ge0}$ is
    a realisation of $\mathrm{RW}(f,G)$ initialised at $u$ at time
    $s$. We shall often write this process simply as
    $(u_t^\mathrm{RW})_{t\ge s}$.
  \item For each $A\in\binom{V}{k}$, the process
    $\{I_{[s,s+t]}(A)\}_{t\ge0}$ is a realisation of
    $\mathrm{EX}(k,f,G)$ initialised at $A$ at time $s$. We shall
    often write this process simply as $(A_t^\mathrm{EX})_{t\ge s}$.
  \item For each ${\bf x}\in(V)_k$, the process $\{I_{[s,s+t]}({\bf
      x})\}_{t\ge0}$ is a realisation of $\mathrm{IP}(k,f,G)$
    initialised at ${\bf x}$ at time $s$. We shall often write this
    process simply as $({\bf x}_t^\mathrm{IP})_{t\ge s}$.
  \end{enumerate}
\end{prop}

\subsection{Total variation and mixing times}
There are several equivalent definitions of total variation that we
shall make use of in this paper. Suppose $\mu$ and $\nu$ are two
probability measures on the same finite set $\Omega$. Then the
\emph{total variation distance} between these measures is defined as 
\begin{align}\label{eq:tvdef}
\|\mu-\nu\|\tv&:=\max_{A\subset \Omega}(\mu(A)-\nu(A))\\&=\sup_{f:\Omega\to[0,1]}\int fd\mu-\int fd\nu.\label{eq:tvint}
\end{align}
We shall also make extensive use of the following equivalent definition, which relates the total variation distance to couplings of $\mu$ and $\nu$:
\begin{align}\label{eq:tvcoupling}
\|\mu-\nu\|\tv=\inf_{(X,Y)}\P\bra{X\neq Y},
\end{align}
where the infimum is over all couplings $(X,Y)$ of random variables with $X\sim\mu$ and $Y\sim\nu$.
We recall a simple result bounding the total variation of product chains (see e.g. pg 59 of \citet{Levin2008}): for $n\in\mathbb{N}$ and $1\le i\le n$, let $\mu_i$ and $\nu_i$ be measures on a finite space $\Omega_i$ and define measures $\mu$ and $\nu$ on $\prod_{i=1}^n\Omega_i$ by $\mu:=\prod_{i=1}^n\mu_i$ and $\nu:=\prod_{i-1}^n\nu_i$. Then 
\begin{align}\label{eq:prodtv}
\|\mu-\nu\|\tv\le\sum_{i=1}^n\|\mu_i-\nu_i\|\tv.
\end{align}
Recall equation \eqref{eq:tmixdef} as the definition of the mixing time of a continuous-time Markov process. We will require several general mixing-time bounds throughout this work, which we present here.

\begin{prop}[\citet{Levin2008}]\label{P:1to1eps} Let $X$ be a Markov process on a finite state space. Then for every $\eps_1,\eps_2\in(0,1/2)$,
  \[
  T_X(\eps_2)\le\left\lceil\frac{\log\eps_2}{\log(2\eps_1)}\right\rceil T_X(\eps_1).
  \]
\end{prop}
\begin{prop}\label{P:2and4to1indep}
For any $m,n\in\mathbb{N}$,
\[
T_{\mathrm{RW}(2^m,f,G)}(2^{-n})\le (n+m)T_{\mathrm{RW}(f,G)}(1/4).
\]

\end{prop}
\begin{proof}
This follows by combining Proposition \ref{P:1to1eps} with \eqref{eq:prodtv}.
\end{proof}
\begin{prop}[\citet{AldousFill2014}]\label{P:AldousFill}
Let $X$ be a Markov process on a finite state space $\Omega$ with symmetric transition rates. Then the equilibrium distribution is uniform over $\Omega$ and for all $0<\eps<1/2$ and $t\ge 2 T_X(\eps)$,
\[
\P\bra{X_t=\omega_2|\,X_0=\omega_1}\ge \frac{(1-2\eps)^2}{|\Omega|},
\]
for all $\omega_1,\omega_2\in \Omega$.
\end{prop}

\subsection{Failure of negative correlation}\label{s:neg_cor}
We conclude this preliminary section with a quick example to
demonstrate that the exclusion process on a hypergraph does not
enjoy the negative correlation property satisfied by the exclusion
process on a graph. We first recall the version of the negative
correlation property of the exclusion process on a graph to which we
refer, and whose proof may be found in \citep{Liggett1985}. Let $B\subset V$ and let $(A_t^\mathrm{EX})_{t\ge0}$ be a 2-particle exclusion process on a graph $G=(V,E)$ with $A=\{u,v\}$. Suppose $(u_t^{\mathrm{RW}})_{t\ge0}$ and
$(v_t^{\mathrm{RW}})_{t\ge0}$ are two independent realisations of
RW$(1,f,G)$, started from $u$ and $v$ respectively. Then for every $t\ge0$,
\[
\P\bra{A_t^\mathrm{EX}\subseteq B}\le\P\bra{u_t^\mathrm{RW}\in B}\P\bra{v_t^\mathrm{RW}\in B}.
\]

Now suppose $G=(V,E)$ is the hypergraph with $V=\{1,2,3,4\}$ and $E=\{V\}$ (i.e. there
is only one edge), and that $f$ is concentrated uniformly on the six
possible 4-cycles. Let $(A_t^{\mathrm{EX}})_{t\ge0}$ be a realisation of
EX$(2,f,G)$, with $A= \{u,v\}=\{1,2\}$, and let $B=\{3,4\}$. We claim that there
exist values of $t$ such that
\begin{equation}\label{e:neg_cor}
 \P\bra{u_t^\mathrm{RW}\in
  B,v_t^\mathrm{RW}\in B}<\P\bra{A_t^\mathrm{EX}= B}.
\end{equation}
Indeed,
since the event $\{u_t^\mathrm{RW}\in B\}$ is less likely than seeing at least one incident
in a unit-rate Poisson process by time $t$, we have
\begin{align*}
  \P\bra{u_t^\mathrm{RW}\in B,v_t^\mathrm{RW}\in B}=
  \P\bra{u_t^\mathrm{RW}\in B} \P\bra{v_t^\mathrm{RW}\in
    B}\le
  (1-e^{-t})^2.
\end{align*}
On the other hand, the event $\{A_t^\mathrm{EX}= B\}$ is at
least as likely as the edge ringing exactly once by time $t$, with the
chosen permutation satisfying $\sigma(\{1,2\}) = \{3,4\}$. That is,
\begin{align*}
  \P\bra{A_t^\mathrm{EX}= B} &\ge
  \frac13te^{-t}.
\end{align*}
Inequality \eqref{e:neg_cor} is therefore satisfied for any $t<0.33$.

\bigskip

\section{From $k$-particle interchange to 4-particle exclusion}\label{S:existence}

In this section we shall prove Lemma \ref{L:kto4}. Given a hypergraph with vertex set $V$ and a $(k-1)$-tuple $\mathbf{z}\in(V)_{k-1}$, let \[\mathbf{O}(\mathbf{z}):=\{\mathbf{z}(1),\ldots,\mathbf{z}(k-1)\}\] be the (unordered) set of coordinates of $\mathbf{z}$ and define a space
\[
\Omega_k(V):=\{(\mathbf{z},R,P,W):\,\mathbf{z}\in (V)_{k-1}, \text{ and sets }\mathbf{O}(\mathbf{z}),\,R,\,P,\,W\text{ partition }V\}.
\]

As we shall see, most of the work required to prove Lemma \ref{L:kto4} is to show the existence of a certain Markov process having some key properties, which we outline in the Lemma~\ref{L:existence} below. As we shall see in the sequel, this Markov process is very similar to the chameleon process used in \cite{Oliveira2013a} and it provides a way of tracking how mixed the $k$th particle is in a $k$-particle interchange process. The $k$th particle is replaced by three sets of coloured particles, $R_t$ (red particles), $P_t$ (pink particles) and $W_t$ (white particles), with the colours informing the conditional distribution of the $k$th particle in the interchange process. A process $(\mathrm{ink}^{\bf
		x}_t(b))_{t\ge0}$ is defined for each vertex $b\in V$, which records the amount of \emph{redness} at vertex $b$ (equal to 1 if a red particle is at vertex $b$ and 1/2 if a pink particle is at vertex $b$). We shall also define an event $\mathrm{Fill}^{\bf
  	x}$ as the event that all vertices unoccupied by the first $k-1$ particles in the interchange process are eventually each occupied by a red particle in the chameleon process. 

\begin{lemma}\label{L:existence}
There exist constants $c_1,c_2$ and $\kappa_1$ such that for every regular hypergraph $G=(V,E)$ with $|V|\ge36$, every $f$, every $k\in\{2,\ldots,|V|/2\}$, every $\mathbf{x}=(\mathbf{z},x)\in(V)_k$, and every realisation $(\mathbf{x}_t^\mathrm{IP})_{t\ge0}$ of IP$(k,f,G)$ started from state $\mathbf{x}$, there exists a continuous-time Markov process $(M_t)_{t\ge0}:=(\mathbf{z}_t^C,R_t,P_t,W_t)_{t\ge0}$ with state-space $\Omega_k(V)$ defined on the same probability space as $(\mathbf{x}_t^\mathrm{IP})_{t\ge0}$ satisfying:
\begin{enumerate}
\item\label{item:z_same_path} $(\mathbf{z}_t^\mathrm{IP})_{t\ge0}=(\mathbf{z}^C_t)_{t\ge0}$ almost surely;
\item\label{item:ink} for every $t\ge0$ and $\mathbf{b}=(\mathbf{c},b)\in(V)_k$,
\[
\P\bra{\mathbf{x}_t^\mathrm{IP}=\mathbf{b}}=\E\bra{\mathrm{ink}^{\bf
		x}_t(b)\indic{\mathbf{z}_t^C=\mathbf{c}}},
\]
where $\mathrm{ink}^{\bf
	x}_t(b):=\indic{b\in R_t}+\frac1{2}\indic{b\in P_t}$;
\item\label{item:exp_ink} for every $t\ge0$ and $j\in\mathbb{N}$, \[\E\left[1-\frac{\mathrm{ink}^{\bf
		x}_t}{|V|-k+1}\Big|\,\mathrm{Fill}^{\bf
	x}\right]\le
  c_1\sqrt{|V|}e^{-c_2j}+\exp\left\{j-\frac{t}{\kappa_1\,T_{\mathrm{EX}(4,f,G)}(1/4)}\right\}\]
  where $\mathrm{ink}^{\bf
  	x}_t:=\sum_{b\in V}\mathrm{ink}^{\bf
  	x}_t(b)$ and $\mathrm{Fill}^{\bf
  	x}:=\left\{\lim_{t\to\infty}\mathrm{ink}^{\bf
  	x}_t=|V|-k+1\right \}$;
  \item\label{item:fill_independent}  for every $t\ge0$ and ${\bf c}\in(V)_{k-1},$ \[
  \P\bra{\{{\bf z}_t^C={\bf c}\}\cap\mathrm{Fill}^{\bf
  		x}}=\frac{\P\bra{{\bf z}_t^C={\bf c}}}{|V|-k+1}.
  \]
\end{enumerate}
\end{lemma}

The proof of Lemma~\ref{L:existence} is deferred to Section~\ref{S:chamhasprops} and is a proof by construction: in Section~\ref{S:cham} we will explicitly define a process and then proceed to show that it has the desired properties. We can now relate the total-variation distance between two
realisations of $\mathrm{IP}(k,f,G)$ to a certain expectation
involving the amount of ink in the chameleon process $M$ in the statement of Lemma~\ref{L:existence}. The following result is similar to Lemma 6.1 of \citep{Oliveira2013a}: we include a
sketch of the proof to highlight the importance of constructing in
Section~\ref{S:cham} a chameleon process satisfying
part \ref{item:ink} of Lemma~\ref{L:existence}.

\begin{lemma}\label{L:TVtoink}For every $t\ge0$,\[
  \sup_{{\bf x},{\bf y}\in(V)_k}\|\cL[{\bf
    x}_t^{\mathrm{IP}}]-\cL[{\bf
    y}_t^{\mathrm{IP}}]\|_{\mathrm{TV}}\le 2k\sup_{{\bf w}\in
    (V)_k}\E\left[1-\frac{\mathrm{ink}_t^{\bf
        w}}{|V|-k+1}\Big|\,\mathrm{Fill}^{\bf w}\right]
  \]
\end{lemma}
\begin{proof} Fix ${\bf x}=({\bf z},x)\in(V)_k$ with ${\bf
    z}\in(V)_{k-1}$, and denote by ${\bf x}_t^{\mathrm{IP}}$ an
  interchange process started from ${\bf x}$. Let $\tilde x$ be
  uniform from $V\setminus{\bf O}({\bf z})$ and denote by $\tilde{\bf
    x}_t^{\mathrm{IP}}$ an interchange process started from
  $\tilde{\bf x}=({\bf z},\tilde x)$. Then for any ${\bf b}=({\bf
    c},b)\in (V)_k$,
  \[
  \P\bra{\tilde{\bf x}_t^{\mathrm{IP}}={\bf b}}=\frac{\P\bra{{\bf
        z}_t^{\mathrm{IP}}={\bf c}}}{|V|-k+1}=\frac{\P\bra{{\bf
        z}_t^C={\bf c}}}{|V|-k+1}=\P\bra{\{{\bf z}_t^C={\bf
      c}\}\cap\mathrm{Fill}^{\bf x}},
  \]
  where the second and third equalities follow from parts~\ref{item:z_same_path} and \ref{item:fill_independent} of Lemma \ref{L:existence}, respectively. On the other hand, part~\ref{item:ink} of Lemma \ref{L:existence} gives
  \[
  \P\bra{{\bf x}_t^{\mathrm{IP}}={\bf
      b}}=\E[\mathrm{ink}^{\bf
    x}_t(b)\indic{{\bf z}_t^C={\bf
      c}}]\ge\E[\mathrm{ink}^{\bf
    x}_t(b)\indic{\{{\bf z}_t^C={\bf
      c}\}\cap\mathrm{Fill}^{\bf x}}].
  \]
  Subtracting we obtain
  \[
  \P\bra{\tilde{\bf x}_t^{\mathrm{IP}}={\bf b}}-\P\bra{{\bf
      x}_t^{\mathrm{IP}}={\bf b}}\le
  \E[(1-\mathrm{ink}_t^{\bf x}(b))\indic{\{{\bf z}^C_t={\bf
      c}\}\cap\mathrm{Fill}^{\bf x}}].
  \]Hence
  \begin{align*}
    \|\cL[{\bf x}_t^{\mathrm{IP}}]-\cL[\tilde{\bf x}_t^{\mathrm{IP}}]\|_{\mathrm{TV}}&\le \sum_{({\bf c},b)\in(V)_k}\E[(1-\mathrm{ink}_t^{\bf x}(b))\indic{\{{\bf z}_t^C={\bf c}\}\cap\mathrm{Fill}^{\bf x}}]\\&=\E[(|V|-k+1-\mathrm{ink}^{\bf x}_t)\indic{\mathrm{Fill}^{\bf x}}]\\
    &=\E\left[ 1 - \frac{\mathrm{ink}^{\bf
          x}_t}{|V|-k+1}\Big|\,\mathrm{Fill}^{\bf x}\right]\,.
  \end{align*}
  The result now follows by repeated application of the triangle
  inequality, as in the proof of Lemma 6.1 of \citep{Oliveira2013a}.
\end{proof}
\begin{proof}[Proof of Lemma \ref{L:kto4}]We combine part~\ref{item:exp_ink} of Lemma \ref{L:existence} with Lemma \ref{L:TVtoink} to give
  for every $t\ge0$ and $j\in\mathbb{N}$,\[ \sup_{{\bf x},{\bf
      y}\in(V)_k}\|\cL[{\bf x}_t^{\mathrm{IP}}]-\cL[{\bf
    y}_t^{\mathrm{IP}}]\|_{\mathrm{TV}}\le
  2k\left\{c_1\sqrt{|V|}e^{-c_2j}+\exp\left\{j-\frac{t}{\kappa_1T_{\mathrm{EX}(4,f,G)}(1/4)}\right\}\right\},
  \]
  for some universal positive constants $c_1,c_2$ and $\kappa_1$. We
  choose
  \[
  j=\left\lfloor\frac{t}{(1+c_2)\kappa_1T_{\mathrm{EX}(4,f,G)}(1/4)}\right\rfloor,
  \]
  which gives the bound (using $k\le |V|$),
  \[
  \sup_{{\bf x},{\bf y}\in(V)_k}\|\cL[{\bf
    x}_t^{\mathrm{IP}}]-\cL[{\bf
    y}_t^{\mathrm{IP}}]\|_{\mathrm{TV}}\le
  c_3|V|^{3/2}\exp\left\{\frac{-c_2t}{(1+c_2)\kappa_1T_{\mathrm{EX}(4,f,G)}(1/4)}\right\},
  \]
  for some positive $c_3$. Therefore there exists a universal constant
  $C$ such that for any $\eps\in(0,1/2)$ and
  $t>CT_{\mathrm{EX}(4,f,G)}(1/4)\log(|V|/\eps)$,
  \[
  \sup_{{\bf x},{\bf y}\in(V)_k}\|\cL[{\bf
    x}_t^{\mathrm{IP}}]-\cL[{\bf
    y}_t^{\mathrm{IP}}]\|_{\mathrm{TV}}\le \eps.\qedhere
  \]
\end{proof}

\section{From 4-particle exclusion to 2-particle exclusion}\label{S:4to1}

  In this section we shall prove Lemma \ref{L:4to2}. We begin by characterizing every connected hypergraph in terms of how long it takes two independent random walkers on the hypergraph to arrive onto the
  same edge, which then rings for one of the walkers -- a time we shall
  refer to as the \emph{meeting time} of the two walkers (note that we
  do not require the two walkers to actually occupy the same vertex). It
  will be useful to consider such times, as we will be able to couple
  two independent walkers with a 2-particle interchange process, up until this meeting time (see Proposition \ref{P:couplingIPwithRW} for this statement).
  
  Formalising this, for ${\bf y}\in V^2$, let $({\bf y}^{\mathrm{RW}}_t)_{t\ge0}$ be a realisation
  of RW$(2,f,G)$ with ${\bf y}^{\mathrm{RW}}_0={\bf y}$. Denote by
  $\Lambda^1$ and $\Lambda^2$ the Poisson processes used to generate the
  edge-ringing times for the two particles, and let $\{e^1_n\}_{n\in\mathbb{N}}$ and
  $\{e^2_n\}_{n\in\mathbb{N}}$ be the two sequences of edge-choices (all
  as in Section
  \ref{S:graphical}).
  Define $M^{\mathrm{RW}}({\bf y})$ to be the first time ${\bf
  	y}^{\mathrm{RW}}_t(1)$ and ${\bf y}^{\mathrm{RW}}_t(2) $ are in the
  same edge which then rings in one of the processes:
  \begin{equation}\label{eq:meeting_time}
  M^{\mathrm{RW}}({\bf y}):=\inf\big\{t>0:\,\exists
  e\in\{e^1_{\Lambda^1[0,t]},e^2_{\Lambda^2[0,t]}\} \text{ with } {\bf
  	y}^{\mathrm{RW}}_{t-}(1), {\bf y}^{\mathrm{RW}}_{t-}(2)\in
  e\big\}.
  \end{equation}
  
  \begin{defn}
  	We say that a hypergraph $G$ is \emph{easy} if
  	\[
  	\sup_{{\bf y}\in V^2}\P\bra{M^{\mathrm{RW}}({\bf
  			y})>10^{10}T_{\mathrm{EX}(2,f,G)}(1/4)}\le 1/1000.
  	\]
  \end{defn}
  \begin{rmk}
  	We note that this definition is similar to Definition 4.1 of
  	\citep{Oliveira2013a}, from where we borrow the dichotomy ``easy/non-easy''. However, for the case of hypergraphs, this characterisation does not reflect the associated difficulty of dealing with each case! One difference in the case of hypergraphs is
  	that at the meeting time we cannot guarantee that the two independent
  	walkers occupy the same site, and this results in the analysis being
  	more challenging. 
    \end{rmk}

  \subsection{From 4-particle exclusion to 2-particle exclusion:
  	easy hypergraphs}\label{SS:4to2easy}
  We present a preliminary result which shows that we can couple two $k$-particle exclusion processes  initially sharing $k-1$ occupied vertices such that, with positive probability, at the meeting time of the $k$th particles the two processes will agree, given the $k$th particles meet on an edge of size at least 5 and the permutation chosen at this meeting time does not fix the $k$th particles.
  
Let $\Lambda$ be a Poisson process of rate $2|E|$ (i.e. twice the usual rate), with associated edge-choices $\{e_n\}_{n\in\mathbb{N}}$ and permutations $\{\sigma_n\}_{n\in\mathbb{N}}$ as in Section \ref{S:graphical}. In addition, let $\{\theta_n\}_{n\in\mathbb{N}}$ be an i.i.d. sequence of Bernoulli$(1/2)$ random variables: these will be used to thin the events of $\Lambda$ and ensure that all particles are moving at the correct rate. We write $\hat \Lambda$ for the thinned Poisson process obtained from $\Lambda$ by removing all points corresponding to $\theta_n=0$. Let $\hat I_t$ be constructed from $\hat\Lambda$ as in Section~\ref{S:graphical}.  	We make this modification as it allows us to more easily compare a certain time to the meeting time of two independent random walkers as defined in \eqref{eq:meeting_time}. 
  
\begin{lemma}\label{L:matchupkth}
Let $D\in \binom{V}{k-1}$, $a,b\in V\setminus D$ and $A=D\cup\{a\}$, $B=D\cup\{b\}$. Let $(A_t^{\widehat{\mathrm{EX}}})_{t\ge0}$ and $(B_t^{\widehat{\mathrm{EX}}})_{t\ge0}$ be two realisations of EX$(k,f,G)$ started from $A$ and $B$ respectively and evolving according to $\hat I_t$. Let 	\[
 	\tau_{a,b}:=\inf\{t\ge 0:\, \hat I_t(a),\,\hat I_t(b)\in e_{\Lambda[0,t]} \}\,,
 	\] 	and write $e_{a,b}$ for $e_{\Lambda[0,\tau_{a,b}]}$ and $\sigma_{a,b}$ for $\sigma_{\Lambda[0,\tau_{a,b}]}$. Write also $a^\ast$ for $\hat I_{[0,\tau_{a,b})}(a)$ and $b^\ast$ for $\hat I_{[0,\tau_{a,b})}(b)$. Then there exist two other realisations of EX$(k,f,G)$ denoted $(A_t^{\widetilde{\mathrm{EX}}})_{t\ge0}$ and $(B_t^{\overline{\mathrm{EX}}})_{t\ge0}$ which start and evolve identically to $(A_t^{\widehat{\mathrm{EX}}})_{t\ge0}$ and $(B_t^{\widehat{\mathrm{EX}}})_{t\ge0}$ respectively up to time $\tau_{a,b}-$ but which satisfy, on the event $$\{\sigma_{a,b}(a^\ast)\neq a^\ast\}\cap\{|e_{a,b}|>4\},$$ with probability at least $\frac2{25}$,  $A_t^{\widetilde{\mathrm{EX}}}=B_t^{\overline{\mathrm{EX}}}$ for all $t\ge\tau_{a,b}$.
\end{lemma}
\begin{proof}
We define two events which will be used to determine the coupling strategy of the processes $(A_t^{\widetilde{\mathrm{EX}}})_{t\ge0}$ and $(B_t^{\overline{\mathrm{EX}}})_{t\ge0}$ at time $\tau_{a,b}$:
 	\begin{align*}
 	J_1(\sigma_{a,b}) &= \big\{ \sigma_{a,b} (a^*)\notin \{ a^*,\,b^*\},\, \sigma_{a,b} (b^*) \notin \{ a^*,\,b^*\}\big\} \\
 	J_2(\sigma_{a,b}) &= J_1(\sigma_{a,b}) \cap \left\{\big|\hat I_{\tau_{a,b}-}(D) \cap \sigma_{a,b} (a^*)\big| = \big|\hat I_{\tau_{a,b}-}(D)\cap \sigma_{a,b} (b^*)\big|   \right\} \,.
 	\end{align*}
 	In words, $J_1(\sigma_{a,b})$ is the event that the permutation $\sigma_{a,b}$ moves the set of two `special' particles (those initially at vertices $a$ and $b$) to a new set of positions; event $J_2(\sigma_{a,b})$ further specifies that the two positions to which $\sigma_{a,b}$ moves the special particles should either both contain another particle (i.e. one of the already-matched $k-1$ particles) or both be empty.
 	
 	With this notation in place, we can describe the coupling at time $\tau_{a,b}$:
 	\begin{itemize}
 		\item[(i)] if $\theta_{a,b} = 0$ then we do not update the processes at time $\tau_{a,b}$;
 		\item[(ii)] if $\theta_{a,b} = 1$ but event $J_2(\sigma_{a,b})$ fails to hold, then we apply permutation $\sigma_{a,b}$ in both processes;
 		\item[(iii)] if $\theta_{a,b} = 1$ and event $J_2(\sigma_{a,b})$ holds, we update the `$A$' process with permutation $\sigma_{a,b}$ and the `$B$' process with permutation $\overline\sigma_{a,b}$, where
 		\[ \overline\sigma_{a,b} = \sigma_{a,b}\circ\big(\sigma_{a,b}(a^*)\,\,\sigma_{a,b}(b^*)\big)\big(\sigma^2_{a,b}(a^*)\,\,\sigma^2_{a,b}(b^*)\big)\,. \]  (Here and throughout we use the convention
 		that composition of permutations corresponds to multiplication on the right: $\sigma\circ\rho = \rho\sigma$.)
 	\end{itemize}
 	Figure~\ref{f:l44} demonstrates the relationship between $\sigma_{a,b}$ and $\overline\sigma_{a,b}$ in the simple case where $\sigma_{a,b}$ is a single cycle. 
 	\begin{figure}[!ht]
 	\centering
 	\includegraphics[scale=1]{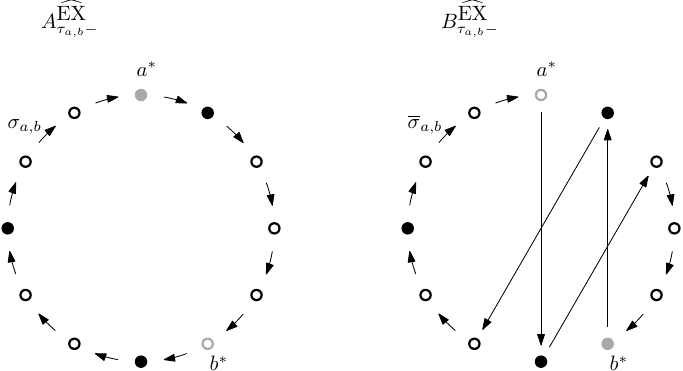}
 	 	\caption{The left/right image shows the state of the process on edge $e_{a,b}$ at time $\tau_{a,b}-$ in the `A'/`B' process. Also indicated  are the permutations $\sigma_{a,b}$ and $\overline\sigma_{a,b}$ which are to be applied in case (iii).}
 	\label{f:l44}
 	\end{figure}
 	To show that this is a valid coupling, it suffices to show that in case (iii) the permutation $\overline\sigma_{a,b}$ belongs to the same conjugacy class as $\sigma_{a,b}$, and that there is a bijection between the two permutations. By inspection, the cyclic decomposition of $\overline\sigma_{a,b}$ is obtained from that of $\sigma_{a,b}$ just by exchanging the elements $\sigma_{a,b}(a^*)$ and $\sigma_{a,b}(b^*)$, and so both permutations belong to the same conjugacy class. Moreover, there is a bijection between them since 
 	\[
 	\sigma_{a,b}(a^*)=\overline\sigma_{a,b}(b^*)\quad\text{and}\quad 	\sigma_{a,b}(b^*)=\overline\sigma_{a,b}(a^*) \,,
 	\]
 	and so $J_1(\sigma_{a,b}) =J_1(\overline\sigma_{a,b})$ and $J_2(\sigma_{a,b}) =J_2(\overline\sigma_{a,b})$.
 	
 	Furthermore, it follows from the above analysis that our coupling strategy in case (iii) gives $\sigma_{a,b}(a^*)=\tilde\sigma_{a,b}(b^*)$,  and furthermore, $\sigma_{a,b}(\hat I_{\tau_{a,b}-}(D))=\overline\sigma_{a,b}(\hat I_{\tau_{a,b}-}(D))$. Thus in order to complete this proof, we need to show that $\P\bra{\theta_{a,b}=1,\,J_2(\sigma_{a,b})\,|\, \sigma_{a,b}(a^\ast)\neq a^\ast\},\,\{|e_{a,b}|>4} \ge 2/25$.

We have
\begin{align*}
\P\bra{J_1(\sigma_{a,b})\mid \sigma_{a,b}(a^\ast)\neq a^\ast,\,|e_{a,b}|>4}&=\P\bra{\sigma_{a,b}(a^\ast)\notin\{a^\ast,b^\ast\}\mid \sigma_{a,b}(a^\ast)\notin a^\ast,\,|e_{a,b}|>4}\\
&\qquad\qquad\cdot\P\bra{\sigma_{a,b}(b^\ast)\notin\{a^\ast,b^\ast\}\mid \sigma_{a,b}(a^\ast)\notin\{a^\ast,b^\ast\},\,|e_{a,b}|>4}.
\end{align*}
Using parts \ref{assumpf-conjugacy} and \ref{assumpf-fixed} of Assumption~\ref{assumpf}) this becomes
 	\begin{align*}
 	\P\bra{J_1(\sigma_{a,b})\mid \sigma_{a,b}(a^\ast)\neq a^\ast,\,|e_{a,b}|>4}&\ge \frac45\left(1-\P\bra{\sigma_{a,b}(b^\ast)=b^\ast}-\frac{1-\P\bra{\sigma_{a,b}(b^\ast)=b^\ast}}{4}\right)\ge\frac{12}{25}.
 	\end{align*}
 	 Therefore,
 	\begin{align*}
 	\P\bra{\theta_{a,b}=1,\,J_2(\sigma_{a,b})\mid \sigma_{a,b}(a^\ast)\neq a^\ast,\,|e_{a,b}|>4} &= \frac12 \P\bra{J_1(\sigma_{a,b})\mid \sigma_{a,b}(a^\ast)\neq a^\ast,\,|e_{a,b}|>4}\\
 	&\qquad\qquad\cdot\P\bra{J_2(\sigma_{a,b})\mid J_1(\sigma_{a,b}),\,  \sigma_{a,b}(a^\ast)\neq a^\ast,\,|e_{a,b}|>4} \nonumber \\
 	&\ge \frac{6}{25} \,\P\bra{J_2(\sigma_{a,b})\mid J_1(\sigma_{a,b}),\, \sigma_{a,b}(a^\ast)\neq a^\ast,\,|e_{a,b}|>4}\,. \nonumber 
 	\end{align*}
 	But conditioned on $J_1(\sigma_{a,b})$ and $\{\sigma_{a,b}(a^\ast)\neq a^\ast,\,|e_{a,b}|>4\}$ both holding, $J_2(\sigma_{a,b})$ is the event that two of the positions in $e_{a,b}$ not containing $a^*$ or $b^*$ either both contain a matched particle or are both empty; since $|e_{a,b}|\ge 5$ this probability is at least $1/3$, thanks to part \ref{assumpf-conjugacy} of Assumption~\ref{assumpf}, and so our proof is complete.	\qedhere
 	
 \end{proof}
 We now present the main result of this section.
 \begin{lemma}\label{L:ktok-1}
 	There exists $\kappa>0$ such that for any easy hypergraph $G$, any $f$ and $0<\varepsilon<1/2$,
 	\[T_{\mathrm{EX}(k,f,G)}(\varepsilon)\le \kappa\log(1/\varepsilon) T_{\mathrm{EX}(k-1,f,G)}(1/4),\] for any $3\le k\le |V|/2$ if $|V|<36$ and any $k\in\{3,4\}$ if $|V|\ge36$.
 \end{lemma}
 In this section we will make use of this lemma only for the case $|V|\ge 36$, but this result will later be used in its full form when dealing with the case of $|V|<36$: see Section \ref{S:smallV}. The proof uses a coupling argument for two realisations of EX($k,f,G$).
 
 \begin{proof}
 	For $U=\{u_1,\ldots,u_k\}, W=\{w_1,\ldots,w_k\}\in \binom{V}{k}$, let $(U^{\mathrm{EX}}_t)_{t\ge0}$ and $(W^{\widetilde{\mathrm{EX}}}_t)_{t\ge0}$
 	be two realisations of EX$(k,f,G)$ started from $U$ and $W$
 	respectively.
 	 	We define the two processes on a common probability space, and will show how to couple them in such a way that we can lower-bound the probability that $U_{\kappa T}^\mathrm{EX}=W_{\kappa T}^{\widetilde{\mathrm{EX}}}$ for some $\kappa>0$ to be determined, where $T:=T_{\mathrm{EX}(k-1,f,G)}(1/4)$. The result will then follow by applying \eqref{eq:tvcoupling}.
 	
 	We begin by allowing the two processes to evolve independently up to time $10T$. Then, for any $S\subset\binom{V}{k}$ and $t\ge0$, we have
 	\begin{align*}
 	\P\bra{U_{10T+t}^\mathrm{EX}\in S}-\P\bra{W_{10T+t}^{\widetilde{\mathrm{EX}}}\in S}&=\E\left[\P\bra{U_{10T+t}^\mathrm{EX}\in S|\,U_{10T}^\mathrm{EX}}-\P\bra{W_{{10T}+t}^{\widetilde{\mathrm{EX}}}\in S|\,W_{10T}^{\widetilde{\mathrm{EX}}}}\right]\\
 	&\le\E\left[\|\mathcal{L}[U_{{10T}+t}^\mathrm{EX}\,|\,U_{10T}^\mathrm{EX}]-\mathcal{L}[W_{{10T}+t}^{\widetilde{\mathrm{EX}}}\,|\,W_{10T}^{\widetilde{\mathrm{EX}}}]\|\tv\right]\,,
 	\end{align*}
 	where the inequality follows from \eqref{eq:tvdef}.
 	Maximizing over $S$ and again using \eqref{eq:tvdef} gives
 	\begin{align}\label{eq:tvUW}
 	\|\mathcal{L}[U_{{10T}+t}^\mathrm{EX}]-\mathcal{L}[W_{{10T}+t}^{\widetilde{\mathrm{EX}}}]\|\tv\le\E\left[\|\mathcal{L}[U_{{10T}+t}^\mathrm{EX}\,|\,U_{10T}^\mathrm{EX}]-\mathcal{L}[W_{{10T}+t}^{\widetilde{\mathrm{EX}}}\,|\,W_{10T}^{\widetilde{\mathrm{EX}}}]\|\tv\right]\,.
 	\end{align}
 	By the Markov property, for any $A,B\in \binom{V}{k}$,
 	\begin{align}\notag
 	\|\mathcal{L}[U_{{10T}+t}^\mathrm{EX}\,|\,U_{10T}^\mathrm{EX}= A]-\mathcal{L}[W_{{10T}+t}^{\widetilde{\mathrm{EX}}}\,|\,W_{10T}^{\widetilde{\mathrm{EX}}}=B]\|\tv &= \| \mathcal{L}[A_t^{\mathrm{EX}}]-\mathcal{L}[B_t^{\widetilde{\mathrm{EX}}}]\|\tv\\
 	&\le\P\bra{A_t^\mathrm{EX}\neq B_t^{\widetilde{\mathrm{EX}}}}\,,\label{eq:tvcondUW}
 	\end{align}
 	for any coupling of $(A_t^\mathrm{EX})_{t\ge0}$ and $(B_t^{\widetilde{\mathrm{EX}}})_{t\ge0}$, by \eqref{eq:tvcoupling}, and where $\mathcal{L}[\cdot|\cdot]$ denotes a conditional law.
 	
Recall from Section~\ref{S:graphical} the construction of the permutation $I_t$ for each $t\ge0$. 	
For any $A,B\in\binom{V}{k}$, let $a$ and $b$ be two uniformly and independently chosen elements of $A$ and $B$, respectively. Given $a$, consider now the $k$-particle process $(A^a_t)_{t\ge0}=(I_t(a),I_t(A\setminus\{a\}))_{t\ge0}$ which evolves in the same way as the exclusion process begun at $A$, but with the label of the particle started from position $a$ being tracked.
 Thus $(A^a_t)_{t\ge0}$ can be thought of as something `between' an exclusion process (in which no labels are tracked) and an interchange process (in which all labels are tracked). It's clear that the $k-1$ particles initially at vertices in $A\setminus\{a\}$ behave marginally as an exclusion process, while the particle started from $a$ behaves (again marginally) as a random walk on $G$. Furthermore, the exclusion process $(A^\mathrm{EX}_t)_{t\ge0}$ can be recovered from $(A^a_t)_{t\ge0}$ simply by `forgetting' which position is occupied by the `special' particle starting from $a$, i.e. $A^\mathrm{EX}_t=\{I_t(a),I_t(A\setminus\{a\})\}$. In a similar manner, for given $b$ and another permutation-valued process $(\tilde I_t)_{t\ge0}$, we also define the process $(\tilde B_t^{b})_{t\ge0}=(\tilde I_t(b),\tilde I_t(B\setminus\{b\}))_{t\ge0}$.
 	
 	Over the time period $[0,10T]$ we couple the processes $(A^a_t)_{t\ge0}$ and $(\tilde B^b_t)_{t\ge0}$ using a maximal coupling of the $(k-1)$-particle exclusion processes $I_{t}(A\setminus\{a\})$ and $\tilde I_{t}(B\setminus\{b\})$. (Recall that a maximal coupling is one which achieves equality in the coupling inequality \eqref{eq:tvcoupling}. This maximal coupling is actually more than is needed here; we will only be interested in the state of the processes at time $10T$.) By Proposition \ref{P:1to1eps} we have
 	\begin{align}\label{eq:TEX10T}
 	T_{\mathrm{EX}(k-1,f,G)}\left(1/500\right)\le\left\lceil\frac{\log\left(\frac1{500}\right)}{\log\left(\frac1{2}\right)}\right\rceil T<10T.
 	\end{align} 
 	Given the choice of $a$ and $b$, let $F_{a,b}$ denote the event that the other $(k-1)$ particles have coupled by time $10T$, i.e. $F_{a,b}=\big\{I_{10T}(A\setminus\{a\})=\tilde I_{10T}(B\setminus\{b\})\big\}$.
 	Using this maximal coupling it follows from \eqref{eq:TEX10T} that $\P\bra{F_{a,b}} \ge 499/500$.
 	 	Combining this with equations \eqref{eq:tvUW} and \eqref{eq:tvcondUW} we see that for any $K\in \mathbb N$,
 	\begin{align}
 	&\|\mathcal{L}[U_{(20+K)T}^\mathrm{EX}]-\mathcal{L}[W_{(20+K)T}^{\widetilde{\mathrm{EX}}}]\|\tv\nonumber\\
 	&\le\sum_{A,B\in\binom{V}{k}}\P\bra{U_{10T}^\mathrm{EX}=A,\,W_{10T}^{\widetilde{\mathrm{EX}}}=B}\P\bra{A_{(10+K)T}^\mathrm{EX}\neq B_{(10+K)T}^{\widetilde{\mathrm{EX}}}} \nonumber \\
 	&=\sum_{A,B\in\binom{V}{k}}\P\bra{U_{10T}^\mathrm{EX}=A}\P\bra{W_{10T}^{\widetilde{\mathrm{EX}}}=B}\P\bra{A_{(10+K)T}^\mathrm{EX}\neq B_{(10+K)T}^{\widetilde{\mathrm{EX}}}} \nonumber\\
 	&\le\sum_{A,B\in\binom{V}{k}}\P\bra{U_{10T}^\mathrm{EX}=A}\P\bra{W_{10T}^{\widetilde{\mathrm{EX}}}=B}\sum_{\substack{a\in A\\b\in B}}\frac1{k^2}\left(1-\P\bra{F_{a,b}} + \P\bra{A_{(10+K)T}^\mathrm{EX}\neq B_{(10+K)T}^{\widetilde{\mathrm{EX}}},\, F_{a,b}}\right) \nonumber \\
 	\label{eq:tvUWfinal}
 	&\le \frac1{500}+\sum_{A,B\in\binom{V}{k}}\P\bra{U_{10T}^\mathrm{EX}=A}\P\bra{W_{10T}^{\widetilde{\mathrm{EX}}}=B}\sum_{\substack{a\in A\\b\in B}}\frac1{k^2}\P\bra{A_{(10+K)T}^\mathrm{EX}\neq B_{(10+K)T}^{\widetilde{\mathrm{EX}}},\, F_{a,b}}\,,
 	\end{align}
 	where the equality is thanks to the independence of $U$ and $W$ over $[0,10T]$.
 	
 	 	From \eqref{eq:tvUWfinal} we see that we now need to upper bound the probability that $(A^\mathrm{EX}_t)_{t\ge0}$ and $(B^{\widetilde{\mathrm{EX}}}_t)_{t\ge0}$ do not agree by time $(10+K)T$, on the event $F_{a,b}$. As pointed out above, this event is equivalent (on $F_{a,b}$) to the \emph{locations} of the $k$ particles in $A^a_{(10+K)T}$ and $\tilde B^b_{(10+K)T}$ not agreeing.
 	
 	We shall bound this probability by coupling the processes $(A_{10T+t}^a)_{t\ge0}$ and $(\tilde B_{10T+t}^{ b})_{t\ge 0}$ in the following manner. Recall the Poisson process $\Lambda$ of rate $2|E|$ at the start of Section~\ref{SS:4to2easy} with associated edge-choices $\{e_n\}_{n\in\mathbb{N}}$, permutations $\{\sigma_n\}_{n\in\mathbb{N}}$, and Bernoulli $(1/2)$ random variables $\{\theta_n\}_{n\in\mathbb{N}}$ (used to thin the events of $\Lambda$). 
Prior to a time $\tau_{a,b}$ defined below we evolve $(A_{10T+t}^a)_{t\ge0}$ and $(\tilde B_{10T+t}^b)_{t\ge 0}$ by applying permutation $\sigma_n$ to edge $e_n$ (in both processes) at the $n^\text{th}$ incident time of $\Lambda$ if and only if $\theta_n=1$, so formally we have for each $0\le t<\tau_{a,b}$,
 	\[
 	A_{10T+t}^a=\hat I_t\big(I_{10T}(a),I_{10T}(A\setminus\{a\})\big),\qquad \tilde B_{10T+t}^{ b}=\hat I_t\big(\tilde I_{10T}(b),\tilde I_{10T}(B\setminus\{b\})\big).
 	\]
Note that, since we use a common set of innovations over the period \mbox{$[10T,10T+\tau_{a,b})$,}  on event $F_{a,b}$ we have $D:=\hat I_{\tau_{a,b}-}(I_{10T}(A\setminus\{a\}))=\hat I_{\tau_{a,b}-}(\tilde I_{10T}(B\setminus\{b\}))$; that is, the locations of the $k-1$ unlabelled particles of $A^a$ and $\tilde B^b$ still agree at time $\tau_{a,b}-$.  By the Markov property, on event $F_{a,b}$ we can thus write
\[
A_{10T+t}^a=\hat I_t(I_{10T}(a),D),\qquad \tilde B_{10T+t}^{ b}=\hat I_t\big(\tilde I_{10T}(b),D).
\]	
 	
 	We define $\tau_{a,b}$ to be the first time that the `special' particles initially at $a$ and $b$ are in a common edge which then rings (note this has a slightly different definition from $\tau_{a,b}$ defined in the statement of Lemma~\ref{L:matchupkth}):
 	\[
 	\tau_{a,b}:=\inf\{t\ge 0:\, \hat I_t(I_{10T}(a)),\,\hat I_t(\tilde I_{10T}(b))\in e_{\Lambda[0,t]} \}\,.
 	\]

Note that the processes $(\hat I_t(I_{10T}(a)))_{t\ge0}$ and $(\hat I_t(\tilde I_{10T}(b)))_{t\ge0}$ when viewed marginally behave as independent random walks over the period $[0,\tau_{a,b})$, and so $\tau_{a,b}$ has the same distribution as the meeting time $M^{\mathrm{RW}}( I_{10T}(a),\tilde I_{10T}(b))$ in \eqref{eq:meeting_time}.

To determine how to couple the processes at time $\tau_{a,b}$ we partition the probability space according to the following four sets (for some $K\in\mathbb N$ which is yet to be determined), denoting $\hat I_{\tau_{a,b}-}(I_{10T}(a))$ by $a^\ast$ for ease of readability:
 	\begin{align*}
 	E^1_{a,b}&:=\{\tau_{a,b}>KT\},\\
 	 	E^2_{a,b}&:=\{\tau_{a,b}\le KT,\,\sigma_{a,b}(a^\ast)=a^\ast\},\\
 	E^3_{a,b}&:=\{\tau_{a,b}\le KT,\,\sigma_{a,b}(a^\ast)\neq a^\ast,\,|e_{a,b}|> 4\},\\
 	E^4_{a,b}&:=\{\tau_{a,b}\le KT,\,\sigma_{a,b}(a^\ast)\neq a^\ast,\,|e_{a,b}|\leq 4\}.
 	\end{align*}
 	
For the first two cases, we shall not specify the coupling, as it does not matter how we update the processes at time $\tau_{a,b}$. 	First, for the case of $E^1_{a,b}$, we have 
 	\begin{align}
 	\sum_{A,B\in\binom{V}{k}}&\P\bra{U_{10T}^\mathrm{EX}=A}\P\bra{W_{10T}^{\widetilde{\mathrm{EX}}}=B}\sum_{\substack{a\in A\\b\in B}}\frac1{k^2}\P\bra{A_{(10+K)T}^\mathrm{EX}\neq B_{(10+K)T}^{\widetilde{\mathrm{EX}}},\, F_{a,b},\, E^1_{a,b}}\nonumber \\
 	&\le \sum_{A,B\in\binom{V}{k}}\P\bra{U_{10T}^\mathrm{EX}=A}\P\bra{W_{10T}^{\widetilde{\mathrm{EX}}}=B}\sum_{\substack{a\in A\\b\in B}}\frac1{k^2}\P\bra{E^1_{a,b}}\le \max_{a,b}\P\bra{E^1_{a,b}}.\label{eq:E1_bound}
 	\end{align}
 	
 	Second, for the case of $E^2_{a,b}$, we have
 	\begin{align}
 	\sum_{A,B\in\binom{V}{k}}&\P\bra{U_{10T}^\mathrm{EX}=A}\P\bra{W_{10T}^{\widetilde{\mathrm{EX}}}=B}\sum_{\substack{a\in A\\b\in B}}\frac1{k^2}\P\bra{A_{(10+K)T}^\mathrm{EX}\neq B_{(10+K)T}^{\widetilde{\mathrm{EX}}},\,F_{a,b},\,E^2_{a,b}}\nonumber\\
 	&\le \sum_{A,B\in\binom{V}{k}}\P\bra{U_{10T}^\mathrm{EX}=A}\P\bra{W_{10T}^{\widetilde{\mathrm{EX}}}=B}\sum_{\substack{a\in A\\b\in B}}\frac1{k^2}\P\bra{\sigma_{a,b}(a^\ast)=a^\ast} \nonumber\\
 	 	&=\sum_{a,b\in V}\frac1{k^2}\P\bra{\sigma_{a,b}(a^\ast)=a^\ast}\P\bra{a\in U_{10T}^\mathrm{EX}}\P\bra{b\in W_{10T}^{\widetilde{\mathrm{EX}}}} \nonumber\\
 	&\le \frac{k^2}{|V|^2}\sum_{a,b\in V}\frac1{k^2}\P\bra{\sigma_{a,b}(a^\ast)=a^\ast} + \|\cL[(a,b)_{10T}^{\mathrm{RW}}]-\textrm{Unif}(V^2)\|_{\mathrm{TV}} \quad\quad\text{(using \eqref{eq:law_I_inverse} and \eqref{eq:tvint})}\label{eq:TV_fn_bound} \\
 	&\le \frac1{|V|^2}\sum_{a,b\in V}\P\bra{\sigma_{a,b}(a^\ast)=a^\ast} + \frac{2}{500}\,,\label{eq:E2_bound}
 	\end{align}
 	where the last inequality uses \eqref{eq:prodtv}, \eqref{eq:TEX10T} and the contraction principle.
 	
 	Third, conditioned on the event $E^3_{a,b}$ and $F_{a,b}$, by Lemma~\ref{L:matchupkth} we can couple the processes so that $\hat I_{\tau_{a,b}}(I_{10T}(A))=\hat I_{\tau_{a,b}}(\tilde I_{10T}(B))$ with probability at least $2/25$, giving
 	\begin{align}
 	\sum_{A,B\in\binom{V}{k}}&\P\bra{U_{10T}^\mathrm{EX}=A}\P\bra{W_{10T}^{\widetilde{\mathrm{EX}}}=B}\sum_{\substack{a\in A\\b\in B}}\frac1{k^2}\P\bra{A_{(10+K)T}^\mathrm{EX}\neq B_{(10+K)T}^{\widetilde{\mathrm{EX}}},\,F_{a,b},\,E^3_{a,b}}\nonumber \\
 	&\le \sum_{A,B\in\binom{V}{k}}\P\bra{U_{10T}^\mathrm{EX}=A}\P\bra{W_{10T}^{\widetilde{\mathrm{EX}}}=B}\sum_{\substack{a\in A\\b\in B}}\frac1{k^2}\,\frac{23}{25}\,\P\bra{\sigma_{a,b}(a^\ast)\neq a^\ast,\,|e_{a,b}|> 4}\nonumber\\
 	&\le \frac{2}{500}+\frac{23}{25|V|^2}\sum_{a,b\in V}\P\bra{\sigma_{a,b}(a^\ast)\neq a^\ast,\,|e_{a,b}|> 4},\label{eq:E3_bound}
 	\end{align}
 	where the last inequality is obtained in the same way as \eqref{eq:E2_bound}.
 	
 	Our fourth and final case to consider is $E^4_{a,b}$: on this event a simple case-by-case analysis (sketched in Appendix~\ref{Appendix3}) shows that as long as there are no other (already matched) particles on edge $e_{a,b}$ at time $\tau_{a,b}-$ (i.e. $|\hat I_{\tau_{a,b}-}(I_{10T}(A))\cap e_{a,b}| = 1$), there exists a bijection between permutations $\sigma_{a,b}$ and $\tilde\sigma_{a,b}$ such that $\tilde\sigma_{a,b}$ is a permutation with the same cycle structure as $\sigma_{a,b}$, and such that with probability at least $1/2$
 	 \[\sigma_{a,b}\big(\hat I_{\tau_{a,b}-}(I_{10T}(a))\big)=\tilde\sigma_{a,b}\big(\hat I_{\tau_{a,b}-}(\tilde I_{10T}(b))\big)\,.\] 
 	
 	That is, in this situation we are able to make the locations of all $k$ particles of $A_{(10+K)T}^\mathrm{EX}$ and $B_{(10+K)T}^{\widetilde{\mathrm{EX}}}$ agree with probability at least $1/2$:
 	\[ \P\bra{A_{(10+K)T}^\mathrm{EX}= B_{(10+K)T}^{\widetilde{\mathrm{EX}}},\,F_{a,b},\,E^4_{a,b}} \ge \frac12\,\P\bra{|\hat I_{\tau_{a,b}-}(I_{10T}(A))\cap e_{a,b}| = 1,\,F_{a,b},\,E^4_{a,b}} \,. \] 	
 	
 	We use a union bound to control the probability of the complement and write $c^\ast$ for $\hat I_{\tau_{a,b}-}(I_{10T}(c))$ and $b^\ast$ for $\hat I_{\tau_{a,b}-}(\tilde I_{10T}(b))$. We have 
 	\begin{align}
 	\sum_{A,B\in\binom{V}{k}}&\P\bra{U_{10T}^\mathrm{EX}=A}\P\bra{W_{10T}^{\widetilde{\mathrm{EX}}}=B}\sum_{\substack{a\in A\\b\in B}}\frac1{k^2}\P\bra{A_{(10+K)T}^\mathrm{EX}\neq B_{(10+K)T}^{\widetilde{\mathrm{EX}}},\,F_{a,b},\,E^4_{a,b}}\nonumber \\
 	&\le \sum_{A,B\in\binom{V}{k}}\P\bra{U_{10T}^\mathrm{EX}=A}\P\bra{W_{10T}^{\widetilde{\mathrm{EX}}}=B}\nonumber\\ &\qquad\qquad\cdot\sum_{\substack{a\in A\\b\in B}}\frac1{2k^2}\left(\P\bra{\,E^4_{a,b}} + \P\bra{|\hat I_{\tau_{a,b}-}(I_{10T}(A))\cap e_{a,b}| > 1,\,a^\ast\neq b^\ast,\,E^4_{a,b}} \right) \nonumber\\
 	&\le \sum_{A,B\in\binom{V}{k}}\P\bra{U_{10T}^\mathrm{EX}=A}\P\bra{W_{10T}^{\widetilde{\mathrm{EX}}}=B} \nonumber\\
 	&\qquad\qquad\cdot\sum_{\substack{a\in A\\b\in B}}\frac1{2k^2}\left(\P\bra{E^4_{a,b}} + 
 	\sum_{c\in A\setminus\{a\}}\P\bra{c^\ast\in e_{a,b},\,c^\ast\neq b^\ast,\,a^\ast\neq b^\ast,\,E^4_{a,b}}\right) \nonumber \\
 	&=\sum_{a,b\in V}\sum_{c\neq a}\frac1{2k^2}\left(\frac{\P\bra{E^4_{a,b}}}{k-1} + 
 	\P\bra{c^\ast\in e_{a,b},\,c^\ast\neq b^\ast,\,a^\ast\neq b^\ast,\,E^4_{a,b}}\right)\nonumber \\
 	&\qquad\qquad\cdot\sum_{\substack{A\supset\{a,c\}\\B\supset\{b\}}}\P\bra{U_{10T}^\mathrm{EX}=A}\P\bra{W_{10T}^{\widetilde{\mathrm{EX}}}=B}\,.\label{eq:nearly_there}
 	\end{align}
 	We now upper bound this using \eqref{eq:law_I_inverse} and \eqref{eq:tvint} (using the same method as in \eqref{eq:TV_fn_bound}). This gives the following upper bound for \eqref{eq:nearly_there}:
 	\begin{align*}
 	&\frac{2}{500}+\frac{k^2(k-1)}{|V|^2(|V|-1)}\sum_{a,b\in V}\sum_{c\neq a}\frac1{2k^2}\left(\frac{\P\bra{E^4_{a,b}}}{k-1} + 
 	\P\bra{c^\ast\in e_{a,b},\,c^\ast\neq b^\ast,\,a^\ast\neq b^\ast,\,E^4_{a,b}}\right)\\
 	&\le \frac{2}{500}+\frac{1}{2|V|^2}\sum_{a,b\in V}\P\bra{E^4_{a,b}} + 
 	\frac{k-1}{2|V|^2(|V|-1)}\sum_{a,b\in V}\sum_{c\neq a} 
 	\P\bra{c^\ast\in e_{a,b},\,c^\ast\neq b^\ast,\,a^\ast\neq b^\ast,\,E^4_{a,b}}\\
 	&\le \frac{2}{500}+\frac{1}{2|V|^2}\sum_{a,b\in V}\P\bra{E^4_{a,b}} + 
 	\frac{k-1}{|V|^2(|V|-1)}\sum_{a,b\in V}\P\bra{E^4_{a,b}}\,,
 	 	\end{align*}
 	since, on the event $E^4_{a,b}$, the size of edge $e_{a,b}$ is at most four and so on the event $\{a^\ast\neq b^\ast\}$ for any choice of $e_{a,b}$ there are only two possibilities for the value of $c$ (since $c^\ast\notin\{a^\ast,b^\ast\}\subset e_{a,b}\}$).
 	This gives
 	\begin{align}
 	\sum_{A,B\in\binom{V}{k}}&\P\bra{U_{10T}^\mathrm{EX}=A}\P\bra{W_{10T}^{\widetilde{\mathrm{EX}}}=B}\sum_{\substack{a\in A\\b\in B}}\frac1{k^2}\P\bra{A_{(10+K)T}^\mathrm{EX}\neq B_{(10+K)T}^{\widetilde{\mathrm{EX}}},\,F_{a,b},\,E^4_{a,b}}\nonumber \\
 	&\le \frac{2}{500}+\frac{1}{|V|^2}\left(\frac12+  
 	\frac{k-1}{|V|-1}\right) \sum_{a,b\in V}\P\bra{\sigma_{a,b}(a^*)\neq a^*,\,|e_{a,b}|\le 4}\,.\label{eq:E4_bound}
 	\end{align}
 	
 	\medskip
 	We now combine the bounds in \eqref{eq:tvUWfinal}, \eqref{eq:E1_bound}, \eqref{eq:E2_bound}, \eqref{eq:E3_bound} and \eqref{eq:E4_bound} to see that
 	\begin{align*}
 	\|\mathcal{L}[U_{(20+K)T}^\mathrm{EX}]-\mathcal{L}[W_{(20+K)T}^{\widetilde{\mathrm{EX}}}]\|\tv&\le \frac{7}{500}+\max_{a,b}\P\bra{E^1_{a,b}}+\frac1{|V|^2}\sum_{a,b}\P\bra{\sigma_{a,b}(a^*)=a^*}\\
 	&\phantom{\le}+\frac1{|V|^2}\max\left\{\frac{23}{25},\,\frac12 +\frac{k-1}{|V|-1}\right\}\sum_{a,b}\P\bra{\sigma_{a,b}(a^*)\neq a^*}.
 	\end{align*}
 	By assumption, $k\le |V|/2$ if $|V|<36$ and $k\in\{3,4\}$ if $|V|\ge 36$, and so 
 	\[ \max\left\{\frac{23}{25},\,\frac12 +\frac{k-1}{|V|-1}\right\} \le \frac{33}{34} \]
 	for all possible combinations of $k$ and $|V|$ being considered here. Combining this bound with that in Assumption \ref{assumpf}, we obtain:
 	\begin{align*}
 	 	\|\mathcal{L}[U_{(20+K)T}^\mathrm{EX}]-\mathcal{L}[W_{(20+K)T}^{\widetilde{\mathrm{EX}}}]\|\tv \le \frac{7}{500}+\max_{a,b}\P\bra{E^1_{a,b}}+\frac1{5}\left(1+4 \cdot\frac{33}{34}\right)\,.
 	\end{align*}
 	 	But since $\tau_{a,b}$ has the same distribution as $M^{\mathrm{RW}}(I_{10T}(a),\tilde I_{10T}(b))$, and $G$ is an easy hypergraph,
 	\begin{align*}
 	\max_{a,b}\P\bra{E^1_{a,b}}&=\max_{a,b}\P\bra{\tau_{a,b}>KT}\le \max_{a,b}\P\bra{M^\mathrm{RW}(a,b)>KT}\le \frac{1}{1000}
 	\end{align*} provided $K\ge 10^{10}$. Therefore,
 	\[
 	\|\mathcal{L}[U_{10^{11}T}^\mathrm{EX}]-\mathcal{L}[W_{10^{11}T}^{\widetilde{\mathrm{EX}}}]\|\tv\le \frac8{500}+\frac1{5}\left(1+4\cdot\frac{33}{34}\right)<\frac{497}{500}.
 	\]
 	Finally, by submultiplicativity of the function 
 	\[\bar d(t):= \max_{U,W\in \binom{V}{k}} \|\mathcal{L}[U_t^{\mathrm{EX}}] - \mathcal{L}[W_t^{\widetilde{\mathrm{EX}}}]\|\tv \]
 	(see e.g. Lemma 4.12 of \cite{Levin2008}), we deduce that
 	\[
 	\|\mathcal{L}[U_{10^{14}\log(1/\varepsilon)T}^\mathrm{EX}]-\mathcal{L}[W_{10^{14}\log(1/\varepsilon)T}^{\widetilde{\mathrm{EX}}}]\|\tv<\left(\frac{497}{500}\right)^{1000\log(1/\eps)}<\eps\,,
 	\]
 	and so the statement of Lemma~\ref{L:ktok-1} is proved upon taking $\kappa=10^{14}$.
 \end{proof}


%
 	
     \begin{proof}[Proof of Lemma \ref{L:4to2} for easy
      hypergraphs]We simply apply Lemma \ref{L:ktok-1} for the case $|V|\ge36$ twice, first with $k=4$ and then with $k=3$ (and take $\varepsilon=1/4$ both times). We deduce that 
      \[
      T_{\mathrm{EX}(4,f,G)}(1/4)\le\kappa^2(\log4)^2T_{\mathrm{EX}(2,f,G)}(1/4),
      \]
      and so it suffices to take $\lambda=\kappa^2(\log4)^2$.
\end{proof}

\subsection{From 4-particle exclusion to 2-particle exclusion:
  non-easy hypergraphs}
We begin with a result showing that for non-easy hypergraphs the
average meeting time for two independent random walkers is unlikely to
be quick. Intuitively, this follows from the following
observations. We know there exists a pair of vertices such that random
walkers started from these two states likely take a long time to
meet. If we look at where these two walkers are at time of order
$T_{\mathrm{RW}(f,G)}(1/4)$, they will be close to uniform. Hence,
starting random walkers from a uniform pair we see that they will
likely still take a long time to meet. The proofs of
Lemmas~\ref{L:averagemeet} and \ref{L:MRWxbound}, and of
Proposition~\ref{P:couplingIPwithRW} are (somewhat technical)
extensions of corresponding results of \cite{Oliveira2013a}, and can be found in 
\opt{noarxivVersion} 
{
\citep{Connor-Pymar-2016}. Throughout this section we let $T$ denote $T_{\mathrm{EX}(2,f,G)}(1/4)$.
}
\opt{arxivVersion} 
{
\hspace{-2.5mm}Appendix~\ref{Appendix}.
}
\begin{lemma}\label{L:averagemeet}For every non-easy hypergraph we have
  \[
  \sum_{{\bf u}\in V^2}\frac{\P\big[M^\mathrm{RW}({\bf u})\le20
    T\big]}{|V|^2}\le \frac1{1000}.
  \]
\end{lemma}

Given a $k$-tuple ${\bf z}\in (V)_{k}$, we once again write ${\bf O}({\bf
  z}):=\{{\bf z}(1),\ldots,{\bf z}(k)\}$ for the (unordered) set of
coordinates of ${\bf z}$. For ${\bf x}\in V^4$, let ${\bf x}^{\mathrm{RW}}_t$ be a realisation
of RW$(4,f,G)$ with ${\bf x}^{\mathrm{RW}}_0={\bf x}$. Denote by
$\Lambda^1,\Lambda^2,\Lambda^3,\Lambda^4$ the (independent) Poisson processes used to generate the
edge-ringing times for the four random walkers, and let $\{e^1_n\}_{n\in\mathbb{N}},
\{e^2_n\}_{n\in\mathbb{N}},  \{e^3_n\}_{n\in\mathbb{N}},  \{e^4_n\}_{n\in\mathbb{N}}$ be the four sequences of edge-choices (all as in Section \ref{S:graphical}).

We now define $\bar M^{\mathrm{RW}}({\bf O}({\bf x}))$ to be the first time
\emph{any two} of ${\bf x}^{\mathrm{RW}}_t(1)$, ${\bf
  x}^{\mathrm{RW}}_t(2)$, ${\bf x}^{\mathrm{RW}}_t(3)$, ${\bf
  x}^{\mathrm{RW}}_t(4)$ first arrive onto the same edge which then
rings for one of them. Formally,
\begin{align*}
\bar M^\mathrm{RW}({\bf O}({\bf x})):=\inf\Big\{&t\ge0:\,\exists\,1\le i<j\le
4,\,\, e\in\{e^i_{\Lambda^i[0,t]},e^j_{\Lambda^j[0,t]}\}\\&\text{ with }{\bf
  x}^\mathrm{RW}_{t-}(i),{\bf x}^\mathrm{RW}_{t-}(j)\in
e\Big\}.
\end{align*}

\begin{lemma}\label{L:MRWxbound}Let ${\bf x}\in (V)_4$. Then for any
  $\eps\in(0,1)$,\[ \P\bra{\bar M^\mathrm{RW}({\bf O}({\bf
      x})_{20T}^\mathrm{EX})\le 20T}\le
  12(\eps+\eps^{-1}2^{-20})+25(1+\eps)\sum_{{\bf u}\in
    V^2}\frac{\P\bra{M^\mathrm{RW}({\bf u})\le20T}}{|V|^2}\]
\end{lemma}

    Next, we provide a bound which relates the total-variation
    distance between two 4-particle exclusion processes to
    the probability that any two of four independent walkers have
    `met'.
    \begin{prop}\label{P:couplingIPwithRW}
      For any ${\bf x}\in(V)_4$ and $s\ge0$:
      \[
      \|\cL[{\bf O}({\bf x}_s^\mathrm{RW})]-\cL[{\bf O}({\bf x})_s^\mathrm{EX}]\|\tv\le
      \P\bra{\bar M^\mathrm{RW}({\bf O}({\bf x}))\le s}.
      \]
    \end{prop}
    
    \begin{lemma}\label{L:tvtomeet}For every non-easy hypergraph $G$
       and any two realisations of
      \textnormal{EX}$(4,f,G)$, denoted $\{A_t^\mathrm{EX}\}$
      and $\{B_t^\mathrm{EX}\}$, we have
      \[
      \|\cL[A_{40T}^{\mathrm{EX}}]-\cL[B_{40
        T}^{\mathrm{EX}}]\|\tv\le\P\bra{\bar M^{\mathrm{RW}}(A_{20T}^{\mathrm{EX}})\le 20T}+\P\bra{\bar
        M^{\mathrm{RW}}(B_{20T}^{\mathrm{EX}})\le 20T}+2^{-18}.
      \]
    \end{lemma}
    \begin{proof}
      By Proposition \ref{P:couplingIPwithRW} and the triangle
      inequality for total-variation, for any ${\bf u},{\bf v}\in
      (V)_4$,
      \begin{align}\label{eq:uvTV}
        \|\cL[{\bf O}({\bf u})_{20T}^{\mathrm{EX}}]-\cL[{\bf O}({\bf
          v})_{20T}^{\mathrm{EX}}]\|\tv&\le \P\bra{\bar
          M^{\mathrm{RW}}({\bf O}({\bf u}))\le 20T}+\P\bra{\bar
          M^{\mathrm{RW}}({\bf O}({\bf v}))\le 20T}\notag\\&\phantom{=}+\|\cL[{\bf O}({\bf
          u}_{20T}^{\mathrm{RW}})]-\cL[{\bf O}({\bf
          v}_{20T}^{\mathrm{RW}})]\|\tv.
      \end{align}

      An identical argument to that used for equation \eqref{eq:tvUW} tells us that
      \[
      \|\cL[A_{40T}^\mathrm{EX}]-\cL[B_{40T}^\mathrm{EX}]\|\tv \le
      \E\left[\|\mathcal{L}[A_{40T}^\mathrm{EX}\,|\,A_{20T}^\mathrm{EX}]-\mathcal{L}[B_{40T}^\mathrm{EX}\,|\,B_{20T}^\mathrm{EX}]\|\tv\right] \,.
      \]
      Applying the inequality in \eqref{eq:uvTV}, with any ${\bf u},{\bf v}$ satisfying
      ${\bf O}({\bf u})=A_{20T}^\mathrm{EX}$ and ${\bf O}({\bf v})=B_{20T}^\mathrm{EX}$, gives
      \begin{align*}
        \|\cL[A_{40T}^\mathrm{EX}]-\cL[B_{40T}^\mathrm{EX}]\|\tv&\le \E\big[\P\bra{\bar
          M^{\mathrm{RW}}(A_{20T}^\mathrm{EX})\le 20T|\,A_{20T}^\mathrm{EX}}\big]\\&\phantom{=}+\E\big[\P\bra{\bar M^{\mathrm{RW}}(B_{20T}^\mathrm{EX})\le 20T|\,B_{20T}^\mathrm{EX}}\big]\\&\phantom{=}+\sup_{{\bf
            u},{\bf v}\in(V)_4}\|\cL[{\bf
          u}_{20T}^{\mathrm{RW}}]-\cL[{\bf
          v}_{20T}^{\mathrm{RW}}]\|\tv\,.
      \end{align*}
      Using Proposition \ref{P:2and4to1indep} and the contraction principle for the third term on
      the right-hand side gives the desired result.
          \end{proof}

    We are now ready to prove the main result of this subsection.

    \begin{proof}[Proof of Lemma \ref{L:4to2} for non-easy
      hypergraphs] We in fact show that for any two realisations of
      \textnormal{EX}$(4,f,G)$, denoted $\{A_t^\mathrm{EX}\}$
      and $\{B_t^\mathrm{EX}\}$, we have
      \[
      \|\cL[A_{40 T}^{\mathrm{EX}}]-\cL[B_{40
        T}^{\mathrm{EX}}]\|\tv\le1/4.
      \]
      Combining Lemmas \ref{L:MRWxbound} and \ref{L:tvtomeet} we have
      that for every $\eps\in(0,1)$,
      \[
      \|\cL[A_{40 T}^{\mathrm{EX}}]-\cL[B_{40
        T}^{\mathrm{EX}}]\|\tv\le
      24(\eps+\eps^{-1}2^{-20})+50(1+\eps)\sum_{{\bf u}\in
        V^2}\frac{\P\bra{M^\mathrm{RW}({\bf u})\le20T}}{|V|^2} +
      2^{-18}.
      \]
      Now by Lemma \ref{L:averagemeet}, this becomes
      \[
      \|\cL[A_{40 T}^{\mathrm{EX}}]-\cL[B_{40
        T}^{\mathrm{EX}}]\|\tv\le
      24(\eps+\eps^{-1}2^{-20})+\frac{50(1+\eps)}{1000}+ 2^{-18}.
      \]Setting $\eps=10^{-3}$ completes the proof.
     
    \end{proof}

\section{From $k$-particle exclusion to 2-particle exclusion for small graphs}\label{S:smallV}

We now prove Lemma \ref{L:smallV}. We begin by showing that any hypergraph $G$ with $|V|<36$ satisfies
\begin{align}\label{smallVeasy}
 \sup_{{\bf y}\in V^2}\P\bra{M^{\mathrm{RW}}({\bf
    y})>10^{10}T_{\mathrm{RW}(f,G)}(1/4)}\le 1/1000,
\end{align}
i.e. the hypergraph $G$ is \emph{easy}.
Indeed, by Proposition \ref{P:AldousFill}, for any $t\ge 2T_{\mathrm{RW}(f,G)}(\varepsilon)$,
\[
\sup_{{\bf y}\in V^2}\P\bra{M^{\mathrm{RW}}({\bf
    y})<t}\ge \frac{(1-2\varepsilon)^2}{|V|}\ge \frac{(1-2\varepsilon)^2}{36},
\]
and so 
\[
\sup_{{\bf y}\in V^2}\P\bra{M^{\mathrm{RW}}({\bf
    y})>2000 T_{\mathrm{RW}(f,G)}(1/4)}\le 1/1000,
\]
which certainly implies \eqref{smallVeasy}. Since $G$ is easy, we can apply Lemma \ref{L:ktok-1} multiple times to deduce that 
\[
T_{\mathrm{EX}(k,f,G)}(\varepsilon)\le \kappa^{k-2}(\log(1/4))^{k-3}\log(1/\varepsilon)T_{\mathrm{EX}(2,f,G)}(1/4).
\]
However, since $|V|<36$ and $k\le |V|/2$ the statement of the proof is complete taking $C=\kappa^{15}(\log(1/4))^{14}$.
\section{The chameleon process}\label{S:cham}
Our aim in this section is to construct a continuous-time Markov process which satisfies the properties of $(M_t)_{t\ge0}$ outlined in Lemma \ref{L:existence}. We will call this process the \emph{chameleon process}. In Section \ref{S:chamhasprops} we will prove Lemma \ref{L:existence} by demonstrating that the chameleon process does indeed have the desired properties.

The chameleon process was originally constructed (in a different form but to serve a similar purpose) by
\citet{Morris2006a}, and then adapted by \citet{Oliveira2013a} to
analyse the mixing time of the $k$-particle interchange process on a graph (as
opposed to on a hypergraph, as we consider here). It is built on top of an underlying
interchange process, with the aim of helping to describe the
distribution of the location of the $k$th particle in this process,
conditional on the locations of the $k-1$ other particles.

Unlike in a $k$-particle interchange process which always has $k$ particles, the chameleon process has $|V|$ particles (one at each vertex), although not all
particles are distinguishable from each other. In addition, each
particle has an associated \emph{colour}: one of black, red, pink and
white (which correspond to the processes ${\bf{z}}_t^C,\,R_t,\,P_t,\,W_t$ respectively, appearing in the statement of Lemma~\ref{L:existence}). 
The movement of particles in the chameleon process follows that
of the underlying interchange process in the sense that the locations
of particles in both processes are updated using the same functions
$I$ as described in the graphical construction of
Section~\ref{S:graphical}. At some of the updates of the underlying interchange process we will colour some of the red and white particles pink (precisely when this happens is rather involved and is the subject of Section~\ref{S:chameleon_construction}). To provide some insight into when these \emph{pinkening} events occur, consider the chameleon
process of \citep{Oliveira2013a}: here, if the vertices at the endpoints of a ringing edge
are occupied by a red and a white particle then both of these particles are
recoloured pink. In the lazy version of the interchange process on a
graph (in which nothing happens with probability 1/2 when an edge
rings), when an edge rings with endpoints occupied by a red and a white
particle, with probability 1/2 they switch places and with probability
1/2 they do not move. Colouring both particles pink (which should be viewed as half red, half white) encodes the fact that
at either vertex just after the edge rings we may have a red
particle or a white particle, and these are equally likely.

We wish to use this notion of pinkening to encode similar
events in the interchange process on hypergraphs, but the situation
here is quite different since more than two particles are moved when an
edge rings, and the way in which they move depends on the permutation
chosen. As a result, describing precisely when these pinkenings occur
for our version of the chameleon process is rather complicated, but
the underlying motivation can be explained relatively simply. As in Oliveira's argument we will use pink particles as a way of tracking particles which are either red or white (equally likely). Whereas Oliveira could make use of laziness to split the conditional distribution among two sites, when dealing with hypergraphs we have to use a new idea of a ``twin'' permutation. Suppose that an edge $e$ rings and a permutation $\sigma$ is chosen to move the particles on that edge. To decide which particles to pinken, we construct a twin permutation $\tilde\sigma$ with the property that the trajectories of all black particles in the edge are identical under both permutations (a required property -- see part~\ref{item:z_same_path} of Lemma~\ref{L:existence}), and such that, viewed marginally, the distribution of $\tilde\sigma$ agrees with that of $\sigma$. We then look for vertices $v$ such that under $\sigma$ a red particle is moved to $v$ and under $\tilde\sigma$ a white particle is moved to $v$; a certain subset of these particles will be pinkened. The simplest example to consider is that of an edge of size 3 which contains one red, one white and one black particle, and for which $f_e$ is constant on $S_e$. In this case, it is straightforward to construct a twin permutation with these required properties, and the construction is sketched in Figure~\ref{fig:section6}.
\begin{figure}[!ht]
	\centering
	\includegraphics[scale=.9]{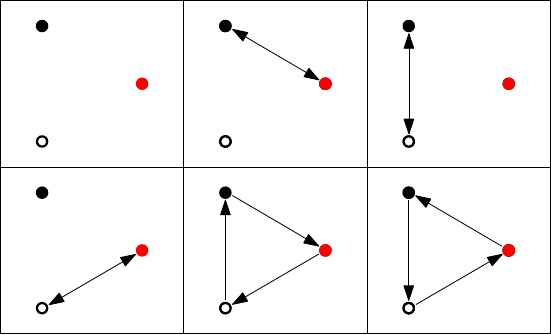}
	\caption{Consider an edge of size 3, containing one red, one white and one black particle, and for which $f_e$ is constant on $S_e$. The six possible permutations are sketched here: in this example the twin of any given permutation $\sigma$ could be taken to be the permutation immediately above/below $\sigma$. Note that in each case the  black~\protect\includegraphics[scale=.9]{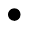} particle follows the same trajectory under both $\sigma$ and its twin; moreover, if $\sigma$ moves a red/white particle to a vertex $v$, then its twin moves a white/red particle to $v$. In this simple example we could therefore pinken the red and white particles, no matter which $\sigma$ is chosen.}\label{fig:section6}
\end{figure}

Although this example demonstrates one possibility for generating twin permutations with our desired properties, this is \emph{very} particular to the situation in which $f_e$ is constant on $S_e$ -- a much stronger condition than we are imposing in Assumption~\ref{assumpf}. In general, we shall make use of the fact that $f_e$ is constant on conjugacy classes to construct a twin permutation $\tilde\sigma$ with the  \emph{same cycle structure} as $\sigma$. (Note that the twin permutations constructed in Figure~\ref{fig:section6} do not have the same cycle structure as $\sigma$, and so we shall need to come up with an alternative method of pinkening, even when considering edges of size 3.) Figure \ref{F:image1} gives an indication of how $\tilde\sigma$ will be produced from knowledge of $\sigma$ and the particle colours in the case of $\sigma$ being a single cycle: by modifying the trajectories of four particular particles we are able to ensure that not only does $\tilde\sigma$ have the same cycle structure as $\sigma$, but that the trajectories of all black particles in the edge are identical under both permutations. It is for this reason (i.e. needing to know the colours of four particular particles) that we are able to relate the mixing time of $k$ particles to that of just four particles in Lemma \ref{L:kto4}.

\begin{figure}[!ht]
  \centering
  \includegraphics[scale=.9]{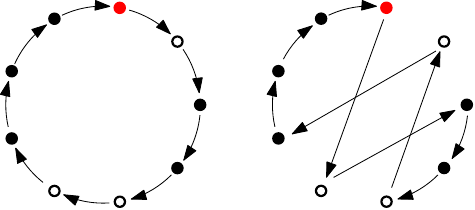}
  \caption{If $\sigma$ is the cycle given by the arrows in the image
    on the left, then one possible candidate for $\tilde\sigma$ is the
    cycle given by the arrows in the image on the
    right. Note that all black~\protect\includegraphics[scale=.9]{black_ball.pdf} particles follow the same trajectories
    under both cycles, and that the colour of $\tilde\sigma(v)$ agrees
    with that of $\sigma(v)$ for all $v$.}\label{F:image1}
\end{figure}

The chameleon process also updates at additional times (compared to its
corresponding interchange). We refer to these additional updates as
\emph{depinkings}, as at these times we get the opportunity to
collectively recolour all pink particles in the system either red or
white. As in \citep{Oliveira2013a}, we will only perform a depinking
once there are a large number of pink particles (compared to the
number of red and white) in the system.

\subsection{Twin permutations}\label{s:new_from_old} 
The first step towards constructing the chameleon process is describing how to generate the \emph{twin} permutation $\tilde\sigma$ from $\sigma$, which is the subject of this section. We begin with the case of $\sigma$ being either a product of disjoint transpositions or a single cycle (of size at least 3), and then describe how to construct $\tilde\sigma$ for a general permutation $\sigma$. We conclude the section by describing an algorithm to generate a certain set $A$ which plays a crucial role in the construction of the chameleon process.

\subsubsection{Composition of transpositions, and cycles of size at least 3}\label{S:transpositions}
We begin with some notation: for $d\in\mathbb{N}$ let $[d]=\{1,2,\dots,d\}$, and let $d' =
\lfloor d/4\rfloor$ (the floor of $d/4$). For convenience we also let $[0] = \{0\}$. 

For $d\in2\mathbb{N}$ we let $T_d$ be the set of products of disjoint transpositions: 
\[
T_d:=\left\{\prod_{i=1}^{d/2} (a_{2i-1}\,\,a_{2i}):\,1\le a_i\le d\text{ for all } 1\le i\le d,\,a_i\neq a_j\text{ for all } i\neq j\right\}.
\]
For $\sigma\in T_d$ we define an ordering, denoted $\prec$, of the transpositions in $\sigma$ as follows: $(a_i\,\,a_j)\prec (a_k\,\,a_\ell)$ if and only if $(a_i\wedge a_j)<(a_k\wedge a_\ell)$. Without loss of generality we shall always suppose that any $\sigma\in T_d$ is written such that
\[
(a_1\,\,a_2)\prec(a_3\,\,a_4)\prec\cdots\prec(a_{d-1}\,\,a_d),
\]
and $a_{2i-1}<a_{2i}$ for all $1\le i\le d/2$.

Given a set $A\subseteq[d']$ and a permutation $\sigma\in T_d$, we define the permutation $\rhofn_A(\sigma)$ to be the result of multiplying $\sigma$ (on the right) by a particular set of disjoint transpositions, as follows:
\begin{equation}\label{e:beta_trans}
\rhofn_A(\sigma) = \sigma \prod_{\substack{i\in A:\\a_{4i-1}<a_{4i-2}}}(a_{4i-3}\,\,a_{4i-1})(a_{4i-2}\,\,a_{4i})  \,. 
\end{equation}

The permutation $\rhofn_A(\sigma)$ satisfies some nice properties, which we put together in the following Lemma. The proofs are straightforward, but it is worth emphasising that part 2 of Lemma~\ref{L:hat_rho_A_props} (that $\rhofn_A$ is an involution) holds precisely because in \eqref{e:beta_trans} we only multiply by transpositions for which $a_{4i-1}<a_{4i-2}$.
\begin{lemma}\label{L:hat_rho_A_props}
	For any $\sigma\in T_d$ and set $A\subseteq[d']$:
	\begin{enumerate}
		\item $\rhofn_A(\sigma)\in T_d$;
		\item $\rhofn_A(\rhofn_A(\sigma))=\sigma$;
		\item for all $x\in[d]$, $\rhofn_A(\sigma)(x)=\sigma(x)$ unless $x\in \{a_{4i-3},a_{4i-2},a_{4i-1},a_{4i}\}$ for some $i\in A$ with $a_{4i-1}<a_{4i-2}$.
	\end{enumerate}
\end{lemma}

\medskip
We now move onto cycles of size at least 3. For $d\in\mathbb{N}$ we denote by $\cC_d$ the conjugacy class of $\cS_d$ (the symmetric group on $[d]$) consisting of cycles of length $d$. 
For a cycle $\sigma\in\cC_d$ we may write
\[ \sigma=(\sigma^0(1) \,\, \sigma^1(1) \,\, \sigma^2(1)\,\, \dots \,\,
\sigma^{d-1}(1)) \] where for $m\in\mathbb{N}$ we write $\sigma^m$ for
the composition of $m$ copies of $\sigma$ (and where we may sometimes
write $\sigma^d\equiv\sigma^0$ for the identity permutation).

For $d\ge4$ and $i=1,\dots,d'$ define the function $\rhofn_i:[d]\to[d]$ by
\begin{equation}\label{e:rho_i_defn} \rhofn_i(j) = \begin{cases}
2d'+2i-1 & j=2i-1 \\
2i-1 & j=2d'+2i-1 \\
j & \text{otherwise.}
\end{cases}
\end{equation}
For $d=3$ we similarly define the function $\rhofn_0:[d]\to[d]$ by
\begin{equation}\label{e:rho_i_defn2} 
\rhofn_0(j) = \begin{cases}
2 & j=1 \\
1 & j=2 \\
3 & j=3\,.
\end{cases}
\end{equation}

For $d\ge4$ and a cycle $\sigma\in \cC_d$, for each $1\le i\le d'$, we define $\rhofn_i(\sigma)\in \cC_d$ to be the
permutation satisfying
\begin{equation}\label{e:beta_fn}
\rhofn_i(\sigma)^j(1) = \sigma^{\rhofn_i(j)}(1) \,, \quad
j=1,\dots,d.
\end{equation}
This permutation is clearly a cycle, and may be
written as
\begin{align*} \rhofn_i(\sigma) = (&\sigma^0(1)\,\, \cdots \,\,
\sigma^{2i-2}(1)\,\,\sigma^{2d'+2i-1}(1)\,\,\sigma^{2i}(1)\\&\cdots
\,\,
\sigma^{2d'+2i-2}(1)\,\,\sigma^{2i-1}(1)\,\,\sigma^{2d'+2i}(1)\,\,\cdots
\,\, \sigma^{d-1}(1)). \end{align*}
For a cycle $\sigma\in\cC_3$, we define $\rhofn_0(\sigma)$ using the same formula as in \eqref{e:beta_fn}, yielding the 3-cycle 
$\rhofn_0(\sigma) = \sigma^2 = (1\,\,\sigma^2(1)\,\,\sigma(1))$.

\begin{rmk}\label{r:function_or_perm}
	Note that for the case $d\ge 4$, $\rhofn_i(\sigma)$ may be obtained from $\sigma$ by
	multiplication by the product of two disjoint transpositions:
	\[
	\rhofn_i(\sigma) = \sigma \,(\sigma^{2i-2}(1)\,\,\sigma^{2d'+2i-2}(1))(\sigma^{2i-1}(1)\,\,\sigma^{2d'+2i-1}(1))
	\,.
	\]
\end{rmk}

\begin{lemma}\label{L:rho_function_props}
	For $d\ge3$ and any $i,j\in[d']$:
	\begin{enumerate}
		\item The function $\rhofn_i$ is self-inverse;
		\item functions $\rhofn_i$ and $\rhofn_j$ commute for $i\neq j$.
	\end{enumerate}
\end{lemma}

\begin{proof}
	Part 1 follows directly from the definition of $\rhofn_i$. Part 2 only applies when $d\ge 4$, and 
	follows from the observation that the transpositions in
	Remark~\ref{r:function_or_perm} corresponding to $\rhofn_i$ and
	$\rhofn_j$ commute for $i\neq j$.
\end{proof}

\begin{defn}\label{d:rho_algorithm}
	Given a set $A\subseteq[d']$ and a cycle $\sigma\in\cC_d$, we define
	$\rhofn_A:[d]\to[d]$ to be the composition of the functions
	$\{\rhofn_i\,:\, i\in A\}$ appearing in \eqref{e:rho_i_defn} and \eqref{e:rho_i_defn2}. (Thanks
	to the second statement of Lemma~\ref{L:rho_function_props} this function is
	well-defined.) If $A=\varnothing$ then we set
	$\rhofn_A$ to be the identity function. \\
	This in turn defines a cycle $\rhofn_A(\sigma)\in\cC_d$ satisfying
	\begin{equation}\label{e:rho_A_defn}
	\rhofn_A(\sigma)^j(1) = \sigma^{\rhofn_A(j)}(1) \,, \quad
	j=1,\dots,d. 
	\end{equation}
\end{defn}

When $\sigma$ is a $d$-cycle the permutation $\rhofn_A(\sigma)$ satisfies analogous properties to those already observed to hold (in Lemma~\ref{L:hat_rho_A_props}) for $\rhofn_A(\sigma)$ when $\sigma\in T_d$. The proofs all follow simply from the
definition of $\rhofn_i$ and Lemma \ref{L:rho_function_props}.  For each $i\in[d']$, write
\begin{equation}\label{e:H_i-defn}
H_i =\begin{cases}
\{2i-2,2i-1,2d'+2i-2,2d'+2i-1\} & \quad\text{ if $d'\ge 1$ (i.e. $d\ge 4$)}\,, \\
\{1,2,3\} & \quad \text{ if $d'=0$ (i.e. $d= 3$)\,,}
\end{cases}
\end{equation} 
and for $A\subseteq[d']$
let $H_A = \bigcup_{i\in A}H_i$.
\begin{lemma}\label{L:rho_A_props}
	For any $d\ge3$, $\sigma\in\cC_d$ and set $A\subseteq[d']$: 
	\begin{enumerate}
		\item $\rhofn_A(\sigma)\in\cC_d$;
		\item\label{item:self_inverse} $\rhofn_A(\rhofn_A(\sigma)) = \sigma$;
		\item for all $x\in[d]$, $\rhofn_A(\sigma)(x) = \sigma(x)$ unless
		$x=\sigma^j(1)$ for some $j\in H_A$.
	\end{enumerate}
\end{lemma}

\medskip
So far we have defined a method for producing a permutation $\rhofn_A(\sigma)$ in the event that $\sigma$ is either a product of disjoint transpositions or a cycle of length at least 3. Now consider what happens when we apply the function $\rhofn_A$ to a permutation chosen uniformly from either $\cC_d$ or $T_d$ (for some $d$). Clearly, for any set $A$
chosen independently of $\sigma$, the resulting permutation $\rhofn_A(\sigma)$
will be uniformly distributed on the same conjugacy class as $\sigma$. Most importantly, this
remains true even when $A$ is allowed to depend upon $\sigma$, as long
as a certain condition is met, as explained in the following Lemma. We
denote by $\cP^\Omega$ the power set of a set $\Omega$.

\begin{lemma}\label{L:algorithm_random_input_SBC}
	Let $G_d$ denote either of the conjugacy classes $\cC_d$ ($d\ge 3$) or $T_d$ ($d\in2\mathbb{N}$).
	Suppose that $A:G_d\to \cP^{[d']}$ satisfies for all $\sigma\in G_d$,
	\begin{align}\label{eq:Aprop}
	A(\rhofn_{A(\sigma)}(\sigma))=A(\sigma)\,,
	\end{align}
	and that $\sigma$ is chosen uniformly from $G_d$. Then
	$\rhofn_{A(\sigma)}(\sigma)$ is also uniform on $G_d$. Moreover, if
	we average over the input permutation $\sigma$, then the output
	$\rhofn_{A(\sigma)}(\sigma)$ is independent of the choice of $A$.
\end{lemma}

\begin{proof}
	Given the permutation $\sigma$, let $\tilde\sigma =
	\rhofn_{A(\sigma)}(\sigma)$. The assumption on $A$ says that
	$A(\tilde\sigma) = A(\sigma)$. 
	Since $\rhofn_A$ is an involution (Lemmas~\ref{L:hat_rho_A_props} and \ref{L:rho_A_props}) it follows that
	\[ \rhofn_{A(\tilde\sigma)}(\tilde\sigma) = \rhofn_{A(\sigma)}(\tilde\sigma) =
	\rhofn_{A(\sigma)}(\rhofn_{A(\sigma)}(\sigma)) = \sigma\,. \]
	Thus the function $\rhofn_{A(\cdot)}(\cdot)$ is self-inverse.
\end{proof}

Although Lemma~\ref{L:algorithm_random_input_SBC} is relatively
simple, its importance should be emphasised at this point. We shall
make use of the function $\rhofn_A$ to generate the random permutations
$\tilde\sigma$ used in the construction of the chameleon process, and in doing so the input $A$ will
depend on the state of the chameleon process. 
The second part of Lemma~\ref{L:algorithm_random_input_SBC} will be
used to guarantee that the permutation
$\tilde\sigma=\rhofn_{A(\sigma)}(\sigma)$ is independent of $A$. (The
permutations $\tilde\sigma$ will be used to generate an interchange
process $\tilde{\bf x}^\mathrm{IP}$, and so it will be crucial that
these do not depend on the state of the process.)


\subsubsection{General permutations}\label{SSS:generalperm}
By combining the ideas from the previous two sections we can now describe the algorithm for the construction of the twin permutation $\tilde\sigma$ (which will be given by $\rhofn_{A(\sigma)}(\sigma)$ for some function $\rhofn_{A(\cdot)}(\cdot)$ to be defined) when $\sigma\in\cS_n$ is a general permutation. The first step is to decompose the input permutation $\sigma$ into its canonical cyclic decomposition form. Indeed, except for transpositions, the function $\rhofn_{A(\cdot)}(\cdot)$ will act independently on each
cycle in a given permutation's decomposition.

Suppose $\sigma$ has canonical cyclic decomposition form (where we omit fixed points):
\begin{align}\label{eq:decomp}
\sigma=\rho_0\circ\rho_1\circ\cdots\circ\rho_K,
\end{align}
where $K$ denotes the number of cycles in $\sigma$ of size at least 3, and $\rho_0$ is a (possibly empty) product of disjoint transpositions. 

For $i=0,1,\dots,K$ we write $m_i$ for the minimal element of $\rho_i$, and write $d_i$ for the size of the non-trivial orbit of $\rho_i$. (For example, if $\sigma = (1\,4)(2\,9)(3\,7\,6\,8\,5)\in\cS_9$ then $K=1$, $m_0 = 1$, $m_1=3$, $d_0 = 4$ and $d_1 = 5$.) Given the elements of the non-trivial orbit of $\rho_i$, there is an obvious natural bijection between permutations of those elements and permutations of the set $[d_i]=\{1,\dots,d_i\}$, in which the minimal element $m_i$ is mapped to 1. Rather than writing out this correspondence in detail, in order to ease notation in what follows we shall simply consider $\rho_i$ to be a member of the set $\cC_{d_i}$ etc, even though the set of elements belonging to $\rho_i$ will not in general be $\{1,\dots,d_i\}$.

With this understanding in mind, suppose that $A(\sigma)$ is a vector of the form

\begin{equation}\label{E:A_sigma}
A(\sigma) = (A_0(\rho_0),A_1(\rho_1),\dots,
A_K(\rho_K))\,,
\end{equation}
where $A_0:T_{d_0}\to\cP^{[d_0']}$ and $A_i:\cC_{d_i}\to \cP^{[d_i']}$ for $i=1,\dots,K$. Then we can easily
extend the idea of our functions $\rhofn_A$ to apply to general permutations.
\begin{defn}\label{D:rho_A_perms}
	Let $\sigma\in\cS_n$ be a permutation with cyclic decomposition \eqref{eq:decomp}, and assume that $A$ is a function on $\cS_n$ satisfying
	\eqref{E:A_sigma}. Then we define $\tilde\rhofn_{A(\sigma)}(\sigma)$ to be
	the composition of the permutations obtained by applying the functions
	$\rhofn_{A_i(\rho_i)}$ separately to each $\rho_i$:
	\[ \tilde\rhofn_{A(\sigma)}(\sigma) = \prod_{i=0}^K \rhofn_{A_i(\rho_i)}(\rho_i)\,, \]
	where $\rhofn_{A_i(\rho_i)}(\rho_i)$ are as defined in Section~\ref{S:transpositions} (but with $m_i$ replacing the element 1, as already explained).
\end{defn}

Definition~\ref{D:rho_A_perms} says that $\tilde\rhofn_{A(\sigma)}(\sigma)$ is obtained from $\sigma$ by
modifying each of its cycles of size at least 3, and the set of disjoint transpositions, independently using functions $\rhofn_{A_i(\cdot)}(\cdot)$ with which we are already familiar. We therefore have the following corollary to Lemmas~\ref{L:hat_rho_A_props} and~\ref{L:rho_A_props}.

\begin{cor}\label{C:tilde_rho_A_props}
	For any $\sigma\in S_n$ and function $A$ on $\cS_n$ satisfying
	\eqref{E:A_sigma}:
	\begin{enumerate}
		\item $\tilde\rhofn_{A(\sigma)}(\sigma)$ belongs to the same conjugacy class as $\sigma$;
		\item $\tilde\rhofn_{A(\sigma)}(\tilde\rhofn_{A(\sigma)}(\sigma))=\sigma$;
		\item for all $x\in[n]$, $\tilde\rhofn_{A(\sigma)}(x)=\sigma(x)$ unless $x\in \{a_{4i-3},a_{4i-2},a_{4i-1},a_{4i}\}$ for some $i\in A_0(\rho_0)$ with $a_{4i-1}<a_{4i-2}$, or $x=\sigma^j(m_i)$ for some $j\in \cup_{i=1}^K H_{A_i(\rho_i)}$.
	\end{enumerate}
\end{cor}

Furthermore, note that if we choose a random permutation
$\sigma\in\cS_n$ according to a law $f$ which is constant
on conjugacy classes, then given the sizes of the cycles in the
decomposition of $\sigma$, the elements of $[n]$ belonging to each cycle are (marginally)
uniform. We can therefore also obtain a corollary to
Lemma~\ref{L:algorithm_random_input_SBC}:
\begin{cor}\label{C:algorithm_random_input}
	Suppose that $A$ is a function on $\cS_n$ satisfying
	\eqref{E:A_sigma}, and that for all $\sigma\in \cS_n$ with cyclic decomposition \eqref{eq:decomp} and each
	$i=0,1,\dots,K$,
	\[
	A_i(\rhofn_{A_i(\rho_i)}(\rho_i))=A_i(\rho_i)\,.
	\]
	If $\sigma$ is chosen according to law $f$
	on $\cS_n$ which is constant on conjugacy classes then $\tilde \rhofn_{A(\sigma)}(\sigma)$ also has law $f$ on $\cS_n$. Moreover, if we average over the input permutation $\sigma$, then the output $\tilde\rhofn_{A(\sigma)}(\sigma)$ is independent of the choice of $A$.
\end{cor}

\subsubsection{Choosing the set $A$}
We have described in Section \ref{SSS:generalperm} how to generate the twin permutation $\tilde\sigma=\tilde \rhofn_{A(\sigma)}(\sigma)$ from a
permutation $\sigma$ such that it has the same law as $\sigma$. We now
detail our method for choosing the vector $A(\sigma)$ appearing in the
definition of $\tilde\rhofn_A$, in such a way that the conditions of
Corollary~\ref{C:algorithm_random_input} are satisfied; an illustrative example can be found in  Figures~\ref{fig:constructing_A} and~\ref{fig:constructing_A_tilde_sigma}.
Our choice of $A$ will depend not only on $\sigma$
but also on particular subsets of vertices in the edge under
consideration. (Later on these subsets will be specified in the
chameleon process, but for now we keep them as general subsets.)
Indeed, given an edge $e\in E$ and a permutation $\sigma\in \cS_e$, the function $A$ is of the form $A(R,W,\sigma)$, where $R$ and $W$ are two disjoint subsets of $V$.

Recall the definition of the set $H_i$ in \eqref{e:H_i-defn}, and the canonical cyclic decomposition of $\sigma$ from \eqref{eq:decomp} in which $K$ denotes the number of cycles of size at least 3 in $\sigma$. For a set of integers $B$ let us write $\sigma^B(x) = \{\sigma^i(x)\,:\,i\in B\}$. Then for each $1\le i\le K$, we define
\begin{equation}\label{E:A_i_defn2}
A_i(R,W,\rho_i)=\begin{cases}
\Big\{j \in[d_i']\,:\, \rho_i^{H_j}(m_i) \in \left\{\{r_1,w_1,w_2\},\{w_1,r_1,r_2\}\right\} &\\ 
\qquad\qquad\qquad\qquad\text{for some $r_1,r_2\in R,\,w_1,w_2\in W$}\Big\} & \text{ if }d_i = 3\\
\Big\{j \in[d_i']\,:\, \rho_i^{H_j}(m_i) \in \left\{\{r_1,w_1,w_2,w_3\},\{w_1,r_1,r_2,r_3\}\right\} &\\ \qquad\qquad\qquad\qquad\text{for some $r_1,r_2,r_3\in R,\,w_1,w_2,w_3\in W$}\Big\} & \text{ if }d_i \ge 4\,.
\end{cases}
\end{equation}

Recall that $\rho_0$ denotes the composition of all disjoint transpositions in $\sigma$. Using our usual ordering we can write 
\[ \rho_0 = \prod_{i=1}^{d_0/2} (a_{2i-1}\,a_{2i})\]
for some $d_0\in2\mathbb{N}$ and elements $a_j \in e$, where $a_{2i-1}<a_{2i}$ for all $1\le i\le d_0/2$.
This allows us to define
\begin{align}
A_0(R,W,\rho_0)&=\Big\{j\in[d_0']:\,\{a_{4j-3},a_{4j-2},a_{4j-1},a_{4j}\}\in\big\{\{r_1,w_1,w_2,w_3\},\{w_1,r_1,r_2,r_3\}\big\} \notag\\&\qquad\qquad\qquad\qquad\qquad\mbox{ for some }r_1,r_2,r_3\in R,\,w_1,w_2,w_3\in W\Big\}\,.\label{E:A_i_defn3}
\end{align}

Finally, we define $A(R,W,\sigma)$ to be the vector 
\begin{equation}\label{e:final_A}
A(R,W,\sigma)=\left(A_0(R,W,\rho_0), A_1(R,W,\rho_1),
A_2(R,W,\rho_2),\dots,A_K(R,W,\rho_K)\right) \,.
\end{equation}

\begin{figure}[!ht]
	
	\centering
	\includegraphics[scale=1.2]{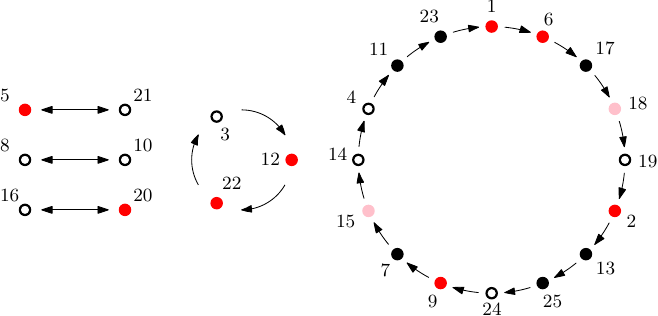}
	\caption{An example of a permutation applied to an edge of size 25. Particles at vertices in $R=\{1,2,5,6,9,12,20,22\}$ are coloured red \protect\includegraphics[scale=1.2]{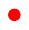} and at vertices in $W=\{3,4,8,10,14,16,19,21,24\}$ are coloured white~\protect\includegraphics[scale=1.2]{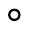}. (The other colours -- black~\protect\includegraphics[scale=1.2]{black_ball.pdf} and pink~\protect\includegraphics[scale=1.2]{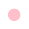} -- will become important later, but may be ignored for now.) Arrows are used to denote the input permutation $\sigma$ so here we see that $\sigma=(5\,\,21)(8\,\,10)(16\,\,20)(3\,\,12\,\,22)(1\,\,6\,\,17\,\,18\,\,19\,\,2\,\,13\,\,25\,\,24\,\,9\,\,7\,\,15\,\,14\,\,4\,\,11\,\,23)$. Since the first pair of transpositions $(5\,21)(8\,10)$ involve one red and three white particles, we deduce that $A_0=\{1\}$; similarly, $A_1=\{0\}$. Looking at the large cycle $\rho_2$, we see that both of the sets $\rho_2^{H_1}(1)=\{1,6,24,9\}$ and $\rho_2^{H_3}(1)=\{19,2,14,4\}$ contain a 3:1 split of reds:whites or whites:reds, and so $A_2=\{1,3\}$. Putting these together we arrive at $A=(\{1\},\{0\},\{1,3\})$.}\label{fig:constructing_A}
	
\end{figure}

\begin{figure}[!ht]
	
	\centering
	\includegraphics[scale=1.2]{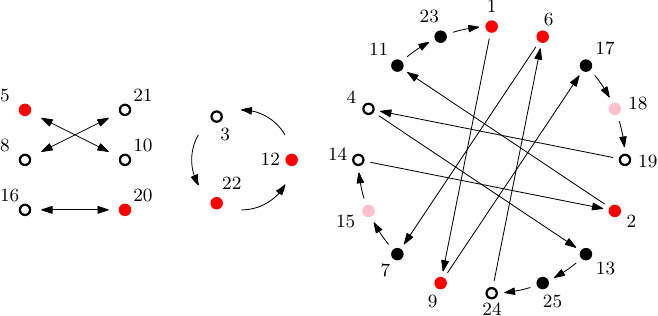}
	\caption{A pictorial description of the result of using the set A derived in Figure~\ref{fig:constructing_A} to construct the twin permutation $\tilde\sigma=\tilde\rhofn_{A(R,W,\sigma)}(\sigma)$; here we see that $\tilde\sigma=(5\,\,10)(8\,\,21)(16\,\,20)(3\,\,22\,\,12)(1\,\,9\,\,17\,\,18\,\,19\,\,4\,\,13\,\,25\,\,24\,\,6\,\,7\,\,15\,\,14\,\,2\,\,11\,\,23)$.}\label{fig:constructing_A_tilde_sigma}
	
\end{figure}

We conclude this section by showing that this construction of $A$
satisfies the conditions of
Corollary~\ref{C:algorithm_random_input}. This in turn guarantees that 
the permutation $\tilde\sigma=\tilde\rhofn_{A(\sigma)}(\sigma)$ has the same law on $\cS_e$ as
$\sigma$.

\begin{lemma}\label{L:Ahasproperty}
	Fix an edge $e\in E$. For any disjoint subsets $R,W$ of $V$ and
	permutation $\sigma\in \cS_e$, the functions $A_i(R,W,\rho_i)$ defined in \eqref{E:A_i_defn2} and \eqref{E:A_i_defn3} satisfy
	\[
	A_i(R,W,\rhofn_{A_i(R,W,\rho_i)}(\rho_i))=A_i(R,W,\rho_i).
	\]
\end{lemma}
\begin{proof}
	The idea here is that, as in part 3 of Corollary~\ref{C:tilde_rho_A_props}, $\rhofn_{A_i(R,W,\rho_i)}(\rho_i)(x) = \rho_i(x)$ unless $x$ belongs to a special set of three or four elements (whose exact definition depends upon the conjugacy class of $\rho_i$). Furthermore, $\rhofn_{A_i(R,W,\rho_i)}$ permutes all elements of such a special set amongst themselves, and so the numbers of red and white vertices within the set are unchanged by the action of $\rhofn_{A_i(R,W,\rho_i)}$. (See Figure~\ref{fig:constructing_A_tilde_sigma} again for a pictorial example.) 
	
	We provide some details here for the case when $\rho_i$ is a cycle of length $d_i\ge 4$: the arguments for 3-cycles and $\rho_0$ are similar.	
	Suppose $j\in A_i(R,W,\rho_i)$. 
	Without loss of generality, suppose that
	\begin{align*}
	\rho_i^{2j-2}(m_i)&\in W,\quad\rho_i^{2d_i'+2j-2}(m_i)\in R,\\
	\rho_i^{2j-1}(m_i)&\in W,\quad\rho_i^{2d_i'+2j-1}(m_i)\in W.
	\end{align*}
	Since $j\in A_i(R,W,\rho_i)$, from equation~\eqref{e:rho_A_defn} we deduce that
	\begin{align*}
	\rhofn_{A_i(R,W,\rho_i)}(\rho_i)^{2d_i'+2j-2}(m_i)&=\rho_i^{2j-2}(m_i)
	\in W,\\
	\rhofn_{A_i(R,W,\rho_i)}(\rho_i)^{2j-2}(m_i)&=\rho_i^{2d_i'+2j-2}(m_i)
	\in R,\\
	\rhofn_{A_i(R,W,\rho_i)}(\rho_i)^{2d_i'+2j-1}(m_i)&=\rho_i^{2j-1}(m_i)
	\in W,\\
	\rhofn_{A_i(R,W,\rho_i)}(\rho_i)^{2j-1}(m_i)&=\rho_i^{2d_i'+2j-1}(m_i)
	\in W.
	\end{align*}
	Therefore $k\in A_i(R,W, \rhofn_{A_i(R,W,\rho_i)}(\rho_i))$.
	The other cases follow similarly. This shows that $A_i(R,W,\rho_i)\subseteq A_i(R,W,
	\rhofn_{A_i(R,W,\rho_i)}(\rho_i))$,  but an identical argument
	shows the reverse implication and we deduce the result.
\end{proof}

\subsection{Construction of the chameleon process}\label{S:chameleon_construction}
In this section we detail the construction of the chameleon process. The connection to the algorithm described in the previous section to generate $\tilde\sigma$ from $\sigma$ will be made clear in Lemma~\ref{L:connection}. In order to deal with edges of size 2, it will be convenient to modify the graphical construction of IP$(k,f,G)$ introduced in Section \ref{S:graphical}, by doubling the rate at which edges ring, and compensating for this by making the process lazy.

More formally, consider the following ingredients:
\begin{enumerate}
	\item a Poisson process $\Lambda$ of rate $2|E|$;
	\item an i.i.d. sequence of $E$-valued random variables $\{e_n\}_{n\in\mathbb{N}}$;
	\item an i.i.d. sequence of permutations $\{\sigma_n\}_{n\in\mathbb{N}}$ with $\sigma_n\in\mathcal{S}_{e_n}$ for each $n\in\mathbb{N}$ and with $\P\bra{\sigma_n=\sigma}=f_{e_n}(\sigma)$;
	\item an i.i.d. sequence of coin flips $\{\theta_n\}_{n\in\mathbb{N}}$ with $\P\bra{\theta_n=1}=\P\bra{\theta_n=0}=1/2$.
\end{enumerate}
We now define $\sigma_n^{\theta_n}$ to equal $\sigma_n$ if $\theta_n=1$ and to be the identity if $\theta_n=0$.
We modify the definition of the maps $I_{[s,t]}$ from Section \ref{S:graphical} as follows: 
\[
I_{[s,t]}=\sigma^{\theta_{\Lambda[0,t]}}_{e_{\Lambda[0,t]}}\circ\sigma^{\theta_{\Lambda[0,t]-1}}_{e_{\Lambda[0,t]-1}}\circ\cdots\circ\sigma^{\theta_{\Lambda[0,s)+1}}_{e_{\Lambda[0,s)+1}}.
\]
The thinning property of Poisson processes implies that the joint distribution of the maps $I_{[s,t]}$, $0\le s\le t<\infty$, is the same as in Section \ref{S:graphical}.
The chameleon process will be built on top of this modified interchange process.

\subsubsection{Formal description of the chameleon process}

Given a $(k-1)$-tuple ${\bf z}\in (V)_{k-1}$, recall that ${\bf O}({\bf
	z}):=\{{\bf z}(1),\ldots,{\bf z}(k-1)\}$ denotes the (unordered) set of
coordinates of ${\bf z}$. The chameleon process is a continuous-time
Markov process with state-space
\[\Omega_k(V):=
\{({\bf z},R,P,W):\,{\bf z}\in (V)_{k-1},\,\text{and sets ${\bf
		O}({\bf z}), R, P, W$ partition $V$}\}.
\]
A particle at vertex $v$ is said to be \emph{red} if $v\in R$,
\emph{white} if $v\in W$, \emph{pink} if $v\in P$ and \emph{black} if
$v\in {\bf O}({\bf z})$. Note that, due to the nature of the
state-space, we can distinguish between the various black particles,
whereas any two red/white/pink particles are indistinguishable from each
other.

We denote the state at time $t$ of the chameleon process started from
$M_0=({\bf z},R,P,W)$ as
\[
M_t=({\bf z}_t,R_t,P_t,W_t).
\]
We say that a particle at vertex $v$ at time $t$ is \emph{red at time
	$t$} if $v\in R_t$ (and similarly for the other colours). We define
now a notion of \emph{ink}, which represents the amount of
\emph{redness} either at a vertex or in the whole system. A vertex $v$
has 1 unit of ink at time $t$ if $v\in R_t$ and half a unit if $v\in
P_t$. Formally then, we define for each $v\in V$ and $t\ge0$,
\begin{equation}\label{e:ink_defn}
\mathrm{ink}_t(v):=\indic{v\in R_t}+\frac1{2}\indic{v\in P_t}.
\end{equation}

We are now able to complete our formal definition of the chameleon
process corresponding to an interchange process on a
hypergraph. 
We set $T=20T_{\mathrm{EX}(4,f,G)}$, and call $T$ the \emph{phase length}. As stated previously, the chameleon process is
time-inhomogeneous, and behaves differently depending on which phase
we are in. There will be just two different kinds of phase: those in
which no colour-changing is permitted and particles are just moved
around the graph according to the underlying interchange process; and
those in which colour-changing (pinkening of red and white particles)
can occur. Furthermore, there will be (deterministic) times at which
\emph{depinking} can occur. To be more precise, the chameleon process is updated at the incident
times $\{\tau_n\}$ of the Poisson process $\Lambda$ and also at deterministic times $2iT$,
$i\in\mathbb{N}$.

To describe which particles are pinkened during a colour-changing
phase, let $\sigma$ ($=\sigma_n$) be the permutation applied to some
edge $e$ ($=e_n)$ at time $t=\tau_n$ and once again recall the cyclic decomposition from \eqref{eq:decomp}:
\[
\sigma=\rho_0\circ\rho_1\circ\cdots\circ\rho_K.
\]
Given that $t$ is in a colour-changing phase, we define subsets of
$V$ in the following way. 

For cycles $\rho_i$ with $d_i=3$, recall that $A_i(R_{t-},W_{t-},\rho_i)$ is either equal to $\{0\}$ or $\varnothing$. For $j\in A_i(R_{t-},W_{t-},\rho_i)$ we define
\[
L_t^{i,j}:=\begin{cases}
\{r,\rho_i(r)\} &\mbox{ if }\rho_i^{\{0,1,2\}}(m_i)=\{r,w_1,w_2\}\mbox{ for some }r\in R_{t-},w_1,w_2\in W_{t-},\\
\{w,\rho_i(w)\}&\mbox{ if }\rho_i^{\{0,1,2\}}(m_i)=\{w,r_1,r_2\}\mbox{ for some }w\in W_{t-},r_1,r_2\in R_{t-}.
\end{cases}
\]
We note that, by construction, if $j=0$ then $L_t^{i,j}$ contains two vertices, with one in $R_{t-}$ and the other in $W_{t-}$. 

For cycles $\rho_i$ with $d_i\ge 4$, we define a set $L_t^{i,j}$ for each $j\in A_i(R_{t-},W_{t-},\rho_i)$ as follows:
\begin{itemize}
	\item if
	$|\rho_i^{H_j}(m_i)\cap R_{t-}|=1$ (so one vertex in the set $\rho_i^{H_j}(m_i)$ contains a red particle,  and the other three contain white particles), then set \[
	L_t^{i,j}:=\begin{cases}
	\{\rho_i^{2j-2}(m_i),\rho_i^{2d_i'+2j-2}(m_i)\}\quad&\mbox{ if } |\{\rho_i^{2j-2}(m_i),\rho_i^{2d_i'+2j-2}(m_i)\}\cap R_{t-}|=1,\\
	\{\rho_i^{2j-1}(m_i),\rho_i^{2d_i'+2j-1}(m_i)\}&\mbox{ otherwise.}
	\end{cases}
	\]
	\item alternatively, if
	$|\rho_i^{H_j}(m_i)\cap W_{t-}|=1$ (so one vertex in the set $\rho_i^{H_j}(m_i)$ contains a white particle, and the other three contain red particles), then set \[
	L_t^{i,j}:=\begin{cases}
	\{\rho_i^{2j-2}(m_i),\rho_i^{2d_i'+2j-2}(m_i)\}\quad&\mbox{ if } |\{\rho_i^{2j-2}(m_i),\rho_i^{2d_i'+2j-2}(m_i)\}\cap W_{t-}|=1,\\
	\{\rho_i^{2j-1}(m_i),\rho_i^{2d_i'+2j-1}(m_i)\}&\mbox{ otherwise.}
	\end{cases}
	\]
\end{itemize}
Once again, this ensures that $L_t^{i,j}$ contains two vertices, one in $R_{t-}$ and the other in $W_{t-}$.

For $\rho_0$ (the product of disjoint transpositions), we proceed similarly. For each $j\in A_0(R_{t-},W_{t-},\rho_0)$ satisfying $a_{4j-1}<a_{4j-2}$, we define $L_t^{0,j}$ as follows: 
\begin{itemize}
	\item if $|\{a_{4j-3},a_{4j-2},a_{4j-1},a_{4j}\}\cap R_{t-}|=1$, then set
	\[
	L_t^{0,j}:=\begin{cases}
	\{a_{4j-3},a_{4j-1}\}&\mbox{ if }|\{a_{4j-3},a_{4j-1}\}\cap R_{t-}|=1,\\
	\{a_{4j-2},a_{4j}\}&\mbox{ otherwise};
	\end{cases}
	\]
	\item alternatively, if $|\{a_{4j-3},a_{4j-2},a_{4j-1},a_{4j}\}\cap W_{t-}|=1$, then set
	\[
	L_t^{0,j}:=\begin{cases}
	\{a_{4j-3},a_{4j-1}\}&\mbox{ if }|\{a_{4j-3},a_{4j-1}\}\cap W_{t-}|=1,\\
	\{a_{4j-2},a_{4j}\}&\mbox{ otherwise}.
	\end{cases}
	\]
\end{itemize}
(If $a_{4j-1}>a_{4j-2}$ then set $L_t^{0,j} = \varnothing$.)

We then let
\[ L_t:=\bigcup_{i=0}^K \bigcup_{j\in A_i(R_{t-},W_{t-},\rho_i)}L_t^{i,j} \,.   \]

Recall the example in Figure~\ref{fig:constructing_A} (and suppose $R=R_{t-}$ and $W=W_{t-}$). Here we obtain $L_t^{0,1}=\{5,8\}$, $L_t^{1,0}=\{3,12\}$, $L_t^{2,1}=\{1,24\}$ and $L_t^{2,3}=\{2,4\}$, and hence $L_t=\{1,2,3,4,5,8,12,24\}$.

The particles at the pairs of vertices selected in this way are those that we
wish to pinken at time $t$. However, it turns out to be useful to limit the number of pinkenings that can occur (during a single colour-changing phase) so that the total number of pinks cannot exceed either the number of reds or the number of whites (this will be crucial to be able to appeal directly to a result of \citep{Oliveira2013a} in the proof of Lemma~\ref{L:inktodepink}).
In order to achieve this, we pick (arbitrarily) 
a subset $L_t^\ast$ of 
\[ \left\{L_t^{i,j}\,:\, i=0,1,\dots,K,\,\, j\in A_i(R_{t-},W_{t-},\rho_i) \right\}\]
with the property that $|L_t^\ast|$ is as large as possible while still satisfying
\[
|P_{t-}|+2|L_t^\ast|\le \min\{|R_t|,|W_t|\}=\min\{|R_{t-}|,|W_{t-}|\} - |L_t^\ast|,
\]
i.e. \[|L_t^\ast|\le \frac1{3}\left(\min\{|R_{t-}|,|W_{t-}|\}-|P_{t-}|\right).\]

Note that $L_t^\ast$ is a set of pairs of vertices, with each pair containing one red and one white particle. Finally, we let $\bar L_t$ be the union of the elements of $L_t^\ast$.

It is precisely the particles at vertices in $\bar L_t$ that we will
pinken at time $t$.

\begin{framed}
	\begin{inbox}{\bf Formal description of chameleon process updates}\label{box:cham_process} \\
		There are three kinds of updates to the chameleon process -- which
		update is performed at time $t$ depends on the value of $t$.
		
		\textbf{Constant-colour phases:} For $t=\tau_n\in(2(i-1
		)T,(2i-1)T]$, update as the interchange process:
		\[
		({\bf z}_t,R_t,P_t,W_t)=(\sigma_n^{\theta_n}({\bf
			z}_{t-}),\sigma_n^{\theta_n}(R_{t-}),\sigma_n^{\theta_n}(P_{t-}),\sigma_n^{\theta_n}(W_{t-}))\,.
		\]
		
		\textbf{Colour-changing phases:} For $t=\tau_n\in((2i-1)T,2iT]$,
		update as interchange (i.e. update as in a constant-colour phase)
		unless $|P_{t-}|<\min\{|R_{t-}|,|W_{t-}|\}$ and we are in one of the following two situations:
		\begin{itemize}
			\item $|e_n|>2$, $\theta_n=1$ and $\sigma_n\neq id$. \\In this case, we
			pinken all particles in the set $\bar L_t$ (half of which belong
			to $R_{t-}$, the others to $W_{t-}$, by design).\\
			Regardless of which particles are pinkened, we then update as the
			interchange process at this time (using $\sigma_n$). Formally
			then, we update as
			\begin{align*}
			&({\bf z}_t,R_t,P_t,W_t)\\&\quad=(\sigma_n({\bf
				z}_{t-}),\sigma_n(R_{t-}\setminus(\bar L_t\cap
			R_{t-})),\sigma_n(P_{t-}\cup \bar
			L_t),\sigma_n(W_{t-}\setminus(\bar L_t\cap W_{t-}))).
			\end{align*}
			\item $e_n=\{w,r\}$ for some $w\in W_{t-}$, $r\in R_{t-}$, and $\sigma_n\neq id$. \\In this case, we pinken both particles on the edge. Formally update as 
			\[
			({\bf z}_t,R_t,P_t,W_t)=({\bf z}_{t-},R_{t-}\setminus\{r\},P_{t-}\cup\{r,w\},W_{t-}\setminus\{w\}).
			\]
		\end{itemize}

		\textbf{Depinking:} For $t=2iT$ with $i\in\mathbb{N}$, if
		$|P_{t-}|\ge\min\{|R_{t-}|,|W_{t-}|\}$ then we generate a Bernoulli($1/2$) random variable
		$Y_i$: if $Y_i=1$, we colour all pink particles red,
		otherwise we colour all pink particles white. Hence the update is
		\[
		({\bf z}_t,R_t,P_t,W_t)=
		\begin{cases}({\bf z}_{t-},R_{t-}\cup P_{t-},\varnothing,W_{t-})&\mbox{if }Y_i=1,\\
		({\bf z}_{t-},R_{t-},\varnothing,W_{t-}\cup P_{t-})&\mbox{if
		}Y_i=0.
		\end{cases}
		\]
	\end{inbox}
\end{framed}

Recall again the example from Figure~\ref{fig:constructing_A}, and suppose this represents the state at time $t-$ of a chameleon process. Then Figure~\ref{fig:chameleonstep} represents the state at time $t$ (assuming that $t$ belongs to a colour-changing phase, and that the associated random variable $\theta$ equals 1).

\begin{figure}[!ht]
	\centering
	\includegraphics[scale=1.2]{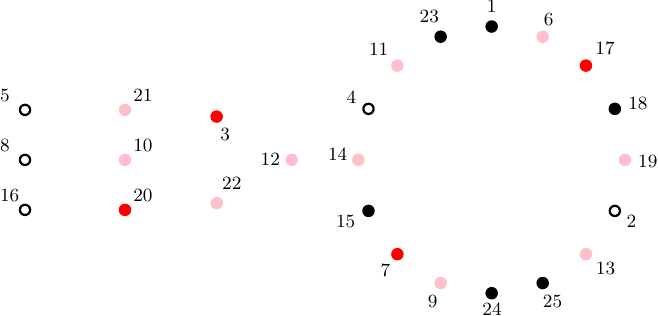}
	\caption{The result of updating the chameleon process from the state pictured in Figure~\ref{fig:constructing_A}. Particles belonging to the set $L_t = \{1,2,3,4,5,8,12,24\}$ have been pinkened, and then all particles have been moved according to the permutation $\sigma=(5\,\,21)(8\,\,10)(16\,\,20)(3\,\,12\,\,22)(1\,\,6\,\,17\,\,18\,\,19\,\,2\,\,13\,\,25\,\,24\,\,9\,\,7\,\,15\,\,14\,\,4\,\,11\,\,23)$. (The number of particles that we are allowed to pinken depends upon the values of $|R_{t-}|$ and $|W_{t-}|$ of course, but here we have assumed for simplicity that $\bar L_t =L_t$.)}\label{fig:chameleonstep}
\end{figure}

The connection between the permutations $\tilde\rhofn_A(\sigma_n)$, which we spent time developing in Section~\ref{s:new_from_old}, and the chameleon process is made explicit in the following lemma, which shall later be employed in the proof of part~2 of Lemma~\ref{L:existence}.
\begin{lemma}\label{L:connection}
	Suppose $t=\tau_n$ is in a colour-changing phase and $|e_n|>2$. Then there exists a permutation $g:V\to V$ such that both of the following statements hold:
	\begin{enumerate}
		\item for any vertex $u$ containing (at time $t-$) a particle which is pinkened at time $t$ in the chameleon process,
		\[
		u\in R_{t-}\quad\mbox{iff}\quad g(u)\in W_{t-}\quad\mbox{and}\quad u\in W_{t-}\quad\mbox{iff}\quad g(u)\in R_{t-} \,;
		\]
		\item for any vertex $u$ containing a particle which is not pinkened at time $t$, the particle at vertex $g(u)$ at time $t-$ has the same colour at time $t$ as the particle at vertex $u$ at time $t-$.
	\end{enumerate}
	
	Moreover, we can take $g$ to satisfy
	\[
	\tilde\rhofn_{A(R_{t-},W_{t-},\sigma_n)}(\sigma_n)(g(u))=\sigma_n(u).
	\]
\end{lemma}
\begin{proof}
	This follows simply by comparing the construction of $\tilde\rhofn_{A(R_{t-},W_{t-},\sigma_n)}(\sigma_n)$ with the construction of the chameleon process.
\end{proof}

\section{Properties of the chameleon process}\label{S:chamhasprops}
In this section we show that the chameleon process satisfies the properties outlined in Lemma~\ref{L:existence}. Part~\ref{item:z_same_path} follows immediately from the construction of the chameleon process, since each black particle moves identically in the chameleon process and the underlying interchange process. 

In order to prove the other three parts, we will need to understand the evolution
of the total amount of ink in the chameleon process. We first of all
note that the number of pink particles accumulates over time until we
have a large number of them; at the next depinking time all pink particles are
recoloured (either red or white) and the process of accumulation
starts again. The process will continue in this manner until either we
have no white particles or we have no red particles (which will occur
immediately after some depinking). At this point, no more pink
particles can be made and so there is no more recolouring of
particles. 
In order to bound the mixing time of the interchange process we need a
good understanding of how quickly the chameleon process reaches the
state where no more recolouring can occur. There are two factors which
affect this: the time we must wait between depinking events and how
the process behaves at depinking times.

Writing ${\bf x}=({\bf z},x)$, for each $j\in \mathbb{N}$ let $D_j({\bf x})$ denote the $j$th
depinking time of a chameleon process started from state $({\bf
	z},\{x\},\varnothing,V\setminus({\bf O}({\bf
	z})\cup\{x\}))\in\Omega_k(V)$. Let
${\mathrm{ink}}_t^{\bf x}$ denote the total amount of ink in the
process at time $t$; note that $0\le\mathrm{ink}_t^{\bf x}\le
|V|-k+1$. Motivated by \cite{Oliveira2013a}, recall that in part~\ref{item:ink} of Lemma~\ref{L:existence} we defined the event
\[
\mathrm{Fill}^{\bf x}:=\left\{\lim_{t\to\infty}\mathrm{ink}_t^{\bf
	x}=|V|-k+1\right \}\,.\] This is the event that all
initially-white particles are eventually coloured red.
We shall make use of the following result concerning
${\mathrm{ink}}_t^{\bf x}$, whose proof may be found in
\citep{Oliveira2013a}. It is applicable in this setting because the event $\mbox{Fill}^{\bf x}$ is independent of ${\bf z}_t^C$ (as it depends only on the outcomes of coin-flips at depinking times and these do not affect ${\bf z}_t^C$) and because $\mbox{ink}_t^{\bf x}$ is a martingale (clear from the construction). We note also that this result is identical to Lemma~\ref{L:existence} part~\ref{item:fill_independent}, and thus serves as its proof.

\begin{prop}\label{P:inkfacts}Fix ${\bf x}=({\bf z},x)\in(V)_k$.
	For each ${\bf c}\in (V)_{k-1}$ and $t\ge0$,
	\[
	\P\bra{\{{\bf z}_t^C={\bf c}\}\cap\mathrm{Fill}^{\bf
			x}}=\frac{\P\bra{{\bf z}_t^C={\bf c}}}{|V|-k+1}.
	\]
\end{prop}

Consider now the expectation on the right-hand side of the statement
of Lemma~\ref{L:TVtoink}: an identical argument to that in Section 6
of \citep{Oliveira2013a} shows that this can be bounded in terms of the
tail probability of the time of the $j$th depinking.
\begin{lemma}\label{L:inktodepink}There exist positive constants $c_1$
	and $c_2$ such that for every $j\in\mathbb{N}$,
	\[
	\sup_{{\bf x}\in (V)_k}\E\left[1-\frac{\mathrm{ink}_t^{\bf
			x}}{|V|-k+1}\Big|\,\mathrm{Fill}^{\bf x}\right]\le
	c_1\sqrt{|V|}e^{-c_2j}+\sup_{{\bf x}\in (V)_k}\P\bra{D_j({\bf x})>
		t\,|\,\mathrm{Fill}^{\bf x}}\,.
	\]
\end{lemma}

We therefore see that we need good control on the probability that
there have only been a few depinkings by time $t$. Here we cannot
simply rely on results from \citep{Oliveira2013a}, since our chameleon
process constructed in Section~\ref{S:cham} clearly obeys very different
dynamics. We shall need the following fundamental result -- a lower
bound on the number of red particles that are lost (due to pinkening)
during a colour-changing phase of the chameleon process (where we start the phase with more white particles than red). The proof is
deferred to Section~\ref{S:num_pinkenings}.

\begin{lemma}\label{L:lossofred}
	Suppose $|V|\ge36$ and consider a chameleon process with initial configuration $({\bf
		z},R,P,W)$ satisfying $|P|<|R|\le |W|$. Then
	\[
	\E[|R_{2T-}|]\le (1-10^{-6})|R|.
	\]
\end{lemma}

This result allows us to bound the probability appearing in the
statement of Lemma~\ref{L:inktodepink}.

\begin{lemma}\label{L:depinkbound}
	There exists a universal constant $\kappa_1>0$ such that for every
	interchange process on a regular hypergraph $G=(V,E)$, every $j\in\mathbb{N}$ and ${\bf
		x}\in (V)_k$, if $|V|\ge 36$ then
	\[
	\P\bra{D_j({\bf x})> t\,|\,\mathrm{Fill}^{\bf x}}\le
	\exp\left\{j-\frac{t}{\kappa_1\,T_{\mathrm{EX}(4,f,G)}(1/4)}\right\}
	.\]
\end{lemma}
\begin{proof}
	Thanks to Lemma~\ref{L:lossofred}, the proofs of Lemmas 6.2 and 9.2
	of \citep{Oliveira2013a} can be emulated to show that there exists a
	positive constant $\kappa$ such that $\E[e^{D_j({\bf x})/\kappa T}\,|\,\mathrm{Fill}^{\bf x}]\le e^j$
	for all $j\in \mathbb N$. 
	Thus by Markov's inequality,
	\begin{align*}
	\P\bra{D_j({\bf x})>t\,|\,\mathrm{Fill}^{\bf x}}&=\P\bra{e^{D_j({\bf
				x})/\kappa T}>e^{t/\kappa T}\,|\,\mathrm{Fill}^{\bf x}}\\&\le
	e^{-t/\kappa T}\E[e^{D_j({\bf x})/\kappa T}\,|\,\mathrm{Fill}^{\bf x}]
	\le e^{j-t/\kappa T}\,.\end{align*} Writing $\kappa_1 = 20\kappa$ completes
	the proof.
\end{proof}
Combining Lemmas \ref{L:inktodepink} and \ref{L:depinkbound} completes the proof of part~\ref{item:exp_ink} of Lemma~\ref{L:existence}.

\medskip
It therefore only remains to show that the chameleon process also satisfies part~\ref{item:ink} of Lemma~\ref{L:existence}.

Let $\{\bar\tau_n\}_{n\in\mathbb N}$ denote the update times of the chameleon process $\{M_t\}_{t\ge0}$; thus each $\bar\tau_n$ is either an incident time of the Poisson process $\Lambda$ from Section~\ref{S:chameleon_construction}, or a depinking time (of the form $2iT$ with $i\in\mathbb N$, as in Box~\ref{box:cham_process}).  For each $j\in\mathbb{N}$, consider a process $\{M_t^j\}_{t\ge0}$
which is identical to $\{M_t\}_{t\ge0}$ for all $t<\bar\tau_j$ but
evolves as the interchange process (i.e. with no further
recolourings) for all $t\ge\bar\tau_j$. More formally, for all $t\ge
\bar\tau_j$,
\[
M_t^j=(I_{(\bar\tau_j,t]}({\bf
	z}_{\bar\tau_j}),I_{(\bar\tau_j,t]}(R_{\bar\tau_j}),I_{(\bar\tau_j,t]}(P_{\bar\tau_j}),I_{(\bar\tau_j,t]}(W_{\bar\tau_j})),
\]
where $I$ is the map used in the modified graphical construction of the interchange process $\{{\bf
	x}_t^\mathrm{IP}\}$ (see Section~\ref{S:chameleon_construction}).

Notice that the almost-sure limit of $\{M_t^j\}_{t\ge0}$ as
$j\to\infty$ is the chameleon process $\{M_t\}_{t\ge0}$. As a
result, by the dominated convergence theorem, it suffices to prove
that for each $j\in\mathbb{N}$ and $b\in V$,\[
\P\bra{x_t^\mathrm{IP}=b\,|\,{\bf
		z}_t^\mathrm{IP}}=\E[\ink_t^j(b)\,|\,{\bf z}_t^\mathrm{IP}],
\]
where $\ink_t^j(b)$ is the amount of ink at vertex
$b$ in the process $M^j_t$.  We prove this by induction on $j$. The case $j=1$ is trivial
since the particle initially at $x$ is the only red particle (and
there are no pink particles).  For the inductive step we wish to
show that almost surely
\[
\E[\ink_t^j(b)\,|\,{\bf z}_t^\mathrm{IP}]=\E[\ink_t^{j+1}(b)\,|\,{\bf
	z}_t^\mathrm{IP}].
\]
For $t<\bar\tau_j$, these are equal since the two processes evolve
identically for such times. The update at time $\bar\tau_j$ of
process $\{M_t^{j+1}\}$ is a chameleon step and could be of two
types: also an update of the interchange process (i.e. $\bar\tau_j$
is an incident time of the Poisson process $\Lambda$), or not (i.e. it is a depinking time). Suppose we are in the first case. We condition
on the common state of $M^j$ and $M^{j+1}$ at time
$\bar\tau_{j-1}$. We want to show that almost surely 
\begin{align}\label{eq:expink}
\E\left[\E[\ink_{\bar\tau_j}^j(b)\,|\,{\bf
	z}_{\bar\tau_j}^\mathrm{IP},M^j_{\bar\tau_{j-1}}]\right]=\E\left[\E[\ink_{\bar\tau_j}^{j+1}(b)\,|\,{\bf
	z}_{\bar\tau_j}^\mathrm{IP},M^{j+1}_{\bar\tau_{j-1}}]\right].
\end{align}
By the strong Markov property at time $\bar\tau_{j-1}$ we
can construct a chameleon process $\{\tilde M_t^j\}$ (with
associated interchange process $\tilde{\bf x}^\mathrm{IP}$) which is
identical to $\{M_t\}$ for all $t<\bar\tau_j$, but for all $t\ge\bar\tau_j$ evolves as an
interchange process (i.e. with no further recolourings) and uses
permutation choices:
\begin{itemize}
	\item $\sigma_n$ if $t=\tau_n$ is in a constant-colour phase,
	\item $\tilde\rhofn_{A(\sigma_n)}(\sigma_n)$ if $t=\tau_n$ is in a colour-changing phase.
\end{itemize} 
We claim that
\[
\frac1{2}\E[\ink_{\bar\tau_j}^j(b)\,|\,{\bf
	z}_{\bar\tau_j}^\mathrm{IP},M^j_{\bar\tau_{j-1}}] +\frac1{2}\E[\widetilde{\ink}_{\bar\tau_j}^j(b)\,|\,\tilde{\bf
	z}_{\bar\tau_j}^\mathrm{IP},\tilde M^j_{\bar\tau_{j-1}}]=\E[\ink_{\bar\tau_j}^{j+1}(b)|\,\,{\bf
	z}_{\bar\tau_j}^\mathrm{IP},M^{j+1}_{\bar\tau_{j-1}}]\,,
\] for all $b\in V$, almost surely (where $\widetilde{\ink}$ is the
ink process under $\tilde M^j$). If $\bar\tau_j$ is in a constant-colour phase, then the statement is immediate (since all three processes update in exactly the same way). If $\bar\tau_j$ is in a colour-changing phase and the particle which is at $b$ at time $\bar\tau_j$ has just been pinkened in the chameleon process then $\ink_{\bar\tau_j}^{j+1}(b)=1/2$ and by Lemma~\ref{L:connection}, $\{\ink_{\bar\tau_j}^{j}(b),\widetilde{\ink}_{\bar\tau_j}^j(b)\}=\{0,1\}$, and so the statement is true. Finally, if $\bar\tau_j$ is in a colour-changing phase but the particle at $b$ at time $\bar\tau_j$ has not just been pinkened, then the three expectations are all equal since $\sigma_j$ and $\tilde\rhofn_{A(\sigma_j)}(\sigma_j)$ have the same distribution, by Corollary~\ref{C:algorithm_random_input} and Lemma~\ref{L:Ahasproperty} (and black particles move identically under each by Corollary~\ref{C:tilde_rho_A_props}). We thus have
\begin{align*}
\E\left[\E[\ink_{\bar\tau_j}^j(b)\,|\,{\bf
	z}_{\bar\tau_j}^\mathrm{IP},M^j_{\bar\tau_{j-1}}]\right]&=\frac1{2}
\E\left[\E[\ink_{\bar\tau_j}^j(b)\,|\,{\bf
	z}_{\bar\tau_j}^\mathrm{IP},M^j_{\bar\tau_{j-1}}]\right]+\frac1{2}
\E\left[\E[\widetilde{\ink}_{\bar\tau_j}^j(b)\,|\,\tilde{\bf z}_{\bar\tau_j}^\mathrm{IP},\tilde M^j_{\bar\tau_{j-1}}]\right]\\
&=\E\left[\E[\ink_{\bar\tau_j}^{j+1}(b)\,|\,{\bf
	z}_{\bar\tau_j}^\mathrm{IP},M^{j+1}_{\bar\tau_{j-1}}]\right].
\end{align*}
We are left to deal with the second case, when $\bar\tau_j$ is not
an update of the interchange process. In this case there must be a
depinking at time $\bar\tau_j$. We wish to show \eqref{eq:expink}
holds, so again use the strong Markov property at time
$\bar\tau_{j-1}$ to obtain
\begin{align*}  \E\left[\E[\ink_{\bar\tau_j}^{j+1}(b)\,|\,{\bf
	z}_{\bar\tau_j}^\mathrm{IP},M^{j+1}_{\bar\tau_{j-1}}]\right]&=\E\left[\E[\ink_{\bar\tau_j}^{j+1}(b)\,|\,{\bf
	z}_{\bar\tau_{j-1}}^\mathrm{IP},M^{j+1}_{\bar\tau_{j-1}}]\right]\\&=\E\left[\E[\ink_{\bar\tau_j}^{j}(b)\,|\,{\bf
	z}_{\bar\tau_{j-1}}^\mathrm{IP},M^{j+1}_{\bar\tau_{j-1}}]\right],
\end{align*}where the second equality follows from
the fact that an independent Bernoulli$(1/2)$ random variable is used to determine the outcome of a depinking. This completes the induction, and with it the proof of part~\ref{item:ink} of Lemma~\ref{L:existence}.

\subsection{Bounding the rate of pinkening}\label{S:num_pinkenings}
In order to prove Lemma~\ref{L:lossofred}, we need to show that during a colour-changing phase
(started with more white particles than red particles) the number of
pink particles we create is in expectation at least a constant times the
number of red particles at the start of that phase. We prove this
result in this section.

Suppose we wish to lower-bound the number of white particles that are pinkened (which we shall refer to as the \emph{number of pinkenings}) during the
first colour-changing phase $[T,2T]$. Since we start with 1 red
particle, there will be more white particles at time $T$ than red. We
will wish to apply the following analysis for a general
colour-changing phase (and not just the first) but the calculations
will carry through since we are assuming the number of white particles is at least the
number of red.

We make a change to the chameleon process in this section in
order to ease our analysis -- we remove the condition that we only
pinken if we have fewer pink particles than either red or white
particles and replace the set $\bar L_t$ in the formal description (Box~\ref{box:cham_process}) with the potentially larger set $L_t$. Although this means we can end up with more pinkening
events, this will only happen if a certain number of pinkening events
have already happened (since pink particles are only created at times of pinkening events), and in that case we will be happy regardless.
We shall
refer to this new process as the \emph{modified chameleon process}.

Let $a\in V$. We define $t_a$ to be the smallest integer $n$ such that
$T<\tau_n\le 2T$ and $a\in e_n$.  If no such $n$ exists we set
$t_a=\infty$. Also, we set $\phi_a=\tau_{t_a}$ with notation
$\tau_{\infty}=\infty$; hence $\phi_a$ is the first time (after time
$T$) that vertex $a$ is in a ringing edge of the underlying
Poisson process. We define a third variable, $F_a$, set to be
equal to $\ast$ in the case $\phi_a=\infty$. If, on the other hand,
$\phi_a<\infty$, there are five possible cases. Let $c_{t_a}$ denote
the cycle of $\sigma_{t_a}$ containing vertex $a$ and $|c_{t_a}|$ denote the number of elements in $c_{t_a}$. For ease of notation we write $d'$ for $\lfloor\frac{|c_{t_a}|}{4}\rfloor$ and $m$ for the smallest element in $c_{t_a}$.
\begin{enumerate}
	\item\label{F_a:2edge} If $|c_{t_a}|=2$ and $|e_{t_a}|=2$, then set $F_a=I^{-1}_{[T,\phi_a)}(c_{t_a}(a))$.
	\item\label{F_a:2cycle_larger_edge} If $|c_{t_a}|=2$ and $|e_{t_a}|\ge3$, then denote by
	\[
	(a_1\,\,a_2)\prec\cdots\prec(a_{l-1}\,\,a_l)
	\]
	the ordered transpositions in the canonical cyclic decomposition of $\rho_{t_a}$. If there exists $j\in\{1,\ldots,\lfloor l/4\rfloor\}$ with $a\in\{a_{4j-3},a_{4j-2},a_{4j-1},a_{4j}\}$ and $a_{4j-1}<a_{4j-2}$, then:
	\begin{enumerate}[(a).]
		\item if $a=a_{4j-3}$ set $F_a=I^{-1}_{[T,\phi_a)}(a_{4j-1},a_{4j-2},a_{4j})$;
		\item if $a=a_{4j-2}$ set $F_a=I^{-1}_{[T,\phi_a)}(a_{4j},a_{4j-3},a_{4j-1})$;
		\item if $a=a_{4j-1}$ set $F_a=I^{-1}_{[T,\phi_a)}(a_{4j-3},a_{4j-2},a_{4j})$;
		\item if $a=a_{4j}$ set $F_a=I^{-1}_{[T,\phi_a)}(a_{4j-2},a_{4j-3},a_{4j-1})$.
	\end{enumerate}
	\item\label{F_a:3cycle} If $|c_{t_a}|=3$, then set $F_a=I^{-1}_{[T,\phi_a)}(c_{t_a}(a),c_{t_a}^2(a))$.
	\item\label{F_a:large_cycle}  If $|c_{t_a}|\ge 4$, and there exists $j\in
	\{1,\ldots,d'\}$ with $a\in c_{t_a}^{H_j}(m)$, then:
	\begin{enumerate}[(i).]
		\item if $a=c_{t_a}^{2j-2}(m)$, 
		set \[F_a=I^{-1}_{[T,\phi_a)}\big(c_{t_a}^{2d'+2j-2}(m),c_{t_a}^{2j-1}(m),c_{t_a}^{2d'+2j-1}(m)\big);\]
		\item if $a=c_{t_a}^{2d'+2j-2}(m)$, set
		\[
		F_a=I^{-1}_{[T,\phi_a)}\big(c_{t_a}^{2j-2}(m),c_{t_a}^{2j-1}(m),c_{t_a}^{2d'+2j-1}(m)\big)
		;\]
		\item if $a=c_{t_a}^{2j-1}(m)$, set
		\[
		F_a=I^{-1}_{[T,\phi_a)}\big(c_{t_a}^{2d'+2j-1}(m),c_{t_a}^{2j-2}(m),c_{t_a}^{2d'+2j-2}(m)\big)
		;\]
		\item if $a=c_{t_a}^{2d'+2j-1}(m)$, set
		\[
		F_a=I^{-1}_{[T,\phi_a)}\big(c_{t_a}^{2j-1}(m),c_{t_a}^{2j-2}(m),c_{t_a}^{2d'+2j-2}(m)\big)
		.\]
	\end{enumerate}
	\item\label{F_a:ast} In all other cases, set $F_a = \ast$.
\end{enumerate}

\begin{rmk}\label{R:Fequalities}
	From this construction it is easy to see that (for any $b,c,d\in V$)
	\begin{enumerate}
		\item In case~\ref{F_a:2edge} above, we have \[\{F_a=b,\,\phi_a=\phi_b\}=\{F_b=a,\,\phi_a=\phi_b\}.\]
		\item In case~\ref{F_a:3cycle} above, we have
		\[
		\{F_a=(b,c),\,\phi_a=\phi_b=\phi_c\}=\{F_b=(c,a),\,\phi_a=\phi_b=\phi_c\}=\{F_c=(a,b),\,\phi_a=\phi_b=\phi_c\}.
		\]
		\item In all other cases, we have 
		\begin{align*}
		\{F_a=(b,c,d),\,\phi_a=\phi_b=\phi_c=\phi_d\}&=\{F_b=(a,c,d),\,\phi_a=\phi_b=\phi_c=\phi_d\}\\=\{F_c=(d,a,b),\,\phi_a=\phi_b=\phi_c=\phi_d\}&=\{F_d=(c,a,b),\,\phi_a=\phi_b=\phi_c=\phi_d\}.
		\end{align*}
	\end{enumerate}
	
\end{rmk}

A possible evolution of the chameleon process during the first two phases is shown in Figure~\ref{f:graphical}.

\begin{figure}[!ht]
\centering
\includegraphics[scale=1.5]{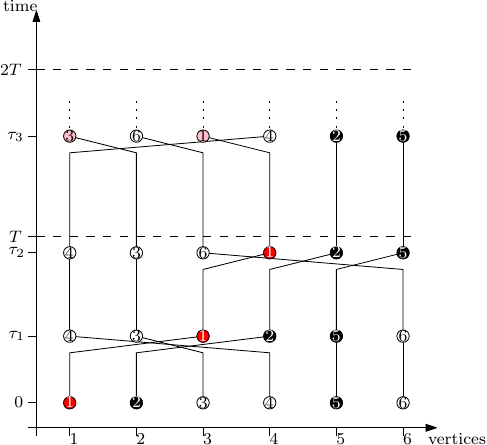}\caption{In this example we suppose that $V=\{1,2,3,4,5,6\}$, $E=\{\{1,2,3,4\},\{1,2,5,6\},\{3,4,5,6\}\}$, and permutations chosen are uniform 4-cycles. Particles are labelled according to their initial location. We see that $I_T(R)=I_T(\{1\})=\{4\}$, $t_4=3$ (the first time vertex 4 is in a ringing edge after time $T$ is $\tau_3$), $\tau_3=\phi_1=\phi_2=\phi_3=\phi_4$ (the first time vertices 1, 2, 3 and 4 are in a ringing edge after time $T$ is $\tau_3$), and $F_4=(2,1,3)$. Particles at vertices 2 and 4 (at time $T$) are coloured pink at time $\tau_3$. }\label{f:graphical}
\end{figure}

We now present a method to count the
number of pinkenings during a colour-changing phase of the modified
chameleon process. For ease of notation we shall write $I$ for the map
$I_{[0,T]}$. The proofs of the first three results below are fairly simple extensions of equivalent results in \citet{Oliveira2013a} and can be found in 
\opt{noarxivVersion} 
{
	\citet{Connor-Pymar-2016}.
}
\opt{arxivVersion} 
{ 
	\hspace{-2.5mm}Appendix~\ref{Appendix2}.
}

\begin{lemma}\label{L:numberpinkenings}
	Consider a modified chameleon process with starting configuration
	$({\bf z},R,P,W)$ satisfying $|P|<|R|\le|W|$. Then the number of pinkenings during $(T,2T)$ is
	at least the number of $b\in I(W)$ such that one of the following holds:
	\begin{itemize}
		\item $F_a=b$ for some $a\in I(R)$ with $\phi_a=\phi_b$,
		\item $F_a=(b,c)$ for some $a\in I(R)$ and $c\in I(W)$ with $\phi_a=\phi_b=\phi_c$, and $\theta_{t_a}=1$,
		\item  $F_a=(b,c,d)$ for some $a\in I(R)$ and $c,d\in I(W)$ with
		$\phi_a=\phi_b=\phi_c=\phi_d$,  and $\theta_{t_a}=1$.
	\end{itemize}
	
\end{lemma}

In bounding the expected number of pinkenings during a colour-changing
phase, it turns out to be useful to have a lower bound on the
probability that $F_a\neq\ast$ given $\phi_a\neq\infty$. This is
because even if vertex $a$ is in a ringing edge during time interval
$[T,2T]$, in order for the particle initially at $a$ to be pinkened in
the modified chameleon process at this time, it is necessary (but not
sufficient) that $F_a\neq\ast$. The proof of this lemma makes use of part \ref{assumpf-fixed} of Assumption~\ref{assumpf}.

\begin{lemma}\label{L:probF_a}For every $a\in V$, $\P\bra{F_a\neq\ast\,|\,\phi_a\neq\infty}\ge4/15$.
\end{lemma}

\begin{prop}\label{P:red_white_mix}
	Consider a modified chameleon process with initial configuration
	$({\bf z},R,P,W)$. Then for any vertices $a,b,c,d$, 
	\begin{enumerate}[(i).]
		\item $\P\bra{|\{a,b\}\cap I(R)|=1,|\{a,b\}\cap I(W)|=1} \ge
		(1-2^{-9})^2|R|\frac{|W|}{\binom{|V|}{2}}$,
		\item $\P\bra{|\{a,b,c\}\cap I(R)|=1,|\{a,b,c\}\cap I(W)|=2} \ge
		(1-2^{-9})^2|R|\frac{\binom{|W|}{2}}{\binom{|V|}{3}}$,
		\item $\P\bra{|\{a,b,c,d\}\cap I(R)|=1,|\{a,b,c,d\}\cap I(W)|=3} \ge
		(1-2^{-9})^2|R|\frac{\binom{|W|}{3}}{\binom{|V|}{4}}$.
	\end{enumerate}   
\end{prop}

We now present the main result of this section -- a version of
Lemma~\ref{L:lossofred} but proved for the \emph{modified} chameleon
process. As explained earlier in this section, this implies the
corresponding result for our original chameleon process.

\begin{lemma}\label{L:lossofred2}
	Suppose $|V|\ge 36$ and consider a modified chameleon process with initial configuration
	$({\bf z},R,P,W)$ satisfying $|P|<|R|\le |W|$. Then
	\[
	\E[|R_{2T-}|]\le (1-10^{-6})|R|.
	\]
\end{lemma}

\begin{proof}
	Write $N(b)$ for the set of vertices that share at least one edge of the hypergraph 
	with $b$. By Lemma \ref{L:numberpinkenings}, we have
	\begin{align*}
	|R_{2T-}|&\le|R|-\sum_{b\in I(W)}\indic{\bigcup_{\substack{a\in N(b)}}\{F_a=b,\phi_a=\phi_b,\,a\in I(R)\}}\\
	&\phantom{\le |R|}-\sum_{b\in I(W)}\indic{\bigcup_{\substack{a,c\in N(b)}}\{F_a=(b,c),\phi_a=\phi_b=\phi_c,\,a\in I(R),c\in I(W),\,\theta_{t_a}=1\}} \\  
	&\phantom{\le |R|}-\sum_{b\in
		I(W)}\indic{\bigcup_{\substack{a,c,d\in N(b)}}\{F_a=(b,c,d),\phi_a=\phi_b=\phi_c=\phi_d,\,a\in I(R),c,d\in
		I(W),\,\theta_{t_a}=1\}}\\
	&=|R|- \sum_{b\in
		I(W)}\sum_{\substack{a\in N(b)}}\Bigg\{\indic{F_a=b,\phi_a=\phi_b,\,a\in I(R)}\\&\phantom{|R|- \sum_{b\in
			I(W)}\sum_{\substack{a\in N(b)}}\Bigg\{\,\,}+\sum_{c\in N(b)}\bigg\{\indic{F_a=(b,c),\phi_a=\phi_b=\phi_c,a\in I(R),c\in I(W),\theta_{t_a}=1}\\&\phantom{|R|- \sum_{b\in
			I(W)}\sum_{\substack{a\in N(b)}}\Bigg\{+\sum_{c\in N(b)}\bigg\{}+\sum_{d\in N(b)}\indic{F_a=(b,c,d),\phi_a=\phi_b=\phi_c=\phi_d,a\in I(R),c,d,\in I(W),\theta_{t_a}=1}\bigg\}\Bigg\}.
	\end{align*}
	
	Note now that
	the event $\{F_a=(b,c,d),\phi_a=\phi_b=\phi_c=\phi_d\}$ is determined entirely
	by the process after time $T$, and in particular is independent
	of the process between times 0 and $T$, and hence of the map $I
	= I_{[0,T]}$. This is also true for the event $\{F_a=(b,c),\phi_a=\phi_b=\phi_c\}$ and the event $\{F_a=b,\phi_a=\phi_b\}$. Recalling Remark \ref{R:Fequalities} we see that the expectation of the above is equal to\begin{align*}
	&|R|- \sum_{b\in
		V}\sum_{\substack{a\in N(b)}}\Bigg\{\frac1{2}\P\bra{F_a=b,\phi_a=\phi_b}\,\P\bra{|\{a,b\}\cap I(R)|=1,\,|\{a,b\}\cap I(W)|=1}\\&+\P\bra{\theta_{t_a}=1} \,\sum_{c\in N(b)}\bigg\{\frac1{3}\P\bra{F_a=(b,c),\phi_a=\phi_b=\phi_c}\,\P\bra{|\{a,b,c\}\cap I(R)|=1,\,|\{a,b,c\}\cap I(W)|=2}\\&+\sum_{d\in N(b)}\frac1{4}\P\bra{F_a=(b,c,d),\phi_a=\phi_b=\phi_c=\phi_d}\,\P\bra{|\{a,b,c,d\}\cap I(R)|=1,\,|\{a,b,c,d\}\cap I(W)|=3}\bigg\}\Bigg\}.
	\end{align*}Using Proposition~\ref{P:red_white_mix} and $\P\bra{\theta_{t_a}=1}=1/2$ we obtain the bound
	\begin{align}
	\label{eq:countpinks1}
	\E\bra{|R_{2T-}|}-|R|\le& -(1-2^{-9})^2\sum_{b\in V}\sum_{a\in N(b)}\Bigg\{\frac1{2}\frac{|R|\,|W|}{\binom{|V|}{2}}\P\bra{F_a=b,\phi_a=\phi_b}\\ \notag &+\sum_{c\in N(b)}\bigg\{\frac1{6}\frac{|R|\binom{|W|}{2}}{\binom{|V|}{3}}\P\bra{F_a=(b,c),\phi_a=\phi_b=\phi_c}\\ \notag &+\sum_{d\in N(b)}\frac1{8}\frac{|R|\binom{|W|}{3}}{\binom{|V|}{4}}\P\bra{F_a=(b,c,d),\phi_a=\phi_b=\phi_c=\phi_d}\bigg\}\Bigg\}.
	\end{align}
	Consider the final probability in the above equation. We can write it as
	\begin{align*}
	&\P\bra{F_a=(b,c,d),\phi_a=\phi_b=\phi_c=\phi_d}\\&=\sum_{\substack{e\in E:\\a,b,c,d\in e}}\P\bra{F_a=(b,c,d),\phi_a=\phi_b=\phi_c=\phi_d,e_{t_a}=e}\\
	&=\sum_{\substack{e\in E:\\a,b,c,d\in e}}\P\bra{F_a=(b,c,d)\big|\,|F_a|=3, e_{t_a}=e}\P\bra{|F_a|=3,\phi_a=\phi_b=\phi_c=\phi_d,e_{t_a}=e}\\&=
	\sum_{\substack{e\in E:\\a,b,c,d\in e}}\P\bra{F_a=(b,c,d)\big|\,|F_a|=3, e_{t_a}=e}\P\bra{|F_a|=3}\P\bra{\phi_a=\phi_b=\phi_c=\phi_d, e_{t_a}=e},
	\end{align*}
	where we have made use of the fact that the choice of the permutations (which determines $|F_a|$) is independent of the choice of the edges that ring.
	We next make use of the regularity of the hypergraph (and that all edges ring at the same rate) to obtain that $\P\bra{\phi_a=\phi_b=\phi_c=\phi_d, e_{t_a}=e}\ge 1/(4D)$ where $D$ is the degree of each vertex. Summing the above over $a,b,c,d$ gives
	\begin{align}\label{eq:F_a(1)=b}
	&\sum_{b\in V}\sum_{a,c,d\in N(b)}\P\bra{F_a=(b,c,d),\phi_a=\phi_b=\phi_c=\phi_d}
	\ge \frac1{4D}\sum_{e\in E}\sum_{a\in e}\P\bra{|F_a|=3}\,.
	\end{align} Similarly 
	\begin{align}\label{eq:F_a(1)=b2}
	&\sum_{b\in V}\sum_{a,c\in N(b)}\P\bra{F_a=(b,c),\phi_a=\phi_b=\phi_c}\ge\frac1{3D}\sum_{e\in E}\sum_{a\in e}\P\bra{|F_a|=2},
	\end{align}and
	\begin{align}\label{eq:F_a(1)=b3}
	&\sum_{b\in V}\sum_{a\in N(b)}\P\bra{F_a=b,\phi_a=\phi_b}\ge\frac1{2D}\sum_{e\in E}\sum_{a\in e}\P\bra{|F_a|=1}.
	\end{align}
	Combining \eqref{eq:countpinks1}, \eqref{eq:F_a(1)=b}, \eqref{eq:F_a(1)=b2} and \eqref{eq:F_a(1)=b3} gives  
	\begin{align*}
	\E\bra{|R_{2T-}|}-|R|\le& -\frac1{4D}(1-2^{-9})^2\frac1{8}\frac{|R|\binom{|W|}{3}}{\binom{|V|}{4}}\sum_{a\in V}\sum_{e:\,e\ni a}\P\bra{F_a\neq \ast}.
	\end{align*}
	Using Lemma \ref{L:probF_a},
	\begin{align}\label{eq:phiaphibFa}
	\P\bra{F_a\neq\ast}=\P\bra{\phi_a\neq\infty}\P\bra{F_a\neq\ast\,|\,\phi_a\neq\infty}\ge\frac4{15}\P\bra{\phi_a\neq\infty}.
	\end{align}
	Also,
	\[
	\P\bra{\phi_a=\infty}\le
	\P\bra{I_{(T,2T)}(a)=a}=\P\bra{a^\text{RW}_T=a},
	\]
	where $\{a_t^\text{RW}\}$ is a realisation of RW$(1,f,G)$ started from
	$a$. Since
	\[
	T=20T_{\mathrm{IP}(4,f,G)}(1/4)\ge 20 T_{\mathrm{RW}(f,G)}(1/4)\ge
	T_{\mathrm{RW}(f,G)}(2^{-20})
	\] (by Proposition~\ref{P:1to1eps}) we have
	\begin{align*}
	\P\bra{\phi_a=\infty}\le\P\bra{a_T^\text{RW}=a}\le
	\frac1{|V|}+2^{-20}.
	\end{align*}
	From \eqref{eq:phiaphibFa}, we deduce that
	\begin{equation}\label{eq:phiainfty}
	\P\bra{F_a\neq\ast}\ge\frac4{15}(1-2^{-20}-1/|V|)\,.
	\end{equation}
	Finally, the assumptions in Lemma~\ref{L:lossofred} on the sizes of the sets $P$, $R$ and $W$ imply that 
	\[
	3|W|\ge|W|+|R|+|P|=|V|-k+1\ge|V|/2\,,
	\]
	and since $|V|\ge36$ we arrive at our stated result:
	\begin{align*}
	\E\bra{|R_{2T-}|}-|R|\le& -\frac1{4D}(1-2^{-9})^2 \frac{|R|}{864|V|}\sum_{a\in V}\sum_{e:e\ni a} \frac4{15}(1-2^{-20}-1/|V|)\\
	=&-\frac1{4D}(1-2^{-9})^2\, \frac{|R|}{864|V|}\, |V|D \,\frac4{15}(1-2^{-20}-1/|V|)\\
	\le& -10^{-6}|R|\,.
	\end{align*}

\end{proof}

 \newpage
 
 \appendix
 \section{Case by case analysis from Proof of Lemma~\ref{L:ktok-1}}
 \label{Appendix3}
 \begin{figure}[!h]
  \centering
  \includegraphics[scale=.75]{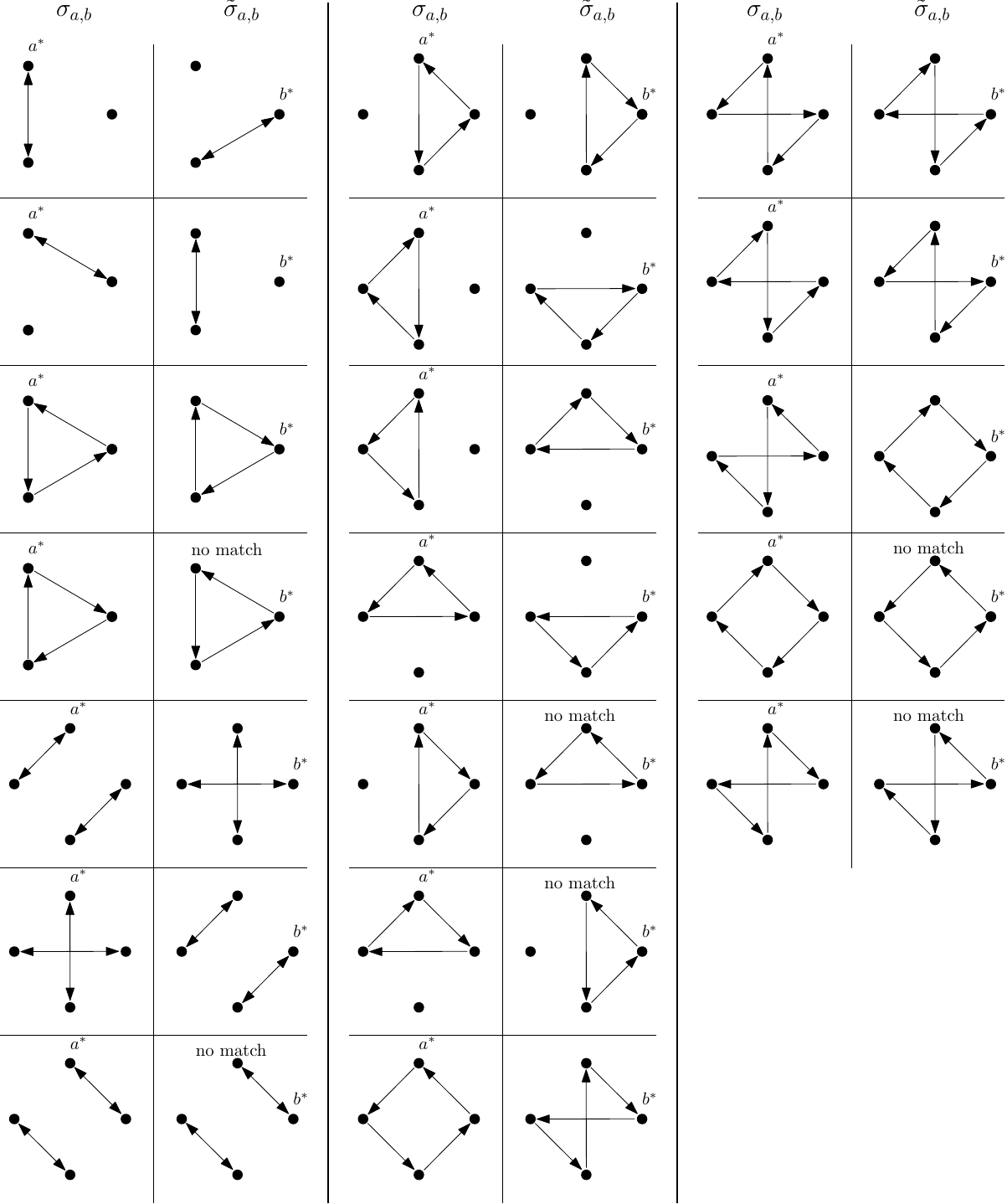}
 \caption{The bijection between $\sigma_{a,b}$ and $\tilde\sigma_{a,b}$. We see that there are certain failure permutations, for which no match occurs, where a match refers to the event that particles $a^\ast$ and $b^\ast$ are moved to the same location after applying permutations $\sigma_{a,b}$ and $\tilde\sigma_{a,b}$, respectively. For a fixed cycle structure we see that the probability of `no match' is at most 1/2, which is achieved when the edge-size is 3 and a cycle of size 3 is chosen.}\label{fig:lemma4.3}
\end{figure}

 \section{Technical proofs for Section~\ref{S:4to1}}\label{Appendix}
 
 Here we include some of the more technical proofs required to compare
 the mixing time of EX$(4,f,G)$ with that of EX$(2,f,G)$.

 \begin{proof}[Proof of Lemma~\ref{L:averagemeet}]
 	Since the hypergraph is non-easy there exists ${\bf x}\in V^2$ such
 	that
 	\[
 	\P\bra{M^{\mathrm{RW}}({\bf x})>10^{10}T}>1/1000.
 	\]
 	
 	We have
 	\begin{align*}
 	& \P\bra{M^{\mathrm{RW}}({\bf x})>10^{10}T}\\&\qquad=\E\big[\E[\indic{M^{\mathrm{RW}}({\bf x})>10^{10}T}|\,{\bf x}^{\mathrm{RW}}_{(10^{10}-40)T}]\big]\\
 	&\qquad=\E\bigg[\E\Big[\indic{M^{\mathrm{RW}}({\bf x})>(10^{10}-40)T}|\,{\bf x}^{\mathrm{RW}}_{(10^{10}-40)T}\Big]\E\Big[\indic{M^{\mathrm{RW}}({\bf x}_{(10^{10}-40)T}^{\mathrm{RW}})>40 T}|\,{\bf x}^{\mathrm{RW}}_{(10^{10}-40)T}\Big]\bigg]\\
 	&\qquad\le\E\left[\E\Big[\indic{M^{\mathrm{RW}}({\bf x})>(10^{10}-40)T}|\,{\bf x}^{\mathrm{RW}}_{(10^{10}-40)T}\Big]\sup_{{\bf y}\in V^2}\P\bra{M^{\mathrm{RW}}({\bf y})>40 T}\right]\\
 	&\qquad=\sup_{{\bf y}\in V^2}\P\bra{M^{\mathrm{RW}}({\bf y})>40 T}\P\left[M^{\mathrm{RW}}({\bf x})>(10^{10}-40)T\right]\\
 	&\qquad\le\left(\sup_{{\bf y}\in V^2}\P\bra{M^{\mathrm{RW}}({\bf y})>40
 		T}\right)^{\frac{10^{10}}{40}}.
 	\end{align*}
 	Hence there exists ${\bf y}\in V^2$ such that
 	\[
 	\P\bra{M^{\mathrm{RW}}({\bf y})>40 T}\ge
 	\left(\frac{1}{1000}\right)^{40(10^{-10})}> 1-10^{-7}.
 	\]

 	Now,
 	\begin{align*}
 	\P\bra{M^{\mathrm{RW}}({\bf y})>40
 		T}&\le\P\big[M^{\mathrm{RW}}({\bf y}_{20 T}^{\mathrm{RW}})>20
 	T\big]\\&=\sum_{{\bf v}\in V^2}\P\bra{{\bf y}_{20
 			T}^\mathrm{RW}={\bf v}}\P\big[M^\mathrm{RW}({\bf v})>20
 	T\big].
 	\end{align*}
 	However, by Definition~\eqref{eq:tvint} of total-variation,
 	Proposition \ref{P:1to1eps}, and the fact that (by the contraction
 	principle) $T_{\mathrm{RW}(2,f,G)}(1/4)\le
 	T_{\mathrm{EX}(2,f,G)}(1/4)$,
 	\begin{align*}
 	\sum_{{\bf v}\in V^2}\P\bra{{\bf y}_{20 T}^\mathrm{RW}={\bf
 			v}}\P\big[M^\mathrm{RW}({\bf v})>20 T\big]-\sum_{{\bf v}\in
 		V^2}\frac{\P\big[M^\mathrm{RW}({\bf v})>20
 		T\big]}{|V|^2}\\\le\|\cL[{\bf y}_{20
 		T}^\mathrm{RW}]-\mathrm{Unif}(V^2)\|\tv\le 2^{-20}.
 	\end{align*}

 	Hence
 	\begin{align*}
 	\sum_{{\bf v}\in V^2}\frac{\P\big[M^\mathrm{RW}({\bf v})>20
 		T\big]}{|V|^2}&\ge
 	\sum_{{\bf v}\in V^2}\P\bra{{\bf y}_{20 T}^\mathrm{RW}={\bf v}}\P\big[M^\mathrm{RW}({\bf v})>20 T\big]-2^{-20}\\
 	&\ge 1-10^{-7}-2^{-20}\\
 	&\ge 1-\frac1{1000}.\qedhere
 	\end{align*}
 \end{proof}
 
 \begin{proof}[Proof of Lemma~\ref{L:MRWxbound}]
 	We begin by conditioning on the value of ${\bf O}({\bf
 		x})_{20T}^\mathrm{EX}$.
 	\begin{align*}
 	&\P\bra{\bar M^\mathrm{RW}({\bf O}({\bf x})_{20T}^\mathrm{EX})\le
 		20T}\\&\qquad=\sum_{\{a_1,a_2,a_3,a_4\}\in \binom{V}{4}} \P\bra{\bar
 		M^\mathrm{RW}(\{a_1,a_2,a_3,a_4\})\le 20T}\,\P\bra{{\bf O}({\bf
 			x})_{20T}^\mathrm{EX}=\{a_1,a_2,a_3,a_4\}}.
 	\end{align*}
 	For each ${\bf a}\in (V)_4$, we have
 	\[
 	\P\bra{\bar M^\mathrm{RW}({\bf O}({\bf a}))\le 20T}\le
 	\sum_{i=2}^4\sum_{j=1}^{i-1}\P\bra{M^\mathrm{RW}(({\bf a}(i),{\bf
 			a}(j)))\le 20T},
 	\]
 	and so
 	\begin{align}\notag
 	&\P\bra{\bar M^\mathrm{RW}({\bf O}({\bf x})_{20T}^\mathrm{EX})\le
 		20T}\\\notag&\le \sum_{\{u_1,u_2\}\in
 		\binom{V}{2}}\P\bra{M^\mathrm{RW}((u_1,u_2))\le 20T}\P\bra{{\bf O}({\bf x})_{20T}^{\mathrm{EX}}\supset\{u_1,u_2\}}\\ &\le\sum_{i=2}^4\sum_{j=1}^{i-1}\sum_{\{u_1,u_2\}\in
 		\binom{V}{2}}\P\bra{M^\mathrm{RW}((u_1,u_2))\le 20T}\P\bra{\{{\bf x}(i),{\bf x}(j)\}_{20T}^{\mathrm{EX}}=\{u_1,u_2\}}.\label{eq:MRW}
 	\end{align}
 	Now, for each $1\le j<i\le 4$ (and motivated by
 	\cite{Oliveira2013a}) set
 	\[
 	\mathrm{Good}_{i,j}:=\left\{ \{a,b\}\in \binom{V}{2}:\,
 	\Big|\P\bra{\{{\bf x}(i),{\bf
 			x}(j)\}_{20T}^\mathrm{EX}=\{a,b\}}-\frac1{\binom{|V|}{2}}\Big|\le\frac{\eps}{\binom{|V|}{2}}\right\}.
 	\]
 	We decompose the sum over ${\bf u}$ above into ${\bf u}\in
 	\mathrm{Good}_{i,j}$ and ${\bf u}\in\mathrm{Bad}_{i,j}$,
 	where $$\mathrm{Bad}_{i,j}=\binom{V}{2}\setminus\mathrm{Good}_{i,j}.$$
 	For the Good terms, we have
 	\begin{align}\notag
 	&\sum_{i=2}^4\sum_{j=1}^{i-1}\sum_{\{u_1,u_2\}\in\mathrm{Good}_{i,j}}\P\bra{M^\mathrm{RW}((u_1,u_2))\le
 		20T}\P\bra{\{{\bf x}(i),{\bf x}(j)\}_{20T}^\mathrm{EX}= \{u_1,u_2\}}\\\notag
 	&\le\sum_{i=2}^4\sum_{j=1}^{i-1}\sum_{\{u_1,u_2\}\in\mathrm{Good}_{i,j}}\P\bra{M^\mathrm{RW}((u_1,u_2))\le
 		20T}\frac{(1+\eps)}{\binom{|V|}{2}}\\\label{eq:good} &\le
 	25(1+\eps)\sum_{{\bf u}\in V^2}\frac{\P\bra{M^\mathrm{RW}({\bf
 				u})\le20T}}{|V|^2}.
 	\end{align}
 	For the Bad terms, we have
 	\begin{align}\notag
 	&\sum_{i=2}^4\sum_{j=1}^{i-1}\sum_{\{u_1,u_2\}\in\mathrm{Bad}_{i,j}}\P\bra{M^\mathrm{RW}((u_1,u_2) )\le
 		20T}\P\bra{\{{\bf x}(i),{\bf x}(j)\}_{20T}^\mathrm{EX}= \{u_1,u_2\}}\\\notag &\le
 	\sum_{i=2}^4\sum_{j=1}^{i-1}\sum_{\{u_1,u_2\}\in\mathrm{Bad}_{i,j}}\P\bra{\{{\bf x}(i),{\bf
 			x}(j)\}_{20T}^\mathrm{EX}= \{u_1,u_2\}}\\\label{eq:bad}
 	&=\sum_{i=2}^4\sum_{j=1}^{i-1}\P\bra{\{{\bf x}(i),{\bf
 			x}(j)\}_{20T}^\mathrm{EX} \in\mathrm{Bad}_{i,j}}.
 	\end{align}
 	Note that for each $1\le j<i\le 4$,
 	\begin{align*}
 	\|\cL({\bf x}_{20T}^\mathrm{EX}) - \textrm{Uniform}\|\tv &=
 		\frac12 \sum_{\{u_1,u_2\}\in \binom{V}{2}}\Big|\P\bra{\{{\bf x}(i),{\bf
 			x}(j)\}_{20T}^\mathrm{EX}= \{{\bf u}(1),{\bf
 			u}(2)\}}-\frac1{\binom{|V|}{2}}\Big|\\
 	&>\frac12\frac{\eps}{\binom{|V|}{2}}|\mathrm{Bad}_{i,j}|\,,
 	\end{align*}
 	since every $\{u_1,u_2\}\in\mathrm{Bad}_{i,j}$ contributes at least
 	$\eps/\binom{|V|}{2}$ to the sum. However, the left-hand side in the
 	above equation is at most $2^{-20}$ by the choice of $T$. We deduce that
 	\[
 	|\mathrm{Bad}_{i,j}|\le \eps^{-1}2^{-19}\binom{|V|}{2},
 	\]
 	and thus
 	\[
 	|\mathrm{Good}_{i,j}|\ge (1-\eps^{-1}2^{-19})\binom{|V|}{2}.
 	\]
 	However, for each $\{u_1,u_2\}\in\mathrm{Good}_{i,j}$ we know that
 	\[
 	\P\bra{\{{\bf x}(i),{\bf x}(j)\}_{20T}^\mathrm{EX}= \{u_1,u_2\}}\ge\frac{1-\eps}{\binom{|V|}{2}}.
 	\]Therefore,
 	\[
 	\P\bra{\{{\bf x}(i),{\bf
 			x}(j)\}_{20T}^\mathrm{EX}\in\mathrm{Good}_{i,j}\}}\ge
 	\frac{1-\eps}{\binom{|V|}{2}}|\mathrm{Good}_{i,j}|\ge
 	1-\eps-\eps^{-1}2^{-19}.
 	\]
 	We deduce that
 	\[
 	\P\bra{\{{\bf x}(i),{\bf
 			x}(j)\}_{20T}^\mathrm{EX}\in\mathrm{Bad}_{i,j}\}}\le
 	\eps+\eps^{-1}2^{-19}.
 	\]
 	Plugging this into \eqref{eq:bad} gives
 	\begin{align*}
 	&\sum_{i=2}^4\sum_{j=1}^{i-1}\sum_{\{u_1,u_2\}\in\mathrm{Bad}_{i,j}}\P\bra{M^\mathrm{RW}((u_1,u_2))\le
 		20T}\P\bra{\{{\bf x}(i),{\bf x}(j)\}_{20T}^\mathrm{EX}= \{u_1,u_2\}}\\&\qquad\le 12(\eps+\eps^{-1}2^{-19}) .\end{align*} Combining
 	this with \eqref{eq:good} and \eqref{eq:MRW} gives
 	\[
 	\P\bra{\bar M^\mathrm{RW}({\bf x}_{20T}^\mathrm{IP})\le 20T}\le
 	12(\eps+\eps^{-1}2^{-19})+25(1+\eps)\sum_{{\bf u}\in
 		V^2}\frac{\P\bra{M^\mathrm{RW}({\bf u})\le20T}}{|V|^2}.\qedhere
 	\]
 	
 \end{proof}
 
 \begin{proof}[Proof of Proposition~\ref{P:couplingIPwithRW}]
 	This proof is similar to the proof of Proposition 4.6 in
 	\citep{Oliveira2013a}. By the graphical construction of Section~\ref{S:graphical}, ${\bf O}({\bf x})_s^{\mathrm{EX}}$ and ${\bf O}({\bf x}_s^\mathrm{IP})$ have the same distribution. Thus by the contraction principle it suffices to show that 
 	\[
 	\|\mathcal{L}[{\bf x}_s^\mathrm{RW}]-\mathcal{L}[{\bf x}_s^\mathrm{IP}]\|\tv\le\P\bra{\bar M^\mathrm{RW}({\bf O}({\bf x}))\le s}.
 	\]
 	We present a coupling of $\{{\bf
 		x}_t^\mathrm{IP}\}_{t\ge0}$ and $\{{\bf
 		x}_t^\mathrm{RW}\}_{t\ge0}$ such that the two processes agree
 	up to time $\bar M^\mathrm{RW}({\bf O}({\bf x}))$. The coupling has
 	state-space $S:=(V)_2\times V^2$ which we split into two parts:
 	$\Delta:=\{({\bf z},{\bf z}):\,{\bf z}\in (V)_2\}$ and
 	$\Delta^\complement$. Denote by $q(\cdot,\cdot)$ the transition
 	rates. We construct the coupling as follows:
 	\begin{enumerate}
 		\item if $({\bf x},{\bf y})\in\Delta^\complement$, the
 		transition rate to any other state in $S$ is the same as that
 		of independent realisations of IP$(4,f,G)$ and RW$(4,f,G)$.
 		\item if $({\bf x},{\bf x})\in\Delta$,
 		\begin{enumerate}
 			\item for $e\in E$ with $|e\cap\{x(1),x(2),x(3),x(4)\}|=1$ and
 			for each $\sigma_e\in \cS_e$,
 			\[
 			q\big(({\bf x},{\bf x}),(\sigma_e({\bf x}),\sigma_e({\bf
 				x}))\big)=f_e(\sigma_e).
 			\]
 			\item for $e\in E$ with $|e\cap\{{\bf x}(1),{\bf x}(2),{\bf
 				x}(3),{\bf x}(4)\}|>1$ and for each $\sigma_e\in\cS_e$,
 			\begin{align*}
 			&q\Big(\big({\bf x},{\bf x}\big),\big(\sigma_e({\bf x}),(\sigma_e({\bf x}(1)),{\bf x}(2),{\bf x}(3),{\bf x}(4))\big)\Big)=f_e(\sigma_e),\\
 			&q\Big(\big({\bf x},{\bf x}\big),\big({\bf x},({\bf x}(1),\sigma_e({\bf x}(2)),{\bf x}(3),{\bf x}(4))\big)\Big)=f_e(\sigma_e),\\
 			&q\Big(\big({\bf x},{\bf x}\big),\big({\bf x},({\bf x}(1),{\bf x}(2),\sigma_e({\bf x}(3)),{\bf x}(4))\big)\Big)=f_e(\sigma_e),\\
 			&q\Big(\big({\bf x},{\bf x}\big),\big({\bf x},({\bf
 				x}(1),{\bf x}(2),{\bf x}(3),\sigma_e({\bf
 				x}(4)))\big)\Big)=f_e(\sigma_e).
 			\end{align*}
 		\end{enumerate}
 		\item all other transitions have rate 0.
 	\end{enumerate}
 	By inspection of the marginals it is clear that this indeed
 	gives a coupling of the two processes. Furthermore, if we start
 	the coupling from a state ${\bf x}\in\Delta$, the two processes
 	can only differ after a transition has occurred according to
 	rule 2.(b); but the first time this happens is precisely $\bar
 	M^\mathrm{RW}({\bf O}({\bf x}))$.
 \end{proof}

 \section{Technical proofs for Section~\ref{S:chamhasprops}}\label{Appendix2}
 
 Here we include proofs of some of the results used in Section~\ref{S:num_pinkenings}.
 
 \begin{proof}[Proof of Lemma~\ref{L:numberpinkenings}]
 	We shall show that in each situation the particle at vertex $b$ at time $T$ is pinkened during $(T,2T)$. 
 	
 	In the first situation with $F_a=b$ for some $a\in I(R)$ with $\phi_a=\phi_b$, we deduce that $|e_{t_a}|=2$ and $I_{[T,\phi_a)}(b)=\sigma_{t_a}(a)$. Since $\phi_a=\phi_b$, we have \[\sigma_{t_a}(a)=I_{[T,\phi_a)}(b)=I_{[T,\phi_b)}(b)=b.\]
 	Since $a\in I(R)$, $b\in I(W)$ and $\phi_a=\phi_b$, we have that $a\in I_{[0,\phi_a)}(R)$ and $b\in I_{[0,\phi_a)}(W)$. This implies that the particle at $b$ at time $T$ (and also the particle at $a$ at time $T$) is pinkened at time $\phi_a$.
 	
 	In the second situation with $F_a=(b,c)$ for some $a\in I(R)$, $c\in I(W)$ with $\phi_a=\phi_b=\phi_c$ and $\theta_{t_a}=1$, we deduce that $|c_{t_a}|=3$ and $I_{[T,\phi_a)}(b,c)=(c_{t_a}(a),c_{t_a}^2(a))$. Since $\phi_a=\phi_b=\phi_c$, we have
 	\[
 	(c_{t_a}(a),c^2_{t_a}(a))=(b,c).
 	\]
 	Since $a\in I(R)$, $b,c\in I(W)$ and $\phi_a=\phi_b=\phi_c$, we have that $a\in I_{[0,\phi_a)}(R)$ and $b, c\in I_{[0,\phi_a)}(W)$ and hence it is immediate that there exists $1\le i\le K$ with $A_i(R_{\phi_a-},W_{\phi_a-},c_{\phi_a})=\{0\}$. Since $\theta_{t_a}=1$ we deduce that the particle at $b$ at time $T$ is pinkened at time $\phi_a$.
 	
 	In the third situation with $F_a=(b,c,d)$ for some $a\in I(R)$, $c,d\in I(W)$ with $\phi_a=\phi_b=\phi_c=\phi_d$ and $\theta_{t_a}=1$ we have two cases. The first case is if $|c_{t_a}|=2$. We denote by 
 	\[
 	(a_1\,\,a_2)\prec\cdots\prec(a_{l-1}\,\,a_l)
 	\]the ordered transpositions in the cyclic decomposition of $\sigma_{t_a}$ and denote by $\rho_0$ the composition of these transpositions. There are four sub-cases which are all similar, and we just prove the result for one of them. So suppose there exists $j\in\{1,\ldots,\lfloor l/4\rfloor\}$ with $a=a_{4j-3}$.
 	Then we have $I_{[T,\phi_a)}(b,c,d)=(a_{4j-1},a_{4j-2},a_{4j}).$ Since $\phi_a=\phi_b=\phi_c=\phi_d$, we have 
 	\[
 	(a_{4j-1},a_{4j-2},a_{4j})=(b,c,d).
 	\]
 	Since $a\in I(R)$, $b,c,d\in I(W)$ and $\phi_a=\phi_b=\phi_c=\phi_d$, we have that $a\in I_{[0,\phi_a)}(R)$ and $b,c,d\in I_{[0,\phi_a)}(W)$ and hence $j\in A_0(R_{\phi_a-},W_{\phi_a-},\rho_0)$ with $L_{\phi_a}^{0,j}=\{a,b\}$. Since $\theta_{t_a}=1$ we deduce that the particle at $b$ at time $T$ is pinkened at time $\phi_a$. The other three sub-cases follow similarly.
 	
 	The second possibility when $F_a=(b,c,d)$ is that $|c_{t_a}|\ge 4$. Again there are four sub-cases which are all similar, and we just prove the result for one of them. So suppose there exists $j\in\{1,\dots,d'\}$ with $a=c_{t_a}^{2j-2}(m)$. Then we have $I_{[T,\phi_a)}(b,c,d)=(c_{t_a}^{2d'+2j-2}(m),c_{t_a}^{2j-1}(m),c_{t_a}^{2d'+2j-1}(m))$.  Since $\phi_a=\phi_b=\phi_c=\phi_d$, we have 
 	\[
 	(c_{t_a}^{2d'+2j-2}(m),c_{t_a}^{2j-1}(m),c_{t_a}^{2d'+2j-1}(m))=(b,c,d).
 	\]
 	Since $a\in I(R)$, $b,c,d\in I(W)$ and $\phi_a=\phi_b=\phi_c=\phi_d$, we have that $a\in I_{[0,\phi_a)}(R)$ and $b,c,d\in I_{[0,\phi_a)}(W)$ and hence $j\in A_i(R_{\phi_a-},W_{\phi_a-},\sigma_{t_a})$ for some $1\le i\le K$ with $L_{\phi_a}^{i,j}=\{a,b\}$. Since $\theta_{t_a}=1$ we deduce that the particle at $b$ at time $T$ is pinkened at time $\phi_a$. 
 	 \end{proof}

 \begin{proof}[Proof of Lemma~\ref{L:probF_a}]
 	For a fixed $a\in V$ we wish to upper bound
 	$\P\bra{F_a=\ast\,|\,\phi_a\neq\infty}$. Recall that we write $c_{t_a}$ for
 	the cycle of $\sigma_{t_a}$ containing vertex $a$, $d'$ for $\lfloor\frac{|c_{t_a}|}{4}\rfloor$, and $m$ for the smallest element in $c_{t_a}$. On the event
 	$\{\phi_a\neq\ast\}$, the event $\{F_a=\ast\}$ is equivalent to the
 	event that at time $\phi_a$ one of the following occurs:
 	\begin{align*}
 	&\text{Event $B_1$: $|c_{t_a}|=1$},\\
 	&\text{Event $B_2$: $|e_{t_a}|\ge 3$, $|c_{t_a}|=2$ and $l=2$ i.e. there is only one transposition and it contains vertex $a$},\\
 	&\text{Event $B_3$: $|e_{t_a}|\ge 3$, $|c_{t_a}|=2$, $l\neq 2$, $a\in\{a_{l-1},a_l\}$ and $l\notin 4\mathbb{N}$},\\
 	&\text{Event $B_4$: $|e_{t_a}|\ge 3$, $|c_{t_a}|=2$, there exists $j\in\{1,\ldots,\lfloor l/4\rfloor\}$ with $a\in\{a_{4j-3},a_{4j-2},a_{4j-1},a_{4j}\}$}\\&\phantom{\text{Event $B_4$;}}\text{ and $a_{4j-1}>a_{4j-2}$},\\
 	&\text{Event $B_5$: $|c_{t_a}|\ge 4$ but there does not exist a
 		$j\in\{1,\ldots,
 		d'\}$ and} \\
 	&\qquad\qquad\qquad\text{$r\in\{2j-2,2d'(c_{t_a})+2j-2,2j-1,2d'(c_{t_a})+2j-1\}$
 		with $a=c_{t_a}^r(m)$.}
 	\end{align*}
 	These events are all disjoint so we have $\P\bra{\bigcup_{i=1}^5B_i}=\sum_{i=1}^5\P\bra{B_i}$. Now, \[\P\bra{B_2}=\P\bra{|c_{t_a}|=2\,\big|\,|e_{t_a}|\ge3,l=2}\P\bra{|e_{t_a}|\ge3,l=2}\le \frac2{3}\P\bra{|e_{t_a}\,|\ge3,l=2},\]
 	by part~\ref{assumpf-conjugacy} of Assumption~\ref{assumpf}. Next, \begin{align*}
 	&\P\bra{B_3}\\&=\P\bra{a\in\{a_{n-1},a_n\}\,\big|\,|e_{t_a}|\ge3,|c_{t_a}|=2,l\neq2,l\notin4\mathbb{N}}\P\bra{|e_{t_a}|\ge3,|c_{t_a}|=2,l\neq2,l\notin4\mathbb{N}}\\&\le\frac1{3}\P\bra{|e_{t_a}|\ge3,|c_{t_a}|=2,l\neq2,l\notin4\mathbb{N}},
 	\end{align*}
 	again by part~\ref{assumpf-conjugacy} of Assumption~\ref{assumpf} and noting that the first probability is maximised when $l= 6$. To deal with event $B_4$, we condition on the values of the two sets $\{a_1,\ldots,a_l\}$ and $\{a_1,\ldots,a_{4j-4}\}$. Now $a_{4j-3}$ is, by construction, the smallest element in $\{a_1,\ldots,a_l\}\setminus\{a_1,\ldots,a_{4j-4}\}$ and so can be identified under the conditioning, and $a_{4j-2}$ is uniform on $\{a_1,\ldots,a_l\}\setminus\{a_1,\ldots,a_{4j-4},a_{4j-3}\}$. Since $a_{4j-1}$ is the smallest element in $\{a_1,\ldots,a_l\}\setminus\{a_1,\ldots,a_{4j-2}\}$, we see that, under this conditioning, the event that $a_{4j-1}>a_{4j-2}$ is the same as the event that $a_{4j-2}$ is chosen to be the smallest element in $\{a_1,\ldots,a_l\}\setminus\{a_1,\ldots,a_{4j-3}\}$. Under the conditioning, this event has probability at most $1/3$ which is achieved when $l=4$ (and so $j=1$). We deduce that 
 	\[
 	\P\bra{B_4}\le\frac1{3}\P\bra{|e_{t_a}|\ge 3, |c_{t_a}|=2, \,\exists j\in\{1,\ldots,\lfloor l/4\rfloor\}:\,a\in\{a_{4j-3},a_{4j-2},a_{4j-1},a_{4j}\}}.
 	\] For the final event we have 
 	\[
 	\P\bra{B_5}\le\frac3{7}\P\bra{|c_{t_a}|\ge4},
 	\]
 	since the worst possible case is when $|c_{t_a}|= 7$ (and so $d'=1$). Combining these gives
 	\begin{align*}
 	& \P\bra{F_a=\ast|\,\phi_a\neq\infty}\\&\le \P\bra{|c_{t_a}|=1}+\frac2{3}\Big(\P\bra{|e_{t_a}|\ge3,l=2}+\P\bra{|e_{t_a}|\ge3,|c_{t_a}|=2,l\neq2,l\notin4\mathbb{N}}\\&\phantom{\le \P\bra{|c_{t_a}|=1}+\frac2{3}\Big(}+\P\bra{|e_{t_a}|\ge 3, |c_{t_a}|=2, \,\exists j\in\{1,\ldots,\lfloor l/4\rfloor\}:\,a\in\{a_{4j-3},a_{4j-2},a_{4j-1},a_{4j}\}}\\&\phantom{\le \P\bra{|c_{t_a}|=1}+\frac2{3}\Big(}+\P\bra{|c_{t_a}|\ge4}\Big)\\
 	&\le \P\bra{|c_{t_a}|=1}+\frac2{3}\big(1-\P\bra{|c_{t_a}|=1}\big)\le\frac{11}{15},
 	\end{align*}by part~\ref{assumpf-fixed} of Assumption~\ref{assumpf}.
 	 \end{proof}
 
 \begin{proof}[Proof of Proposition~\ref{P:red_white_mix}] We prove just the third statement, as the other two are similar. Let $A_t^\mathrm{EX}$ be a realisation
 	of EX$(4,f,G)$ with $A_0^\mathrm{EX}=\{a,b,c,d\}$. Then
 	\begin{align*}
 	&\P\bra{|\{a,b,c,d\}\cap I(R)|=1,|\{a,b,c,d\}\cap I(W)|=3} \\&\qquad =
 	\P\bra{|A_T^{\mathrm{EX}} \cap R|=1,|A_T^\mathrm{EX}\cap W|=3}
 	\\
 	&\qquad\ge (1-2^{-9})^2 \frac{|R|\binom{|W|}{3}}{\binom{|V|}{4}}\,,
 	\end{align*}
 	where the inequality follows from Propositions~\ref{P:1to1eps}
 	and \ref{P:AldousFill}, since $T=20T_{\mathrm{EX}(4,f,G)}$.
 \end{proof}

 \bibliographystyle{chicago} \bibliography{hypergraph}
   
 \vspace{2cm}
 \noindent
 SBC: Department of Mathematics, University of York, York, YO10 5DD, UK. \\
 Email: \url{stephen.connor@york.ac.uk}\\
 WWW: \url{maths.york.ac.uk/www/sbc502}

 \medskip
 \noindent
 RJP: Department of Economics, Mathematics and Statistics, Birkbeck, University of London, London, WC1E 7HX, UK. \\
 Email: \url{r.pymar@bbk.ac.uk}\\
 WWW: \url{bbk.ac.uk/ems/faculty/richard-pymar}

\end{document}